\documentclass{article}

\usepackage[utf8]{inputenc}
\usepackage[T1]{fontenc}
\usepackage[english]{babel}
\usepackage{geometry}
\usepackage[normalem]{ulem} 
\usepackage{amsmath}
\usepackage{amssymb}
\usepackage{amsthm}

\usepackage[colorlinks, pdfpagelabels, pdfstartview = FitH, bookmarksopen=true, bookmarksnumbered = true, linkcolor = black,plainpages = false, hypertexnames = false, citecolor = black]{hyperref} 

\usepackage{tabularx} 
\usepackage{graphicx}
\usepackage{tikz,tikz-3dplot}
\usetikzlibrary{patterns}
\usepackage{pgffor}
\usepackage{pgfplots}

\usepackage{float}

\usepackage{cite}
\usepackage{authblk} 

\usepackage[color, arrow, matrix, curve]{xy}

\newcommand{\dx}[1]{\; \mathrm{d} #1}
\newcommand{\diam}{\mathrm{diam}\:}
\newcommand{\sd}{\: : \:}
\newcommand{\supp}{\mathrm{supp}\:}

\newcommand{\defi}{\: \mathrel{\mathop{\raisebox{1pt}{\scriptsize$:$}}}= \:}
\newcommand{\restr}{\mathbin{\vrule height 1.6ex depth 0pt width 0.13ex\vrule height 0.13ex depth 0pt width 1.3ex}}

\newcommand{\rect}{\mathrm{rect}}

\newcommand{\dist}{\mathrm{dist}}
\newcommand{\tr}{\mathrm{tr}}
\newcommand{\Sym}{\mathrm{Sym}_{0}}
\newcommand{\id}{\mathrm{Id}}

\newcommand{\Lip}{\mathrm{Lip}}
\newcommand{\radM}{\mathbf{r}_0}
\newcommand{\radP}{M\eta} 
\newcommand{\radv}{\epsilon}

\newcommand{\N}{\mathcal{N}}
\newcommand{\E}{\mathcal{E}}
\newcommand{\Eex}{\mathcal{E}_{\eta,\xi}}
\newcommand{\C}{\mathcal{C}}
\newcommand{\R}{\mathcal{R}}
\newcommand{\M}{\mathcal{M}}
\newcommand{\A}{\mathcal{A}}
\newcommand{\T}{\mathcal{T}}

\newcommand{\TT}{\mathbf{T}}
\newcommand{\MM}{\mathbb{M}}
\newcommand{\FF}{\mathbb{F}}

\newcommand{\Qex}{Q_{\eta,\xi}}
\newcommand{\Tex}{T_{\eta,\xi}}
\newcommand{\Sex}{S_{\eta,\xi}}
\newcommand{\Qexn}{Q_{\eta,\xi,n}}
\newcommand{\Texn}{T_{\eta,\xi,n}}
\newcommand{\Sexn}{S_{\eta,\xi,n}}
\newcommand{\lipQ}{\ds{L}}

\newcommand{\Qinfty}{Q_{\infty}}
\newcommand{\Qexinfty}{Q_{\eta,\xi,\infty}}
\newcommand{\sstar}{s_{*}}
\newcommand{\sexstar}{s_{\eta,\xi,*}}

\newcommand{\nnopt}{\mathtt{n}_3}

\newcommand{\polych}{\mathcal{P}}
\newcommand{\flatch}{\mathcal{F}}

\newcommand{\nn}{\mathbf{n}}
\newcommand{\mm}{\mathbf{m}}
\newcommand{\pp}{\mathbf{p}}

\newcommand{\ee}{\mathbf{e}}

\newcommand{\ds}[1]{{\color{black}#1}} 
\newcommand{\dsr}[1]{{\color{black}#1}} 
\newcommand{\dsg}[1]{{\color{black}#1}} 

\newtheorem{theorem}{Theorem}[section]

\newtheorem{proposition}[theorem]{Proposition}
\newtheorem{lemma}[theorem]{Lemma}
\newtheorem{remark}[theorem]{Remark}
\newtheorem{corollary}[theorem]{Corollary}

\newtheoremstyle{name_and_space}{11pt}{20pt}{\itshape}{}{\bfseries}{}{.5em}{#1 #2 #3}
\theoremstyle{name_and_space}

\newif\ifJournal
\Journalfalse

\begin{document}

\title{Convergence to line and surface energies in nematic liquid crystal colloids with external magnetic field}

\date\today
\author[1,2]{François Alouges}
\author[3]{Antonin Chambolle}
\author[1,4]{Dominik Stantejsky}

\affil[1]{Centre de Mathématiques Appliquées, UMR CNRS 7641, École Polytechnique, IP-Paris, 91128 Palaiseau Cedex, France}
\affil[2]{Centre Borelli, ENS Paris-Saclay et CNRS, UMR 9010, \newline 4, avenue des Sciences, 91190 Gif-sur-Yvette, France.}
\affil[3]{CEREMADE, UMR 7534, CNRS \& Université Paris Dauphine PSL, \newline 75775 Paris Cedex 16, France}
\affil[4]{Department of Mathematics and Statistics, McMaster University, \newline Hamilton, ON L8S 4L8 Canada}


\parskip 6pt

\date\today 

\parskip 6pt

\maketitle

\begin{abstract}
We use the Landau-de Gennes energy to describe a particle immersed into nematic liquid crystals with a constant applied magnetic field.
We derive a limit energy in a regime where both line and point defects are present, showing quantitatively that the close-to-minimal energy is asymptotically concentrated on lines and surfaces nearby or on the particle. 
We also discuss regularity of minimizers and optimality conditions for the limit energy.
\linebreak
\textbf{Keywords:}  Nematic liquid crystal colloids, Landau-de Gennes model, $\Gamma-$convergence, flat chains, singular limits \\
\textbf{MSC2020:} 49J45, 
				  49S05, 
				  76A15 
				 
\end{abstract}

\tableofcontents

\section{Introduction}

\ds{This paper is about a physical model of a particle immersed in a liquid crystal, in a regime where the energy is concentrated on lines and surfaces of singularities.}
The history of interaction between variational problems and geometry has been long and of great mutual influence \cite{Goldstine1980}, starting from the geometrically motivated problem of the brachistochrone curve \cite{Bernoulli1696,PaoloFreguglia2016}, Fermat's principle in optics \cite{MaxBorn2019}, material science \cite{Ball1998} to general relativity \cite{Hilbert1924,Masiello2000}.

One particularly important problem arises when the size of geometrical objects themselves is to be minimized leading to so called \emph{minimal surfaces} \cite{Lagrange1761}. A classical example is the two dimensional soap film spanning between predefined (fixed) boundary curves, called Plateau's problem \cite{Douglas1931,Rado1933,Struwe2014}. 
Some solutions can be constructed explicitly \cite{Hoffman1993,Karcher1989} or studied through means of harmonic and complex analysis \cite{Courant1977,Jost1991,Nitsche2011}, but a general theory was not available until the development of geometric measure theory and its language of currents, flat chains, mass and varifolds to describe the objects and how to measure them \cite{Federer1960,Federer1996,Morgan2016,Taylor1976,Almgren2001}.

A different question giving rise to problems involving minimal surfaces is given by the classical $\Gamma-$convergence result of Modica and Mortola \cite{Modica1977} (see also \cite{Modica1987}) of a weighted Dirichlet energy and a penalizing double-well potential to the perimeter functional. 
A constraint such as a prescribed volume ensures the problem to be non trivial. 
The energy typically is concentrated in regions where none of the favourable states of the potential are attained. For the limsup inequality, one constructs a one dimensional profile that minimizes the transition between the favoured states.

Another variational problem in which geometry appears is given by the Ginzburg-Landau model.
In the famous work \cite{Bethuel1994}, the (logarithmically diverging) leading order term and (after rescaling) a limit problem have been derived. 
The limiting variational problem is geometric and consists in finding an optimal distribution of points in the plane subject to constraints coming from the topological degree of the initial problem.
This approach stimulated research which lead to a large literature \cite{Lin1995,Mironescu1995, Alberti2005,Brezis1995, Sandier2004, Ignat2021,Canevari2020}, in particular for problems in micromagnetics \cite{Ignat2016,Kurzke2010}, superconductors \cite{Serfaty2001,Kim2002,SoerenFournais2010} and liquid crystals \cite{Liu2000,Ball2011,Ignat2016a}. 

Our work combines many of the before mentioned ideas to describe the different contributions and effects that take place in our problem. 
For example, we use the generalized three dimensional analogue of estimations in \cite{Bethuel1994} as considered in \cite{Sandier1998,Jerrard1999,Chiron2004,Canevari2015a,Canevari2020} to obtain a length minimization problem for curves. 
Coupled with this optimization problem, we show using a Modica-Mortola type argument that the remaining part of the energy concentrates on hypersurfaces which end either on the boundary of the domain or on the  described line. 

The main motivation for this article was the study of the formation and transition of singularities in colloidal nematic liquid crystals, in particular the \emph{Saturn ring effect}.
It has been observed in experiments that nematic liquid crystals may exhibit line and point singularities.
Those can take the form of a ring around one or several of the colloidal particles depending on the shape and size of the particles and the strength of an external electric or magnetic field  \cite{MondainMonval1999,Loudet2001,Loudet2002,Musevic2017}.
This phenomenon is caused by the incompatibility between the boundary condition at the surface of the particle where a positive topological charge is created, and the condition at infinity where an electric or magnetic field enforces a uniform alignment of the molecules in the direction of the field.
While spheres are the most studied particles, there is also a considerable interest in defect structures around non-spherical inclusions \cite{Sahu2019,Aplinc2019,Sudhakaran2020}.
For the study of phenomena such as self assembly \cite{Smalyukh2010,Xie2018,Solodkov2019} usually a large number of particles is needed. 
In this work however, we focus on the simpler case of a single colloidal particle as a first step for understanding the complex interaction that takes place in colloidal systems\cite{Garbovskiy2010,Musevic2017,Machon2019}.

This article is the continuation of the work started in \cite{ACS2021} where we studied a spherical inclusion, our main theorem (Theorem~\ref{thm:main}) is a generalisation of Theorem 3.1 in \cite{ACS2021} (see Remark~\ref{rem:comments_on_main_thm}). In particular, our new theorem holds for an arbitrary manifold of class $C^{2}$ instead of a sphere and we remove the hypothesis of rotational equivariance \dsg{and convexity}. 

Although the applied ideas could be used to carry out a similar analysis for a larger class of energy functionals, we place ourself in the context of the Landau-de Gennes model for nematic liquid crystals.
A common way to describe liquid crystals is by introducing a unit vector field $\nn$, the so called \emph{director field}, for example in the Oseen-Frank model. 
The vector $\nn$ represents the local orientation of the liquid crystal molecules. 
In practice, it is often not possible to distinguish between $\nn$ and $-\nn$, so that $\nn$ should rather be $\mathbb{R}P^2-$valued, where $\mathbb{R}P^2$ is the two-dimensional real projective space.
More generally, one can think of describing the arrangement of the molecules by a symmetric probability distribution $\rho$ on the sphere of directions.
Because of the symmetry, the first moment of $\rho$ vanishes and the (shifted) second moment $Q$ is a symmetric traceless matrix (also called $Q-$tensor), which is used to represent $\rho$ in the Landau-de Gennes model. 
In the following we will denote $\Sym$ the space of such symmetric traceless matrices.
Under this identification, the uniform distribution on the sphere corresponds to the \emph{isotropic} state in which all three eigenvalues $\lambda_1\geq\lambda_2\geq\lambda_3$ of $Q$ are equal to zero or equivalently $Q=0$.
In case two eigenvalues are equal, we call $Q$ \emph{uniaxial}.
More precisely, if \dsg{$\lambda_1>\lambda_2=\lambda_3$ we say that $Q$ is \emph{prolate} (or positively) uniaxial, while if $\lambda_1=\lambda_2>\lambda_3$} it is called \emph{oblate} (or negatively) uniaxial.
If all three eigenvalues of $Q$ are distinct $\lambda_1>\lambda_2>\lambda_3$, we speak of a \emph{biaxial} $Q-$tensor.
A particularly important role is played by the set $\N$ of prolate uniaxial tensors of a given norm as they are minimizers of the ordering potential in the Landau-de Gennes energy as we will see in Section~\ref{subsec:LdG-model}.
Elements $Q\in\N$ can be written as $Q=s_*(\nn\otimes\nn-\frac13\id)$ ($s_*$ being a constant depending on the liquid crystal) and thus allow an identification with the director field in direction $\pm\nn$.
On the other hand, singularities are described by situations in which one cannot identify a director field, e.g.\ if $Q$ is isotropic or oblate uniaxial.
However, the analysis carried out in this paper does not discriminate between the two different possibilities as they have asymptotically the same energetic cost in our regime.
Nevertheless, in \cite{Canevari2013} it has been proven that in some situations an oblate uniaxial defect core surrounded by a biaxial region has strictly smaller energy compared to an isotopic core.
We refer the interested reader to \cite{Ball2017} for a more detailed introduction to $Q-$tensors, the Landau-de Gennes energy and related models for liquid crystals.

As we will see later in Section~\ref{subsec:LdG-model}, the Landau-de Gennes model in our case comprises three contributions related to the elastic, ordering and magnetic energy.
The relative strength of the individual terms are modulated by the dimensionless parameter $\xi$ describing the ratio between elastic and bulk energy, while $\eta$ couples the elastic with the magnetic term. 
We are concerned with the limit of $\eta,\xi\rightarrow 0$, which can be physically interpreted as a limit of large particles and weak magnetic fields, see \cite{Gartland2018,ACS2021}.
This regime has been studied previously in \cite{Alama2017} for a spherical particle under the assumption that $\eta|\ln(\xi)|\rightarrow 0$ as $\eta,\xi\rightarrow 0$ in which a Saturn ring structure is found.
In \cite{Alama2015} the authors treated the problem in the absence of a magnetic field, i.e.\ for $\eta = \infty$.
For $\xi\rightarrow 0$ they deduce that a point defect occurs.
Our work places itself in the intermediate regime in which $\eta|\ln(\xi)|\rightarrow \beta\in (0,\infty)$ as $\eta,\xi\rightarrow 0$.
As seen in \cite{ACS2021} for a spherical particle, this regime allows for incorporating different minimizing configurations, depending on the parameter $\beta$. 
In a forthcoming paper \cite{ACS2023}, we are going to develop numerical methods to calculate the minimizing configurations around non-spherical particles based on the results in this work.

\section{Preliminaries}

Before we can state our results, we give a short introduction to the Landau-de Gennes model that we use here and the concept of flat chains, stating some results that will be used later in the proofs.

\subsection{Landau-de Gennes model for nematic liquid crystals}
\label{subsec:LdG-model}

Our article has been motivated by the study of liquid crystal colloids with external magnetic field. 
Let $E\subset\mathbb{R}^3$ be a colloidal particle and let $\Omega\defi\mathbb{R}^3\setminus E$ be the region occupied by the liquid crystal.
The Landau-de Gennes energy with additional magnetic field term \cite[Ch. 6, Secs. 3-4 and Ch. 10, Sec. 2.3]{Priestley1974} (see also \cite[Ch. 3, Secs. 1-2]{deGennes1993}) can be stated in dimensionless form as 
\begin{align}\label{def:Eex}
\Eex(Q) 
\ = \ \int_\Omega \frac{1}{2}|\nabla Q|^2 + \frac{1}{\xi^2}f(Q) + \frac{1}{\eta^2}g(Q) + C_0 \dx x\, ,
\end{align}
where the energy density $f$ is given by
\begin{align} \label{def:eq:f} 
f(Q) 
\ &= \ C - \frac{a}{2}\tr(Q^2) - \frac{b}{3}\tr(Q^3) + \frac{c}{4}\tr(Q^2)^2 \, ,
\end{align}
and $g$ is a function taking into account the effects of the external magnetic field in a way we formalize a bit later in this section.
The function $Q:\Omega\rightarrow\Sym$ is a tensorial order parameter taking values in the space of symmetric traceless matrices 
\begin{align*}
\Sym \defi \{ Q\in\mathbb{R}^{3\times 3}\sd Q^T = Q\text{ and }\tr(Q) = 0 \}\, ,
\end{align*}
equipped with the norm $\Vert Q\Vert^2\defi\tr(Q^2)$.
It is used to describe the local distribution of orientation of the liquid crystal molecules.
We consider the case when the parameters $\eta$ and $\xi$ converge to zero in a regime where $\eta|\ln(\xi)|\rightarrow \beta\in (0,\infty)$. 
The constant $C_0=C_0(\eta,\xi)$ (resp. $C$) is chosen such that the energy density (resp. $f$) becomes non-negative and with minimal value $0$.

The following properties of $f$ are going to be used in the sequel:
\begin{enumerate}
\item The function $f$ is non-negative and $\N\defi f^{-1}(0)$ is a smooth, closed, compact, connected manifold,  diffeomorphic to the real projective plane $\mathbb{R}P^2$. Note that $\N$ is given by
\begin{align*}
\N = \left\{ s_*\left(\nn\otimes\nn - \frac13\id\right) \sd \nn\in\mathbb{S}^2 \right\}\, ,
\end{align*}
for $s_* = \frac{1}{4c}(b + \sqrt{b^2 + 24ac})$ (cf. \cite{Majumdar2010a}) and in particular $Q$ is prolate uniaxial.
\item We need $f$ to behave uniformly quadratic close to its minima, i.e.\ we assume that there exist constants $\delta_0,\gamma_1>0$ such that for all $Q\in\Sym$ with $\dist(Q,\N)\leq \delta_0$ it holds
\begin{align*}
f(Q) \ \geq \ \gamma_1 \dist^2(Q,\N)\, .
\end{align*}
\item Lastly, we need to quantify the growth of $f$. More precisely, we assume that there exist constants $C_1,C_2>0$, such that for all $Q\in\Sym$
\begin{align*}
f(Q)\geq C_1\bigg(|Q|^2 - \frac{2}{3}s_*^2\bigg)^2\, , \qquad Df(Q):Q \geq C_1\: |Q|^4 - C_2 \, .
\end{align*}
\end{enumerate}
It can be checked that $f$ given in \eqref{def:eq:f} satisfies these assumptions \cite{Majumdar2010a, ACS2021, Canevari2015a, Canevari2013}.
The exponent $4$ in the last assumption is not crucial but only needs to outweigh the growth of $g$. 

We also recall the following facts about the geometry of $\Sym$:
\begin{enumerate}
\item Elements $Q\in \Sym$ admit the following decomposition: There exist $s\in [0,\infty)$ and $r\in [0,1]$ such that
\begin{equation} \label{eq:q_decom}
Q = s\left( \left(\nn\otimes\nn - \frac{1}{3}\id\right) + r\left(\mm\otimes\mm - \frac{1}{3}\id\right)\right)\, ,
\end{equation}
where $\nn,\mm$ are normalized, orthogonal eigenvectors of $Q$. The values $s$ and $r$ are continuous functions of $Q$.
\item The set where decomposition \eqref{eq:q_decom} is not unique, is given by 
\begin{align}\label{def:cal_C}
\C \defi \{ Q\in\Sym\setminus\{0\}\sd r(Q)=1 \}\cup\{0\} \, ,
\end{align}
i.e.\ $\C$ consists of the isotropic as well as the oblate uniaxial states.
Another characterization of $\C$ is $\C = \{ Q\in\Sym\sd \lambda_1(Q)=\lambda_2(Q) \}$, where the two leading eigenvalues of $Q$ are denoted by $\lambda_1,\lambda_2$. 
Moreover, $\C$ has the structure of a cone over $\mathbb{R}P^2$ and $\C\setminus\{0\} \cong \mathbb{R}P^2\times\mathbb{R}$.
\item There exists a continuous retraction $\R:\Sym\setminus\C\rightarrow\N$ such that $\R(Q)=Q$ for all $Q\in\N$. One can choose $\R$ to be the nearest point projection onto $\N$.
In this case, $\R(Q)=s_*(\nn\otimes\nn-\frac{1}{3}\id)$ for $Q\in \Sym\setminus\C$ decomposed as in \eqref{eq:q_decom}.
\end{enumerate}

The energy density $g$ in \eqref{def:Eex} incorporates an external magnetic field into the model. This motivates the following assumption:
\begin{enumerate}
\item The function $g$ favours $Q$ having an eigenvector equal to the direction of the external field, in our case chosen to be along $\ee_3$. More precisely, assume $g$ is invariant by rotations around the $\ee_3-$axis and the function $O(3)\ni R\mapsto g(R^\top Q R)$ is minimal if $\ee_3$ is eigenvector to the maximal eigenvalue of $R^\top Q R$. 
Decomposing $Q$ as in \eqref{eq:q_decom} with $\nn=\ee_3$ and keeping $s$ and $\mm$ fixed, then $g(Q)$ is minimal for $r=0$.
For a  prolate uniaxial $Q\in\N$, i.e.\ $Q=s_*(\nn\otimes\nn-\frac{1}{3}\id)$ for $s_*\geq 0$ and $\nn\in\mathbb{S}^2$ we have
\begin{align} \label{def:g:uniax}
g(Q) = c_*^2(1-\nn_3^2)\, .
\end{align}
\end{enumerate}
The precise formula for $g$ in \eqref{def:g:uniax} is not \dsg{crucial} to our analysis, but for simplicity we assume this particular form. It would be enough to assume that $g|_\N$ has a strict minimum in $Q=s_*(\ee_3\otimes\ee_3 - \frac12\id)$, see Remark 4.18 in \cite{ACS2021}.
Besides this physical assumption, our analysis requires $g$ to satisfy the following \dsg{hypotheses}:
\begin{enumerate}\addtocounter{enumi}{1}
\item The function $g:\Sym\rightarrow\mathbb{R}$ is of class $C^2$ away from $Q=0$ and in particular satisfies the  Lipschitz condition close to $\N$: There exist constants $\delta_1,C>0$ such that if $Q\in\Sym$ with $\dist(Q,\N)<\delta$ for $0<\delta<\delta_1$, then
\begin{align} \label{def:g:N-lipschitz}
|g(Q) - g(\R(Q))| \leq C\: \dist(Q,\N)\, .
\end{align}
\item The growth of $g$ is slower than $f$, namely
\begin{align} \label{def:g:growQ4}
|g(Q)| \ &\leq \ C\: (1+|Q|^4)\, , \\
|Dg(Q)| \ &\leq \ C\: (1+|Q|^3)\, , \label{def:g:growQ3}
\end{align}
for all $Q\in\Sym$ and a constant $C>0$.
\end{enumerate}

A physically motivated example that satisfies those assumptions \cite[Prop. A.1]{ACS2021} is for example given by
\begin{align} \label{def:eq:g} 
g(Q) 
\ &= \ \frac{2}{3}s_* - Q_{33}\, .
\end{align}

Under these assumptions on $f$ and $g$, it has been shown in \cite[Prop. 2.4 and Prop. 2.6]{ACS2021} that $g$ acts on $f$ as a perturbation in the following sense:

\begin{proposition}\label{prop:f__g_geq}
For $\xi,\eta>0$ with $\xi\ll\eta$, there exists a smooth manifold $\N_{\eta,\xi}\subset\Sym$, diffeomorphic to $\N$ such that
\begin{align} \label{prop:f__g_geq:eq}
f(Q)+\frac{\xi^2}{\eta^2}g(Q)+\xi^2 C_0(\xi,\eta) \geq \gamma_2\: \dist^2(Q,\N_{\eta,\xi})
\end{align}
for a constant $\gamma_2>0$. In addition, there exists a constant $C>0$ such that
\begin{align} \label{prop:f__g_geq:Nex_N}
\sup_{Q\in\N_{\eta,\xi}} \dist(Q,\N) \leq C \frac{\xi^2}{\eta^2}\, .
\end{align}
Furthermore, there exists a unique $\Qexinfty\in\N_{\eta,\xi}$ such that
\begin{align*}
\Qexinfty = \underset{Q\in\Sym}{\mathrm{argmin}} \: \frac{1}{\xi^2}f(Q)+\frac{1}{\eta^2}g(Q)\, ,
\end{align*}
given by $\Qexinfty = \sexstar(\ee_3\otimes\ee_3-\frac13\id)$, where $|s_{*,t}-s_*| \leq C t$.
\end{proposition}

This shows that the constant $C_0$ in \eqref{def:Eex} should be chosen to be $C_0(\xi,\eta) = -\frac{1}{\xi^2}f(\Qexinfty)-\frac{1}{\eta^2}g(\Qexinfty) \geq 0$ and it also holds true that $C_0(\xi,\eta)\leq C \xi^2/\eta^4$.

Since $\sexstar\rightarrow s_{*,0} = s_*$ for $\xi,\eta\rightarrow 0$ in our regime, it is convenient to also introduce $Q_\infty\defi s_*(\ee_3\otimes\ee_3-\frac13\id)$ which minimizes $\xi^{-2}f(Q)+\eta^{-2}g(Q)$ among $Q\in\N$.

So far we have seen that the strong weight $\frac{1}{\xi^2}$ in front of the bulk potential $f$ (compared to $\frac{1}{\eta^2}$ for $g$) favours $Q$ to be close to the manifold $\N$. 
In other words, we expect energy related to $f$ to be concentrated in regions where $Q$ is far from $\N$. 
In a sense that is specified in Theorem~\ref{thm:main}, this region is related to the set where $Q$ takes values in $\C$.
In the same spirit we remark that under prolate uniaxial constraint, $g$ prefers the normalized eigenvector $\nn$ corresponding to the largest eigenvalue to have a large third component $\nn_3$ as formalized in \eqref{def:g:uniax}. 
Therefore we expect that the energy contribution coming from $g$ is concentrated on domains where $|\nn_3|\approx 0$.
More precisely, we introduce
\begin{align}\label{def:cal_T}
\T \defi \{ Q\in\Sym\sd s>0\, , 0\leq r<1\, , n_3 = 0 \}\, ,
\end{align}
where $r,s,\nn$ are defined as in \eqref{eq:q_decom}. 
We study properties of $\T$ later on in Subsection~\ref{subsec:constr_T_Q_uniax} and Section~\ref{sec:app_complex_T}.
Most importantly, we will show in Corollary~\ref{cor:calT_manifold_bdry} that $\partial\T = \C$.
This is a consequence from the fact that if $r(Q)=1$, then $Q$ has a two-dimensional eigenspace for the largest eigenvalue which necessarily intersects the hyperplane $\{n_3=0\}$.

\subsection{Flat chains}
\label{subsec:flat_chains}

One of the main results of this paper is that the previously described energy concentrates on lines and surfaces when $\eta,\xi\rightarrow 0$.
In order to state our main theorem, we therefore need an appropriate framework to describe geometric objects such as lines, surfaces and boundaries which is provided by geometric measure theory and in particular flat chains \cite{White1999,White1999a}.
The very basic idea of geometric measure theory is to represent geometric objects as elements of a vector space and therefore allows for algebraic operations such as addition.
In that respect, flat chains are such elements which in our case are dedicated to represent surfaces and lines.
In the following we give a quick overview of the most important results that we use subsequently.
For a detailed and complete presentation of flat chains and geometric measure theory, we refer to \cite{Fleming1966, Morgan2016, Federer1960, Federer1996, Simon1983}.

\paragraph{Polyhedral flat chains.}
Let $G$ be an abelian group (written additively) with neutral element $0$ and $|\cdot|:G\rightarrow [0,\infty)$ a function satisfying $|g|=0$ if and only if $g=0$, $|-g|=|g|$ and $|g+h|\leq |g|+|h|$ for all $g,h\in G$. 
In this paper, we are only concerned with the easiest case of $G = \mathbb{Z}_2$ and $|\cdot|$ the normal absolute value. 
For $n,k\in\mathbb{N}$, $k\leq n$, we denote by $\polych^k$ the group of polyhedral chains of dimension $k$ with coefficients in $G$ i.e.\ the set of formal sums of compact, convex, oriented polyhedra of dimension $k$ in $\mathbb{R}^n$ with coefficients in $G$ together with the obvious addition. 
We identify a polyhedron that results from glueing along a shared face (and compatible orientation) with the sum of the individual polyhedra. 
Also, we identify a polyhedron of opposite orientation with the negative of the original polyhedron.
An element $P\in \polych^k$ can thus be written as
\begin{align} \label{prelim:polych}
P
\ = \ \sum_{i=1}^p g_i \sigma_i\, ,
\end{align}
where $g_i\in G$ and $\sigma_i$ are compact, convex, oriented polyhedra that can be chosen to be non-overlapping.
Note that in our case of $G=\mathbb{Z}_2$, the non trivial coefficients $g_i$ all equal $1$ and that the orientational aspect of the above definition becomes trivial.
This reflects the symmetry of the director field $\nn\sim -\nn$ in the sense that around singularities we change orientation of $\nn$ without changing the physics.
In other words, we can lift $Q$ locally away from singularities to obtain a well-defined director $\nn$, but in general it is not possible to combine those local liftings into a global one.
The \emph{boundary} $\partial\sigma$ of a polyhedron $\sigma$ is the formal sum of the $(k-1)-$dimensional polyhedral faces of $\sigma$ with the induced orientation and coefficient $1$ under the above mentioned identifications. Note that $\partial(\partial\sigma)=0$. We can linearly extend this operator to a boundary operator $\partial:\polych^k\rightarrow\polych^{k-1}$.

\paragraph{Mass and flat norm.}
For a polyhedral chain $P\in\polych^k$ written as in \eqref{prelim:polych}, we define the \emph{mass} $\MM(P) = \sum_{i=1}^p |g_i| \mathcal{H}^k(\sigma_i)$ and the \emph{flat norm} $\FF(P)$ by
\begin{align*}
\FF(P)
\ = \ \inf\{ \MM(Q) + \MM(R) \sd P = \partial Q + R\, , Q\in\polych^{k+1}\, , R\in\polych^k \}\, .
\end{align*}
Obviously it holds $\FF(P)\leq \MM(P)$ and $\FF(\partial P)\leq \FF(P)$.
One can show that $\FF$ defines a norm on $\polych^k$ \cite[Ch. 2]{Fleming1966}.

\paragraph{Flat chains and associated measures.}
We define the space of \emph{flat chains} $\flatch^k$ to be the $\FF-$completion of $\polych^k$. 
The boundary operator $\partial$ extends to a continuous operator $\partial:\flatch^k\rightarrow\flatch^{k-1}$ and we still denote by $\MM$ the largest lower semicontinuous extension of the mass which was defined on $\polych^k$. Furthermore, one can show \cite[Thm 3.1]{Fleming1966} that for all $A\in\flatch^k$
\begin{align*}
\FF(A)
\ = \ \inf\{ \MM(Q) + \MM(R) \sd P = \partial Q + R\, , Q\in\flatch^{k+1}\, , R\in\flatch^k \}\, .
\end{align*}
For a finite mass flat chain $A\in\flatch^k$ and a measurable set $X\subset\mathbb{R}^n$, we can define the \emph{restriction} $A\restr X$ via an approximation by polyhedral chains,
for which the restriction coincides with the intersection under some technical assumptions and passing to the limit. 
A precise definition is given in \cite[Ch. 4]{Fleming1966}.
\dsg{We also use the notation $\MM(A,X):=\MM(A\restr X)$ and $\FF(A,X):=\FF(A\restr X)$.}
To each flat chain $A\in\flatch^k$, there exists an associated measure $\mu_A$ (see \cite[Ch. 4]{Fleming1966}) such that for each $\mu_A-$measurable set $X$,  $A\restr X$ is a flat chain and $\mu_A(X)=\MM(A\restr X)$.
The \emph{support} of $A$ is denoted $\supp(A)$ and given (if it exists) by the smallest closed set $X$ such that for every open set $U\supset X$ there exists a sequence of polyhedral chains $(P_j)_j$ approximating $A$ and such that all cells of all $P_j$ lie inside $U$. If $A$ is of finite mass, then $\supp(A)=\supp(\mu_A)$ (see \cite[Thm. 4.3]{Fleming1966}).

\paragraph{Cartesian products and induced mappings.}
In the case of finite mass flat chains $A,B$ (or one of the two chains having finite mass \emph{and} finite boundary mass), it is possible to define the product $A\times B$ (by polyhedral approximation), see e.g.\ \cite[Sec. 6]{Fleming1966}.
In particular, it is always possible to define $[0,1]\times B$. 
For $U\subset\mathbb{R}^n,V\subset\mathbb{R}^m$ open sets and a Lipschitz function $f:U\rightarrow V$, one can define an induced mapping $f_\#$ on the level of flat chains, i.e.\ for a flat chain $A$ supported in $U$, $f_\#A$ is a flat chain supported in $V$ (see \cite[Sec. 5]{Fleming1966} and \cite[Sec. 2 and 3]{Federer1960}).

\paragraph{Generic properties and Thom transversality theorem.} 
A property of an object (such as a function or a set) that can be achieved by an arbitrarily small perturbation of the object is called \emph{generic}.
Examples are properties that hold in an ``almost everywhere'' measure theoretic sense or that are true on a dense subset.
In this work we encounter two such properties:
Two dimensional planes \emph{generically} do not contain a fixed single point (can be achieved by shifting normal to the plane).
The second one is that a smooth map $f:M\rightarrow N$ \emph{generically} intersects a submanifold $S\subset N$ transversely i.e.\ $\dx f(T_xM)+T_{f(x)}S = T_{f(x)}N$ for all points $x\in f^{-1}(S)$.
The latter will be used to apply Thom's transversality theorem \cite{Thom1956} in the form given in \cite[Thm. 2.7]{Hirsch1976}. 

\paragraph{Deformations.} In certain situations it is beneficial to approximate a flat $k-$chain $A$ by a polyhedral $k-$chain $P$. 
The usual way to construct $P$ is through ``pushing'' $A$ onto the $k-$skeleton of a grid in the following way.
In this paper, a (cubic) grid of size $h$ is understood to be a cell complex in $\mathbb{R}^3$ which consists of cubes of side length $h$. 
The ``pushing'' operation consists of a radial projection of $A$ from the center of each cube onto the faces of the cubes, assuming that the center does not lie on $A$.
Then, on each face the projected flat chain gets again projected from the center of the face onto the edges (as long as the projected chain does not contain any face center point).
This procedure is stopped once the projected flat $k-$chain is contained in the $k-$dimensional skeleton.
This deformation procedure is a crucial ingredient to prove that
every $A\in\flatch^k$ can be written as $A=P+B+\partial C$, where $P\in\flatch^k$ is a polyhedral chain, $B\in \flatch^k$ and $C\in \flatch^{k+1}$ satisfy the estimates $\MM(P)\lesssim \MM(A)+h\MM(\partial A)$, $\MM(\partial P)\lesssim \MM(\partial A)$, $\MM(B)\lesssim h\MM(\partial A)$ and $\MM(C)\lesssim h\MM(A)$, see \cite{White1999} or \cite[Thm. 7.3]{Fleming1966}.

\paragraph{Compactness.}
One point of importance from the perspective of calculus of variations are the compactness properties of flat chains whose mass and the mass of their boundary is bounded. We will use the result from \cite[Cor. 7.5]{Fleming1966} which holds for coefficient groups $G$ such that for all $M>0$ the set $\{g\in G\sd |g|\leq M\}$ is compact. This is trivially true in our case where $G=\mathbb{Z}_2$. Let $K\subset\mathbb{R}^n$ be compact and $C_1>0$. Then the corollary states that
\begin{align*}
\{ A\in\flatch^k\sd \supp(A)\subset K\text{ and } \MM(A)+\MM(\partial A)\leq C_1 \}
\end{align*}
is compact.

\paragraph{Rectifiability.}
Another aspect of flat chains concerns their regularity and if one can define objects originating in smooth differential geometry such as tangent spaces. 
It turns out that this can be achieved a.e.\ provided the flat chain is rectifiable.
By definition, rectifiability of a flat chain $A\in \flatch^k$ means that there exists a countable union of $k-$dimensional $C^1-$submanifolds $N$ of $\mathbb{R}^n$ such that $A=A\restr N$ \cite[Sec. 1.2]{White1999a}.
A rectifiable flat chain admits an approximate tangent plane for $\mathcal{H}^k-$a.e.\ $x\in A$ \cite[Thm 2.83]{Ambrosio2000}. 
Such a point $x$ is called rectifiability point of $A$ and we denote $\rect(A)$ the set of all points of rectifiability of $A$.
For finite groups $G$, finite mass $\MM(A)<\infty$ implies rectifiability of $A$, see \cite[Thm 10.1]{Fleming1966}.

\section{Statement of result}

Our main result concerns the asymptotic behaviour of the energy $\Eex$ for $\eta,\xi\rightarrow 0$. 
Physically speaking, we consider the regime of large particles and weak magnetic fields, see \cite{ACS2021,Gartland2018} for more discussion of the physical interpretation of our limit. 

The liquid crystal occupies a region $\Omega$ outside a solid particle $E$, i.e.\ $\Omega = \mathbb{R}^3\setminus E$. 
We assume the boundary of the particle $\M\defi \partial E$ to be sufficiently smooth for our analysis, that is we require $\M$ to be a closed, compact and oriented manifold of class at least $C^{2}$. 
The regularity will be needed to ensure that the outward unit normal field $\nu\in C^1$ of $\partial E$ or in other words $\M$ has continuous curvature.
Furthermore, we assume that
\begin{align*}
\Gamma \defi \{\omega\in\M\sd \nu_3(\omega)=0\}
\end{align*}
is a $C^2-$curve (or a union thereof) in $\M$
and that $\nabla_{\omega} \nu_3 \neq 0$ everywhere on $\Gamma$ (seen inside the tangent bundle $T\M$), see also Remark~\ref{rem:technical_comments_on_main_thm}.

In order to make the minimization of the energy $\Eex$ non trivial, we impose the following boundary condition on $\M$:
\begin{align} \label{eq:bc} 
Q 
\ &= \ 
\dsg{Q_b}
\ := \
s_*\left( \nu \otimes \nu - \frac{1}{3}\id \right)
\quad\text{ on }\M\, .
\end{align}
Indeed, without \eqref{eq:bc} the minimizer of $\Eex$ would be the constant function $\Qexinfty$.
We define the class of admissible functions $\mathcal{A}\defi\{ Q\in H^1(\Omega,\Sym) + \Qexinfty\sd Q\text{ satisfies }\eqref{eq:bc} \}$.
It is convenient to define the energy $\Eex^\A$ for $Q\in H^1(\Omega,\mathbb{R}^{3\times 3})+\Qexinfty$ by
\begin{align*}
\Eex^\A(Q) \defi \begin{cases}
\Eex(Q) & \text{if } Q\in\A\, , \\
+\infty & \text{otherwise}\, .
\end{cases}
\end{align*}
We also use the notation $\Eex(Q,U)$ (resp. $\Eex^\A(Q,U)$) for the energy $\Eex$ (resp. $\Eex^\A$) of the function $Q$ on the set $U$.

\begin{theorem}\label{thm:main}
Suppose that
\begin{equation} \label{eq:beta}
\eta|\ln(\xi)|\rightarrow \beta\in (0,\infty) \quad \text{ as } \eta\rightarrow 0\, .
\end{equation}
Then $\eta\:\Eex^\A\rightarrow\E_0$ in a variational sense, where the limiting energy $\E_0$ is given by
\begin{equation} \label{def:lim_energy} 
\E_0(T,S) = 2 s_* c_* E_0(\M,\ee_3) + 4 s_*c_*\int_\M |\cos(\theta)| \dx\mu_{T\restr\M} + \frac{\pi}{2}s_*^2\beta \MM(S) + 4s_*c_*\MM(T\restr\Omega) 
\end{equation}
for $(T,S)\in\A_0\defi \{ (T,S)\in\flatch^2\times\flatch^1\sd \partial T = S+\Gamma \}$ and where
\begin{align*}
E_0(\M,\ee_3) \defi \int_{\{\nu_3>0\}} (1-\cos(\theta))\dx\omega + \int_{\{\nu_3\leq 0\}} (1+\cos(\theta))\dx\omega \, .
\end{align*}
The letter $\theta$ is used to denote the angle between $\ee_3$ and the outward unit normal vector $\nu(\omega)$ at a point $\omega\in\M$.
The variational convergence is to be understood in the following sense: Along any sequence $\eta_k,\xi_k\rightarrow 0$ with $\eta_k|\ln(\xi_k)|\rightarrow\beta$ (not labelled in the following statements):
\begin{enumerate}
\item \textit{Compactness and $\Gamma-$liminf:} For any sequence $\Qex\in H^1(\Omega,\mathbb{R}^{3\times 3})+\Qexinfty$ such that there exists a constant $C>0$ with
\begin{align}\label{thm:global_Eex_bound}
\eta\: \Eex^\A(\Qex)\leq C\, ,
\end{align}
there exists $(T,S)\in\A_0$, functions
$\widetilde{\Qex}\in C^\infty(\Omega,\Sym)$ with $\lim_{\eta,\xi\rightarrow 0}\Vert \Qex-\widetilde{\Qex}\Vert_{L^2} = 0$
and $Y_{\eta,\xi}\in \Sym$ with $\Vert Y_{\eta,\xi}\Vert\rightarrow 0$ such that $\Tex = (\widetilde{\Qex} - Y_{\eta,\xi})^{-1}(\T)$, $\Sex = (\widetilde{\Qex} - Y_{\eta,\xi})^{-1}(\C)$ for $\T$ and $\C$ given as in \eqref{def:cal_T},\eqref{def:cal_C} are smooth flat chains with
\begin{align}\label{thm:main:cptness_pre}
\partial\Tex = \Sex \ + \ \Gamma_{\eta,\xi}\, .
\end{align}
Here, $\Gamma_{\eta,\xi}$ is a smooth approximation of $\Gamma$ which converges to $\Gamma$ in Hausdorff distance and hence also in flat norm.
For any bounded measurable set $B\subset \Omega$ it holds $\Eex^\A(\widetilde{\Qex},B)\leq \Eex^\A(\Qex,B) + C_B$ for a constant $C_B>0$  and, up to extracting a subsequence, 
\begin{equation} \label{thm:main:cptness}
\begin{split}
\lim_{\eta,\xi\rightarrow 0} \FF(\Tex - T,B) = 0 \, ,
\quad \lim_{\eta,\xi\rightarrow 0} \FF(\Sex - S,B) = 0 \, .
\end{split}
\end{equation} 
Furthermore, we have
\begin{equation} \label{thm:main:liminf}
\liminf_{\eta\rightarrow 0} \eta\:\Eex^\A(\Qex) \geq \E_0(T,S)\, .
\end{equation}
\item \textit{$\Gamma-$limsup:} For any $(T,S)\in\A_0$, there exists a sequence $\Qex\in \A$ with $\Vert \Qex\Vert_{L^\infty}\leq \sqrt{\frac{2}{3}}\sexstar$ such that \eqref{thm:main:cptness_pre},\eqref{thm:main:cptness} hold and
\begin{equation} \label{thm:main:limsup}
\limsup_{\eta\rightarrow 0} \eta\:\Eex^\A(\Qex) \leq \E_0(T,S)\, .
\end{equation}
\end{enumerate}
\end{theorem}

\begin{remark}[Assumptions in the theorem]\label{rem:technical_comments_on_main_thm}
\begin{enumerate}
\item We note that due to our assumptions $\beta\in (0,\infty)$, the global energy bound \eqref{thm:global_Eex_bound} can be reformulated as
\begin{align*}
\Eex^\A(\Qex) \ \leq \ \tilde{C} \ |\ln(\xi)|\, .
\end{align*}
This reflects the classical behaviour of a logarithmic divergence of the energy close to singularities as already observed in earlier works e.g.\ in \cite{Bethuel1994}.
\item If $\Qex$ is smooth enough (for example $C^2$) and verifies a Lipschitz estimate as in \eqref{prop:prop_reg_seq_Lip} for $n\sim\ds{\xi^{-a},\, a\in [1,2)}$, we can choose $\widetilde{\Qex}=\Qex$ in the above theorem. 
This is particularly true if $\Qex$ is a minimizer of \eqref{def:Eex}. 
Indeed, from the Euler-Lagrange equations, one can deduce the regularity and the required estimate on the gradient \cite[Lemma A.2]{Bethuel1993} with $n\sim\xi^{-1}$. 
\item The compactness claim holds for almost every $Y\in\Sym$ with $\Vert Y\Vert$ small enough. 
The norm converging to zero is needed to recover the condition $\partial T = S + \Gamma$, the stated energy densities on $T\restr\M$, and the coefficient in front of $\MM(T\restr\Omega)$.
\item Another possibility of introducing $\Sex$ is by using the operator $\mathbf{S}$ defined in \cite{Canevari2019,Canevari2020}. 
In our notation, this operator maps a function $Q$ from $(L^\infty\cap W^{1,1})(\Omega,\Sym)$ to $L^1(B_{\alpha^{*}}(0),\flatch^1)$, where $\alpha_*>0$ and $B_{\alpha^{*}}(0)\subset \Sym$. 
In other words, $\mathbf{S}$ allows us to define a flat $1-$chain $\mathbf{S}_Y(\Qex)$ for $\Qex\in (L^\infty\cap W^{1,1})(\Omega,\Sym)$ and $Y\in B_{\alpha^{*}}(0)$.
\item The assumption of $\Gamma=\{\omega\in\M\sd\nu(\omega)\cdot\ee_3=0\}$ being a $C^2-$curve is not very restrictive. In fact, this can already be achieved by a slight deformation of $\M$ which changes the energies $\Eex$ and $\E_0$ in a continuous way.
The assumption that $\nabla_{\omega} \nu_3$ is nowhere vanishing on $\Gamma$ is used as a sufficient condition to ensure that the perturbed sets $\Gamma_{\eta,\xi}$ stay regular and in a neighbourhood of $\Gamma$.
In fact, since $\nu_3=0$ on $\Gamma$ the derivative vanishes in the direction tangential to $\Gamma$, so the condition is only on the part of $\nabla_\omega \nu_3$ normal to $\Gamma$.
In particular, the condition is verified if the Gaussian curvature $|\kappa_\M|>0$ on $\Gamma$.
\end{enumerate}
\end{remark}

\begin{remark}[An alternative formulation of $\E_0$] \label{rem:comments_on_main_thm}
\begin{enumerate}
\item We can express the energy \eqref{def:lim_energy} in a slightly different way by writing $\mu_{T\restr\M} = \chi_G \mathcal{H}^2\restr\M$ for a mesurable set $G\subset\M$ and defining
\begin{align}\label{def:set_F}
F \defi \{\omega\in\M\setminus G\sd \nu(\omega)\cdot\ee_3>0\} \cup \{\omega\in\M\cap G\sd\nu(\omega)\cdot\ee_3\leq 0\}\, .
\end{align}
Then, \eqref{def:lim_energy} reads
\begin{align}\label{def:limit_energy_w_F}
\E_0(T,S) 
\ &= \ 2s_*c_*\int_F (1 - \cos(\theta)) \dx\omega + 2s_*c_*\int_{\M\setminus F} (1 + \cos(\theta)) \dx\omega  \\ \nonumber
\ &\:\:\:\: + \ \frac{\pi}{2}s_*^2\beta \MM(S) + 4s_*c_*\MM(T\restr\Omega)\, .
\end{align}
The idea behind this reformulation and the definition of $F$ is the following: 
Assume for $\xi,\eta>0$ that $Q$ takes values in $\N$ such that (at least locally) we can lift $Q$ to a director field $\nn$.
Because of the boundary condition, we can assume that for a given point $\omega\in\M$ it holds that $\nn(\omega)=\nu(\omega)$. 
Following a ray in normal direction starting from $\omega$, $\nn$ must approach $\pm\ee_3$ since far from the particle, $Q$ must be close to $Q_\infty$.
If $\nu_3(\omega)>0$, it is energetically favourable for $\nn$ to approach $+\ee_3$.
On the other hand, the ray intersecting $T$ means that $\nn$ switches sign, i.e.\ if we start from $\nu_3(\omega)<0$ and cross $T$ only once, $\nn$ converges to $+\ee_3$.
In this sense, the set $F$ can be understood as the region on $\M$ in which the lifting $\nn$ along the rays starts from $\nu$ and approach $+\ee_3$, while on $\M\setminus F$ the director $\nn$ turns from $\nu$ to $-\ee_3$.
Previously, the energy $E_0(\M,\ee_3)$ describes the \emph{minimal} energy concentrated on $\M$, i.e.\ $\nn$ always turns in the energetically favourable direction and the integral involving $\mu\restr\M$ accounts for the additional energy caused by intersecting $T$.
See Figure~\ref{fig:limit_peanut_horizontal} for an illustration of the different quantities appearing in \eqref{def:limit_energy_w_F}.
\item For convex particles $E$, there exists an orthogonal projection $\Pi:\Omega\rightarrow\M$. 
By convexity of $E$, we find that $\E_0(\Pi_\#T,\Pi_\#S)\leq \E_0(T,S)$, so that we can restrict ourselves to the case $T\restr\Omega=0=S\restr\Omega$.
Using \eqref{def:set_F}, we find that $\partial F = \Pi_\# S$ and \eqref{def:limit_energy_w_F} becomes
\begin{align*}
\E_0(\Pi_\#T,\Pi_\#S) 
\ &= \ 2s_*c_*\int_F (1-\cos(\theta))\dx\omega + 2s_*c_*\int_{\M\setminus F} (1-\cos(\theta))\dx\omega  + \frac{\pi}{2}s_*^2\beta\MM(\partial F)\, .
\end{align*}
In particular, \eqref{def:lim_energy} is a generalization of the limit energy $\E_0$ defined in \cite{ACS2021}.
\end{enumerate}
\end{remark}

\begin{figure}[H]
\begin{center}
\includegraphics[scale=2.5]{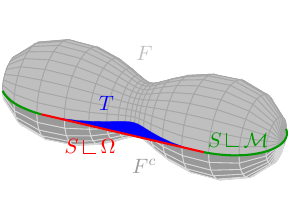}
\end{center}
\caption{Illustration of flat chains $T,S$ and the sets $F,F^c$ appearing in the limit energy $\E_0$.}
\label{fig:limit_peanut_horizontal}
\end{figure}

\begin{remark}[Physical interpretation of $T$ and $S$] \label{rem:comments_on_interpret_T_S}
The line singularity observed in physical experiments \cite{MondainMonval1999,Loudet2001,Loudet2002} has its origin in the isotropic or oblate uniaxial-biaxial defect core of the director field.
In our mathematical framework this corresponds to the set where $\Qex$ takes values in $\C$ and is therefore represented by $\Sex$ which tends towards $S$ in the limit model.
Note that it is a priori not possible to distinguish $+\frac12$ and $-\frac12$ defect lines (see Figure~\ref{fig:limit_peanut_vertical} (left)). 
But since the physical system as a whole must have a trivial topological degree, one can deduce in the situation of Figure~\ref{fig:limit_peanut_vertical} that one $+\frac12$ and two $-\frac12$ defect lines must be present.
By symmetry the line in the middle must be of degree $+\frac12$.

Point singularities of the director $\nn$ are represented by simply connected components of $T$ in our model due to the following reasoning.
As illustrated in Figure~\ref{fig:interpret_T}, the set where $\nn_3=0$ attaches to $\Gamma$ (yellow points on the surface of the sphere) and necessarily passes through the point singularity and creates a simply connected component.
However, with this description it is not possible to determine the exact position of the point defect on the surface $T$.
In the case of a minimizing $T$ around a spherical inclusion, $T$ will approach the particle surface since the nematic and magnetic exchange length become small w.r.t. the particle radius and thus $T$ forms a half-sphere (compare with \cite[Ch. 6]{ACS2021}).
In the case of a peanut-shaped particle aligned with the magnetic field we expect one of three different minimizing configurations, depending on $\beta$, see Figure~\ref{fig:limit_peanut_vertical}.
In particular, there exists a non-simply connected component of $T$ which does not correspond to a point defect, but originates in the connection of two components of $\Gamma$.
In summary, $T$ is a surface which connects $\Gamma$ to the singular set (lines and points).
\end{remark}

\begin{figure}[H]
\begin{center}
\includegraphics[scale=1.0]{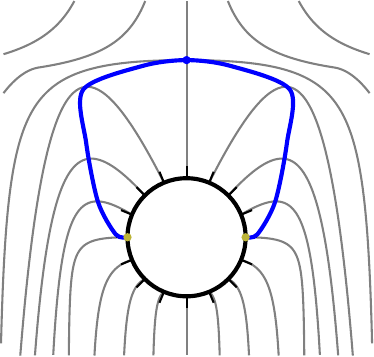}
\end{center}
\caption{Illustration of the integral lines of the director field $\nn$ around a spherical inclusion and the level set $\{x\in\Omega\sd\nn_3(x)=0\}$ representing $T$. The point defect lies on $T$.}
\label{fig:interpret_T}
\end{figure}

\begin{figure}[H]
\begin{minipage}{0.32\textwidth}
\hspace*{-2cm}
\includegraphics[scale=2.0]{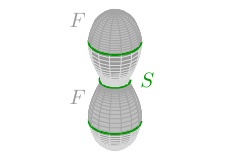}
\end{minipage}
\begin{minipage}{0.32\textwidth}
\hspace*{-2cm}
\includegraphics[scale=2.0]{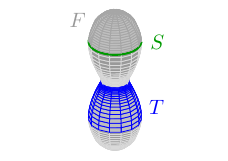}
\end{minipage}
\begin{minipage}{0.32\textwidth}
\hspace*{-2cm}
\includegraphics[scale=2.0]{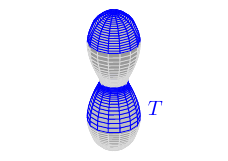}
\end{minipage}
\caption{Expected minimizers of $\E_0$ for $\beta\ll 1$ (left), $\beta\gg 1$ (right) and intermediate $\beta$ (center). 
For small $\beta$ the line $S$ has the tendency to stick to $\M$ and optimize $F$, thus no $T$ appears. 
Here, $S=\Gamma$ and $F$ consists of two connected components bounded by the three components of $S$.
This configuration corresponds to three Saturn rings around the particle.
For intermediate $\beta$ one may find a configuration as depicted in the middle, i.e.\ the energy is decreased by joining two parts of $S$ by a surface $T$ glued to $\M$.
This leads to the disappearance of the two rings that have been connected by $T$, and $F$ contains only the part of $\M$ above $S$. 
Finally, for large $\beta$, the last ring disappears and we obtain a dipole configuration in which $S=0$, $F=\emptyset$ and $T$ has two components, see Remark~\ref{rem:comments_on_interpret_T_S}.
This last configuration has been observed experimentally, see\cite[Fig. 2(a)-(c)]{Sahu2019}.}
\label{fig:limit_peanut_vertical}
\end{figure}

\section{Compactness}
\label{sec:cpt}

The structure of this section is as follows. 
We regularize the sequence $\Qex$ in the first subsection. 
For this new sequence $\Qexn$, we define a $2-$chain $\Texn\in\flatch^2$ and $1-$chain $\Sexn\in\flatch^1$ such that $\partial \Texn=\Sexn$ and we have bounds on the masses to get the existence of limit objects $T$ and $S$ with $\partial T=S$.  
This construction is carried out in steps in the subsections two, three and four.
We distinguish the case of $\Qexn$ being close to $\N$ and hence almost prolate uniaxial and the complementary case when $\Qexn$ is far from $\N$, e.g.\ when $\Qexn$ is isotropic or oblate uniaxial close to the boundary $S$. 
The passage to the limit is to happen in the last subsection.

\subsection{Approximating sequence}

\ds{
This section is devoted to the definition of a sequence of smooth functions $\Qexn$, replacing $\Qex$ in our analysis and proving the properties required for the estimates in the following chapters. More precisely, we need that 
\begin{itemize}
\item the sequence $\Qexn$ approximates $\Qex$,
\item $\Qexn|_\M$ approaches $Q_b$ in $C^1$,
\item $\Qexn$ verifies the energy bound $\eta\:\Eex(\Qexn)\leq \eta\:\Eex(\Qex) + o(1) \leq \tilde{C}$ and 
\item the estimate $\mathrm{Lip}(\Qexn)\leq C\: n$ holds.
\end{itemize}
While $n$ is introduced as regularization parameter, we will later choose $n$ dependent on $\xi$ to obtain a sequence which only depends on the original parameters $\eta,\xi$.
More explicitly, we can simply take e.g.\ $n = \xi^{-\tfrac74}$ as we will see later.

For technical reasons, we are going to extend $\Qex$ into a small neighbourhood into the interior of $E$. 
Since $\M$ is compact and of class $C^{2}$, we can fix a small radius $\radM>0$ such that $\M$ satisfies the inner ball condition for all radii $r \leq 2\radM$. 
In particular, $2\radM$ is smaller than the minimal curvature radius of $\M$. 
For $x\in E$ such that $\dist(x,\M)<2\radM$, define 
\begin{align*}
\Qex(x) 
\ = \ 
\sstar\left(\overline{\nu}(x)\otimes\overline{\nu}(x)-\frac13\id\right)\, ,
\end{align*}
where $\overline{\nu}(x)=\nu(\Pi_\M(x))$, $\Pi_\M$ the orthogonal projection onto $\M$, is the obvious extension of the outward normal unit vector field $\nu$ \dsg{of $\partial E$} in $E$.
Also by $C^2-$regularity of $\M$ and given $n\in\mathbb{N}$, there exists a $C^2-$diffeomorphism $\Phi_n:\mathbb{R}^3\rightarrow\mathbb{R}^3$ such that 
\begin{align}\label{def:Phi_n}
\Phi_n(x) \ = \ \begin{cases}
x & x\in\Omega\text{ with } \dist(x,\M) > \frac{1}{\sqrt{n}}\, , \\
x \dsg{-} \frac1n\: \overline{\nu}(x) & x\in \mathbb{R}^3\text{ with }\dist(x,\M)\leq \frac{1}{n}\, ,
\end{cases}
\end{align}
and $|\Phi_n(x) - x|\leq \frac{1}{n}$ and in operator norm $|\nabla \Phi_n(x)|\leq 1 + \frac{C}{\sqrt{n}}$ for all $x\in\Omega$.

Let $\Pi_R:\Sym\rightarrow B_{R}(0)\subset \Sym$ be the orthogonal projection with $\sqrt{\frac{2}{3}}s_*\leq R < \infty$ to be fixed later. 
Furthermore, let $\varrho\in C_c^\infty(\mathbb{R}^3)$ be a convolution kernel with $0\leq \varrho\leq 1$, $\varrho(x)=0$ if $|x|>1$, $\int_{\mathbb{R}^3}\varrho(x)\dx x=1$ and $\Vert\nabla\varrho\Vert_{\infty}\leq 1$.  We set $\varrho_n(x) = n^{3}\varrho(nx)$. Then, for $n\geq 2 \radM^{-1}$ we define $\Qexn(x)$ for $x\in\overline{\Omega}$ as the convolution
\begin{align}\label{def:reg_seq_eq}
\Qexn(x) \ \defi \ \left((\Pi_R\circ\Qex\circ\Phi_n)*\varrho_n\right)(x) \, . 
\end{align}

\begin{remark} 
\begin{enumerate}
\item \dsg{In this definition $\Qex\circ\Phi_n$ is defined in the interior of $E$ up to distance $\frac1n$ of $\M$ which is necessary in order to define the convolution}. 
\item Through the convolution, we change the boundary values of $\Qex$, i.e.\ $\Qexn$ does not necessarily satisfy \eqref{eq:bc}. 
The diffeomorphism $\Phi_n$ ensures that the regularized sequence $\Qexn$ defined above approximates the boundary data $Q_b$ in $C^1$.
\item Because of the convolution, the approximations of $T$, that we are about to construct, will not end
on $\Gamma$, but on a set $\Gamma_n= \{ \omega\in\M\sd (\overline{\nu}_3*\varrho_n)(\omega)=0 \} $ (which is again a line) in the neighbourhood of $\Gamma$.
Because of the $C^1-$convergence of $\Qexn|_\M\rightarrow Q_b$ we can use a perturbation argument to deduce that $\Gamma_n$ converges in Hausdorff distance and in flat norm to $\Gamma$.
The details of this result are provided in Section~\ref{subsec:compactness_alpha_0}.
\end{enumerate}
\end{remark}

The following proposition shows that this sequence has indeed the desired properties.

\begin{proposition} \label{prop:prop_reg_seq}
There exists $R\geq\sqrt{\frac{2}{3}}s_*$ such that the sequence $\Qexn$ defined in \eqref{def:reg_seq_eq} verifies:
\begin{enumerate}
\item The functions $\Qexn$ are smooth and there exists a constant $C>0$ such that
\begin{align} \label{prop:prop_reg_seq_Lip}
\Vert \nabla \Qexn \Vert_{L^\infty} \ \leq \ C\: n \, .
\end{align}
\item If $\eta\:\Eex(\Qex)\lesssim 1$, we have convergence $\Qexn - \Qex\rightarrow 0$ in $L^2$ and $\Qexn|_\M \rightarrow Q_b$ in $C^1$ for $n\rightarrow\infty$ and $\xi,\eta\rightarrow 0$ provided $\eta n$ diverges in the limit $n\rightarrow\infty$, $\eta\rightarrow 0$.
\item If $\eta\:\Eex(\Qex)\lesssim 1$ and $n\sim\xi^{-a}$ for some $a>\tfrac32$, then there exist constants $C_1,C_2>0$ such that for all measurable sets $\Omega'\subset\Omega$ with $|\Omega'|<\infty$  the energy of $\Qexn$ can be bounded as
\begin{equation}\label{prop:prop_reg_seq_E_bound}
\begin{aligned} 
\Eex(\Qexn,\Omega')
\ &\leq \ 
\Bigg( 1 + C_1\Bigg(\frac{1}{\sqrt{n}} + \frac{\xi^2}{\eta^2} \Bigg)\Bigg)\Eex( \Qex, B_{\frac{2}{n}}(\Omega')\cap\Omega) \\
\ &\qquad\qquad+ \ C_2\Bigg( 
  \frac{1}{\xi^\frac32 \eta n} 
+ \left( \frac{|\Omega'|}{\eta\xi^3 n^2} \right)^\frac12
\Bigg) \\
\ &\leq \ (1+o(1))\: \Eex( \Qex, B_{\frac{2}{n}}(\Omega')\cap\Omega) + o(1)\, ,
\end{aligned} 
\end{equation}
where $B_r(\Omega')$ denotes the $r-$neighbourhood around $\Omega'$ and $o(1)\rightarrow 0$ as $\eta,\xi\rightarrow 0$.
\end{enumerate}
\end{proposition}

Before proving the proposition, we show a series of three lemmata detailing how each step of construction from $\Qex$ to $\Qexn$ modifies the energy.

\begin{lemma} \label{lem:lem_reg_seq_Pi_R}
There exists $R\geq\sqrt{\frac{2}{3}}s_*$ such that we have convergence $\Pi_R\circ\Qex - \Qex\rightarrow 0$ in $L^2$ as $\xi,\eta\rightarrow 0$ if $\eta\:\Eex(\Qex)\lesssim 1$. 
For all measurable sets $\Omega'\subset\Omega$ it holds that
\begin{align*}
\Eex(\Pi_R\circ \Qex,\Omega')
\ \leq \ \Eex(\Qex,\Omega')\, .
\end{align*}
\end{lemma}

\begin{proof}
For the $L^2$ convergence, we note that $\Qex\neq \Pi_R\circ\Qex$ only on the set $A\defi \{ x\in\Omega\sd |\Qex(x)|> R \}$. 
Fixing $R>\sqrt{\frac{2}{3}}s_*$ and using Proposition~\ref{prop:f__g_geq} we get
\begin{align*}
\int_\Omega |\Qex - \Pi_R\circ\Qex|^2 \dx x
\ &\leq \ \int_{A} |\Qex - \Pi_R\circ\Qex|^2 \dx x
\ \lesssim \ \int_{A} \dist^2(\Qex,\N_{\eta,\xi}) \dx x \\
\ &\lesssim \ \int_{\Omega} F(\Qex) \dx x 
\ \lesssim \ \xi^2 \Eex(\Qex,\Omega) \, ,
\end{align*}
where we used the notation $F(Q)=f(Q)+\frac{\xi^2}{\eta^2}g(Q)+\xi^2 C_0$. Since $\eta\: \Eex(\Qex,\Omega)$ is bounded, $\Vert\Qex-\Pi_R\circ\Qex\Vert_{L^2}\lesssim \tfrac{\xi}{\sqrt{\eta}}\rightarrow 0$ as $\eta,\xi\rightarrow 0$.

It remains to prove the energy bound for $\dsg{\Pi_R\circ\Qex}$.
For this, we directly get that
\begin{align*}
\int_{\Omega'} |\nabla (\Pi_R\circ\Qex)|^2 \dx x 
\ &\leq \ \int_{\Omega'} |\nabla \Qex|^2 \dx x  \, .
\end{align*}
For the bulk energy, we use \cite[Prop. 4.1]{ACS2021} to fix $R\geq \sqrt{\frac23}s_*$ such that $F(Q)\geq F(\Pi_R Q)$ for all $Q\in\Sym$, where again $F(Q)=f(Q)+\frac{\xi^2}{\eta^2}g(Q)+\xi^2 C_0$.
Integrating this pointwise inequality implies the energy inequality.
\end{proof}

\begin{lemma} \label{lem:lem_reg_seq_Phi_n}
Let $R\geq\sqrt{\frac{2}{3}}s_*$ be as in Lemma~\ref{lem:lem_reg_seq_Pi_R} and $\Phi_n$ defined by \eqref{def:Phi_n}.
Then, $\Pi_R\circ\Qex\circ\Phi_n - \Pi_R\circ\Qex \rightarrow 0$ in $L^2$ as $n\rightarrow \infty$ and there exist constants $C_1,C_2>0$ such that for all measurable sets $\Omega'\subset\Omega$
\begin{align*}
\Eex(\Pi_R\circ \Qex\circ\Phi_n,\Omega')
\ &\leq \ 
\left( 1 + \frac{C_1}{\sqrt{n}} \right)\Eex(\Pi_R\circ\Qex,B_{\frac{1}{n}}(\Omega')\cap\Omega)
\ + \ \frac{C_2}{\dsg{\eta^2} n}
\, .
\end{align*}
\end{lemma}

\begin{proof}
We start again with the $L^2-$convergence noting that by definition of $\Phi_n$ it holds $|(\Pi_R\circ \Qex\circ\Phi_n)(x) - (\Pi_R\circ \Qex)(x)|=0$ as soon as $\dist(x,\M)\geq\frac{1}{\sqrt{n}}$. 
The complementary set $U_n\defi \{ x\in\Omega\sd \dist(x,\M)\leq\frac{1}{\sqrt{n}} \}$ is of measure $|U_n|\lesssim (\sqrt{n})^{-1}$ an together with the $L^\infty-$ bounds $|\Pi_R\circ\Qex\circ\Phi_n|,|\Pi_R\circ\Qex|\leq R$ this implies $L^2-$convergence.

For the energy estimate we calculate, using the estimate $|\det(\nabla\Phi_n)|\leq 1 + \tfrac{C}{\sqrt{n}}$
\begin{align*}
\int_{\Omega'} |\nabla (\Pi_R\circ \Qex\circ\Phi_n) |^2 \dx x 
\ &\leq \  
\left( 1 + \frac{C}{\sqrt{n}} \right)\int_{\Omega'} |\nabla (\Pi_R\circ \Qex) |^2(\Phi_n(x)) \dx x \\
\ &\leq \  \left( 1 + \frac{C}{\sqrt{n}} \right)\int_{\Phi_n(\Omega')} |\nabla (\Pi_R\circ \Qex) |^2 \frac{1}{|\det(\nabla\Phi_n)|} \dx x \\
\ &\leq \  \left( 1 + \frac{C}{\sqrt{n}} \right)\left( C\int_{\Phi_n(\Omega')\cap E} |\nabla \overline{\nu} |^2 \dx x + 
\int_{\Phi_n(\Omega')\cap\Omega} |\nabla (\Pi_R\circ \Qex) |^2 \dx x \right) \, . 
\end{align*}
Since $\Phi_n(\Omega')\subset B_{\frac1n}(\Omega')$ and the curvature of $\M$ is bounded, we get that $\int_{\Phi_n(\Omega')\cap E} |\nabla \overline{\nu} |^2 \dx x\lesssim \frac{1}{n}$. 
Therefore
\begin{align*}
\int_{\Omega'} |\nabla (\Pi_R\circ \Qex\circ\Phi_n) |^2 \dx x 
\ &\leq \  \left( 1 + \frac{C_1}{\sqrt{n}} \right) 
\int_{B_{\frac1n}(\Omega')\cap\Omega} |\nabla (\Pi_R\circ \Qex) |^2 \dx x + \frac{C_2}{n} \, . 
\end{align*}
We proceed similarly for the bulk potential $F(Q)=f(Q)+\frac{\xi^2}{\eta^2}g(Q)+\xi^2 C_0$, giving
\begin{align*}
\int_{\Omega'} F(\Pi_R\circ\Qex\circ\Phi_n) \dx x 
\ &\leq \  
\left( 1 + \frac{C}{\sqrt{n}} \right)\int_{\Phi_n(\Omega')} F(\Pi_R\circ\Qex) \dx x \\
\ &\leq \  \left( 1 + \frac{C}{\sqrt{n}} \right)\left( \int_{\Phi_n(\Omega')\cap E} F(\Qex) \dx x + 
\int_{\Phi_n(\Omega')\cap\Omega} F(\Pi_R\circ\Qex) \dx x \right) \, . 
\end{align*}
Since $\Qex = s_*(\overline{\nu}\otimes\overline{\nu}-\frac13\id)$ in $E$, we get that $|F(\Qex)|=|\frac{\xi^2}{\eta^2}g(\Qex) + \xi^2 C_0|\lesssim \frac{\xi^2}{\eta^2}$. 
We end up with
\begin{align*}
\int_{\Omega'} \frac{1}{\xi^2}F(\Pi_R\circ\Qex\circ\Phi_n) \dx x 
\ &\leq \  \left( 1 + \frac{C}{\sqrt{n}} \right) 
\int_{B_{\frac1n}(\Omega')\cap\Omega} \frac{1}{\xi^2}F(\Pi_R\circ\Qex) \dx x + \frac{C}{\eta^2 n} \, . 
\end{align*}

\end{proof}

\begin{lemma} \label{lem:lem_reg_seq_rho_n}
Let $R\geq\sqrt{\frac{2}{3}}s_*$ be as in Lemma~\ref{lem:lem_reg_seq_Pi_R}, $\Phi_n$ defined by \eqref{def:Phi_n} and $\varrho_n$ the convolution kernel used to define $\Qexn$.
Assume that $\eta\:\Eex(\Qex)\lesssim 1$.
Then, $\Qexn -  \Pi_R\circ\Qex\circ\Phi_n \rightarrow 0$ in $L^2$ as $\eta,\xi\rightarrow 0$ and $n\rightarrow\infty$ provided $n\eta\rightarrow\infty$. Furthermore, there exist constants $C_1,C_2>0$ such that for all measurable sets $\Omega'\subset\Omega$ of finite measure
\begin{equation}\label{lem:lem_reg_seq_rho_n:eq}
\begin{split}
\Eex(\Qexn,\Omega')
\ \leq \ 
\Bigg( 1 &+ C_1\Bigg(\frac{1}{\sqrt{n}} + \frac{\xi^2}{\eta^2} \Bigg)\Bigg)\Eex( \Pi_R\circ\Qex\circ\Phi_n, B_{\frac{1}{n}}(\Omega')\cap\Omega) \\
\ &+ \ C_2\Bigg( \frac{1}{n}
+ \left(\frac{1}{\xi^\frac32 n}+\frac{1}{\xi^\frac52 n^2}\right) \Eex( \Pi_R\circ\Qex\circ\Phi_n, B_{\frac{1}{n}}(\Omega')\cap\Omega) \\
&\qquad\qquad+ \frac{1}{\xi^\frac32 n^\frac32} (\Eex( \Pi_R\circ\Qex\circ\Phi_n, \Omega'))^\frac12 
+ \frac{1}{\xi^\frac52 n^3} \\
&\qquad\qquad+ |\Omega'|^\frac12\left( \frac{1}{\xi^\frac32 n} (\Eex( \Pi_R\circ\Qex\circ\Phi_n, B_{\frac{1}{n}}(\Omega')\cap\Omega))^\frac12
+ \frac{1}{\xi^\frac32 n^\frac32}  \right)
\Bigg)
\, .
\end{split}
\end{equation} 
\end{lemma}

\begin{proof}
We introduce the notation $\tilde{Q}\defi \Pi_R\circ\Qex\circ\Phi_n$ so that $\Qexn = \tilde{Q}*\varrho_n$.
Now, we observe that for any measurable $\Omega'\subset\Omega$
\begin{equation*}
\begin{split}
\int_{\Omega'} |\Qexn-\tilde{Q}|^2 \dx x
\ &= \ \int_{\Omega'} \left|\int_{B_1(0)}(\tilde{Q}(x-\tfrac{z}{n})-\tilde{Q}(x))\varrho(z)\dx z\right|^2 \dx x \\
\ &= \ \int_{\Omega'} \left|\int_{B_1(0)}\int_{0}^{\frac{1}{n}}z\cdot \nabla\tilde{Q}(x-tz)\varrho(z)\dx z\right|^2 \dx x
\ \leq \ \frac{C}{n^2} \int_{B_{\frac1n}(\Omega')} |\nabla \tilde{Q}|^2 \dx x\, .
\end{split}
\end{equation*}
Similar to the proof of Lemma~\ref{lem:lem_reg_seq_Phi_n} we can split up the last integral into two integrals over the sets $B_{\frac1n}(\Omega')\cap\Omega$ and $B_{\frac1n}(\Omega')\cap E$ to get
\begin{equation}\label{lem:lem_reg_seq_rho_n:estim_L2_grad}
\begin{split}
\int_{\Omega'} |\Qexn-\tilde{Q}|^2 \dx x
\ &\leq \  \frac{C}{n^2} \Eex(\tilde{Q},B_{\frac1n}(\Omega')\cap\Omega)
 + \frac{C}{n^3}
\, .
\end{split}
\end{equation}
The $L^2-$convergence follows from the energy bound for $\tilde{Q}$, which is a consequence of Lemma~\ref{lem:lem_reg_seq_Pi_R}, Lemma~\ref{lem:lem_reg_seq_Phi_n} and the energy bound for $\Qex$, provided $n\eta\rightarrow \infty$ as $\xi,\eta\rightarrow 0$ and $n\rightarrow\infty$.

Next, we estimate the energy of $\Qexn$ in terms of $\tilde{Q}$. 
For the gradient term we simply obtain by Young's inequality
\begin{equation} \label{lem:lem_reg_seq_rho_n:estim_grad}
\begin{split}
\int_{\Omega'} |\nabla \Qexn|^2 \dx x 
\ &= \ \int_{\Omega'} |(\nabla (\tilde{Q}*\varrho_n)|^2 \dx x \\
\ &\leq \ \int_{B_\frac{1}{n}(\Omega')} |\nabla\tilde{Q}|^2 \dx x
\ \leq \ \int_{B_\frac{1}{n}(\Omega')\cap\Omega} |\nabla\tilde{Q}|^2 \dx x + \frac{C}{n} \, .
\end{split}
\end{equation}  

The only tricky part is the estimate for $f(\Qexn)$.
We decompose $\Omega'$ into three sets $\Lambda_1,\Lambda_2,\Lambda_3$ defined as
\begin{align*}
\Lambda_1 \ &\defi \ \{ x\in\Omega'\sd \dist(\tilde{Q},\N_{\eta,\xi})\geq \lambda \}\, , \\ 
\Lambda_2 \ &\defi \ \{ x\in(\Omega'\setminus\Lambda_1)\sd |\Qexn-\tilde{Q}|\geq \lambda \}\, , \\ 
\Lambda_3 \ &\defi \ \Omega'\setminus(\Lambda_1\cup\Lambda_2)\, , 
\end{align*}
where $\lambda>0$ will be chosen later on.
The volume of $\Lambda_1$ can be estimated as follows
\begin{align*}
\Eex(\tilde{Q},\Omega')
\ &\gtrsim \ \int_{\Lambda_1} \frac{1}{\xi^2}F(\tilde{Q}) \dx x
\  \gtrsim \ \frac{\lambda^2}{\xi^2}  |\Lambda_1|\, ,
\end{align*}
where we also used Proposition~\ref{prop:f__g_geq}. 
This gives $|\Lambda_1|\lesssim \frac{\xi^2}{\lambda^2}\Eex(\tilde{Q},\Omega')$.
Using the Lipschitz continuity of $f$ on the set $B_R(0)\subset\Sym$, we can furthermore estimate
\begin{align*}
\left( \int_{\Lambda_1} \frac{\eta}{\xi^2}|f(\Qexn) - f(\tilde{Q})| \dx x \right)^2
\ &\lesssim \ \left( \int_{\Lambda_1} \frac{\eta}{\xi^2}|\Qexn - \tilde{Q}| \dx x \right)^2 \\
\ &\lesssim \ \frac{\eta^2}{\xi^4} |\Lambda_1| \int_{\Lambda_1} |\Qexn - \tilde{Q}|^2\dx x\, .
\end{align*}
Replacing $|\Lambda_1|$ by our above estimate and using \eqref{lem:lem_reg_seq_rho_n:estim_L2_grad} we get
\begin{align} \label{lem:lem_reg_seq_rho_n:estim_Lam1}
\int_{\Lambda_1} \frac{\eta}{\xi^2}|f(\Qexn) - f(\tilde{Q})| \dx x
\ &\lesssim \ \frac{\eta}{\xi\lambda}  \frac{1}{n} \Eex(\tilde{Q},B_{\frac1n}(\Omega')\cap\Omega) 
+ \frac{\eta}{\xi\lambda} \frac{1}{n^\frac32}(\Eex(\tilde{Q},\Omega'))^\frac12\, .
\end{align}

The set $\Lambda_2$ is also seen to be small since
\begin{align*}
\lambda^2|\Lambda_2|
\ &\lesssim \ \int_{\Lambda_2} |\Qexn-\tilde{Q}|^2 \dx x
\ \lesssim \ \frac{1}{n^2} \Eex(\tilde{Q},B_{\frac{1}{n}}(\Omega')\cap\Omega) + \frac{1}{n^3}\, ,
\end{align*}
again by \eqref{lem:lem_reg_seq_rho_n:estim_L2_grad} so that $|\Lambda_2|\lesssim \frac{1}{n^2\lambda^2}\Eex(\tilde{Q},B_{\frac{1}{n}}(\Omega')\cap\Omega) + \frac{1}{\lambda^2 n^3}$.
Proceeding as for $\Lambda_1$, it follows again by Lipschitz continuity of $f$ that
\begin{equation}\label{lem:lem_reg_seq_rho_n:estim_Lam2}
\begin{split}
\int_{\Lambda_2} \frac{\eta}{\xi^2}|f(\Qexn) - f(\tilde{Q})| \dx x
\ &\lesssim \ \frac{\eta}{\xi^2 n^2 \lambda} \Eex(\tilde{Q},B_{\frac1n}(\Omega')\cap\Omega) 
+ \frac{\eta}{\xi^2} \frac{1}{\lambda n^3}
\, .
\end{split}
\end{equation}

It remains to prove an estimate for $f(\Qexn)$ on $\Lambda_3$.
To this goal, we write
\begin{align*}
f(\Qexn) - f(\tilde{Q})
\ &= \ \left(\int_0^1 (D f)(\dsg{\tilde{Q}} + s(\Qexn-\tilde{Q})) \dx s\right) \cdot (\Qexn-\tilde{Q}) \\
 &\qquad - \left(\int_0^1 (D f)(\Pi_\N( \dsg{\tilde{Q}} + s(\Qexn-\tilde{Q}) )) \dx s\right) \cdot (\Qexn-\tilde{Q}) ß, ,
\end{align*}
where we used the fact that $Df = 0$ on $\N$.
To shorten the notation, we define $Q_s\defi \dsg{\tilde{Q}} + s(\Qexn-\tilde{Q})$. 
Then
\begin{align*}
f(\Qexn) &- f(\tilde{Q}) \\
\ &= \ \left(\int_0^1 \left(\int_0^1 (D^2 f)(\Pi_\N(Q_s) + t(Q_s - \Pi_\N(Q_s))) \dx t\right)\cdot(Q_s - \Pi_\N(Q_s)) \dx s\right) \cdot (\Qexn-\tilde{Q})\, .
\end{align*}
Note that $|D^2 f|$ in the above integral is bounded since $f$ is smooth on a compact neighbourhood of $\N$.
Furthermore, since $|\Qexn-\tilde{Q}|\leq \lambda$ on $\Lambda_3$ and $\dist(\tilde{Q},\N_{\eta,\xi})\leq \lambda$, it follows from \eqref{prop:f__g_geq:Nex_N} that if $\lambda\gtrsim \tfrac{\xi^2}{\eta^2}$ then $|Q_s - \Pi_\N(Q_s)|\lesssim\lambda$.
Therefore, 
\begin{align*}
\left(\int_{\Lambda_3} |f(\Qexn) - f(\tilde{Q})| \dx x\right)^2
\ &\lesssim \ \lambda^2 |\Lambda_3| \int_{\Lambda_3} |\Qexn - \tilde{Q}|^2 \dx x \\
\ &\lesssim \ \lambda^2 |\Omega'| \frac{1}{n^2}\Eex(\tilde{Q},B_{\frac1n}(\Lambda_3)\cap\Omega)
+ \frac{\lambda^2}{n^3}|\Omega'|
\, ,
\end{align*}
which gives
\begin{align} \label{lem:lem_reg_seq_rho_n:estim_Lam3}
\int_{\Lambda_3} \frac{\eta}{\xi^2} |f(\Qexn) - f(\tilde{Q})| \dx x
\ &\lesssim \ \frac{\eta}{\xi^2} \frac{\lambda}{n} |\Omega'|^\frac12 \left(\Eex(\tilde{Q},B_{\frac1n}(\Omega')\cap\Omega)\right)^\frac12
+ \frac{\eta}{\xi^2} \frac{\lambda}{n^\frac32}|\Omega'|^\frac12
 \, .
\end{align}

Combining \eqref{lem:lem_reg_seq_rho_n:estim_Lam1},\eqref{lem:lem_reg_seq_rho_n:estim_Lam2}, \eqref{lem:lem_reg_seq_rho_n:estim_Lam3}
and choosing $\lambda = \sqrt{\xi}$ yields
\begin{equation}\label{lem:lem_reg_seq_rho_n:estim_Lam123}
\begin{split}
\int_{\Omega'} \frac{\eta}{\xi^2}f(\Qexn) \dx x
\ \leq \ \int_{\Omega'} \frac{\eta}{\xi^2}f(\tilde{Q}) \dx x
+ C\Bigg( \frac{\eta}{\xi^\frac32 n} &\Eex(\tilde{Q},B_{\frac1n}(\Omega')\cap\Omega) + \frac{\eta}{\xi^\frac32 n^\frac32} (\Eex(\tilde{Q},\Omega'))^\frac{1}{2}\\
&+ \frac{\eta}{\xi^\frac52 n^2} \Eex(\tilde{Q},B_{\frac1n}(\Omega')\cap\Omega) + \frac{\eta}{\xi^\frac52 n^3}\\
&+ \frac{\eta}{\xi^\frac32 n}|\Omega'|^\frac12 (\Eex(\tilde{Q},B_{\frac1n}(\Omega')\cap\Omega))^\frac12 + \frac{\eta}{\xi^\frac32 n^\frac32}|\Omega'|^\frac12 \Bigg) \, .
\end{split}
\end{equation}

We finish the proof of this lemma with the estimates on $g$.
Note that by Proposition 4.2 in \cite{ACS2021} we can bound the energy contribution of $g$ from the set where $|Q|$ is small by $\tfrac{\xi^2}{\eta}\Eex(\Qexn,\Omega')$.
By smoothness of $g$ away from $Q=0$ we can estimate
\begin{align*}
\int_{\Omega'} g(\Qexn) - g(\tilde{Q})\dx x
\ &\lesssim \ \frac{\xi^2}{\eta}\Eex(\Qexn,\Omega') + \Vert \nabla g\Vert_{L^\infty} \int_{\Omega'} |\Qexn - \tilde{Q}| \dx x \\
\ &\lesssim \ \frac{\xi^2}{\eta}\Eex(\Qexn,\Omega') + \left(\frac{|\Omega'|}{n^2}\Eex(\tilde{Q},B_{\frac1n}(\Omega')\cap\Omega) + \frac{|\Omega'|}{n^3}\right)^\frac12\, .
\end{align*}
Combining this with \eqref{lem:lem_reg_seq_rho_n:estim_Lam123} and \eqref{lem:lem_reg_seq_rho_n:estim_grad}, we subtract $C \frac{\xi^2}{\eta^2}\Eex(\Qexn,\Omega')$ from both sides and divide by $1-C\frac{\xi^2}{\eta^2}$ to get the estimate \eqref{lem:lem_reg_seq_rho_n:eq}.
Note, that in our regime for $\eta$ and $\xi$, the terms arising in the estimate for $g$ are smaller than the corresponding terms for $f$ in \eqref{lem:lem_reg_seq_rho_n:estim_Lam123} and hence do not display in \eqref{lem:lem_reg_seq_rho_n:eq}.
\end{proof}

With the results of Lemma~\ref{lem:lem_reg_seq_Pi_R}, Lemma~\ref{lem:lem_reg_seq_Phi_n} and Lemma~\ref{lem:lem_reg_seq_rho_n}, we can now turn to the proof of Proposition~\ref{prop:prop_reg_seq}.

\begin{proof}[Proof of Proposition~\ref{prop:prop_reg_seq}.]
The smoothness of the functions $\Qexn$ is clear by standard convolution arguments, since $\varrho$ is smooth. The bound on the gradient follows from the computation
\begin{align*}
|\nabla \Qexn (x)|
\ &\leq \ \Vert \nabla\varrho_n\Vert_{L^\infty} \int_{B_1(x)} |\Pi_R\Qex(\Phi_{n}(y))| \dx y 
\ \leq \ \frac{4}{3}\pi R\: n\, .
\end{align*}

Next, we prove the $C^1-$convergence on $\M$.
For $\omega\in\M$ it holds that
\begin{align*}
|\Qexn(\omega) - Q_b(\omega)|
\ &\leq \ \int_{B_{\frac1n}(0)} \Big|\Qex\Big(\omega-y-\frac1n\nu_{\omega-y}\Big) - Q_b(\omega)\Big| \varrho_n(y) \dx y \, .
\end{align*}
Note that $\Qex$ does not depend on $\eta,\xi$ here as it is uniquely defined by the extension $\overline{\nu}$.
Since $Q_b$ and $\overline{\nu}$ are continuous on a compact set, they are also uniformly continuous which implies $C^0-$convergence for $n\rightarrow 0$.
Analogously, 
\begin{align*}
|\nabla_\omega \Qexn(\omega) - \nabla_\omega Q_b(\omega)|
\ &\leq \ \int_{B_{\frac1n}(0)} \Big|(\nabla_\omega\Qex)(\omega-y-\nu_{\omega-y}) \Big(\id + \frac1n\nabla_\omega \overline{\nu} \Big) - \nabla_\omega Q_b(\omega) \Big| \varrho_n(y) \dx y \, ,
\end{align*}
and since $\nabla_\omega \overline{\nu}$ is bounded we can use uniform continuity of $\nabla_\omega Q_b$ to deduce $C^1-$convergence on $\M$.

Next, we show the $L^2-$convergence of $\Qexn-\Qex$ to $0$ as $\eta,\xi\rightarrow 0$ \dsg{and $n\rightarrow \infty$ with $\eta n\rightarrow \infty$}. 
Writing
\begin{align*}
\Vert \Qexn-\Qex\Vert_{L^2(\Omega)}
\ &\leq \ \Vert \Qexn-\Pi_R\circ\Qex\circ\Phi_n\Vert_{L^2(\Omega)}
 + \Vert \Pi_R\circ\Qex\circ\Phi_n-\Pi_R\circ\Qex\Vert_{L^2(\Omega)} \\
&\qquad+ \Vert \Pi_R\circ\Qex-\Qex\Vert_{L^2(\Omega)}
\end{align*}
and applying Lemma~\ref{lem:lem_reg_seq_Pi_R} to the third, Lemma~\ref{lem:lem_reg_seq_Phi_n} to the second and Lemma~\ref{lem:lem_reg_seq_rho_n} to the first $L^2-$difference on the right hand side, we see that $\Vert \Qexn-\Qex\Vert_{L^2(\Omega)}$ tends to zero if $n\eta\rightarrow\infty$ as $n\rightarrow\infty$ and $\xi,\eta\rightarrow 0$.

If we assume that $\eta\:\Eex(\Qex,\Omega)\leq C$ uniformly in $\eta,\xi$, then Lemma~\ref{lem:lem_reg_seq_Pi_R} implies that also $\eta\:\Eex(\Pi_R\circ\Qex,\Omega)\leq C$.
In addition, by Lemma~\ref{lem:lem_reg_seq_Phi_n}, it follows that $\eta\: \Eex(\Pi_R\circ\Qex\circ\Phi_n,\Omega)\leq C$ and that
\begin{align*}
\Eex(\Pi_R\circ \Qex\circ\Phi_n,\Omega')
\ &\leq \ 
\left( 1 + \frac{C_1}{\sqrt{n}} \right)\Eex(\Qex,B_{\frac{1}{n}}(\Omega')\cap\Omega)
\ + \ \frac{C_2}{\dsg{\eta^2} n}
\, .
\end{align*}
Combining this with Lemma~\ref{lem:lem_reg_seq_rho_n} and using that $\eta\: \Eex(\Pi_R\circ\Qex\circ\Phi_n,\Omega)\leq C$, we can find new constants $C_1,C_2>0$ such that
\begin{align*}
\Eex(\Qexn,\Omega')
\ \leq \ 
\Bigg( 1 &+ C_1\Bigg(\frac{1}{\sqrt{n}} + \frac{\xi^2}{\eta^2} \Bigg)\Bigg)\Eex( \Qex, B_{\frac{2}{n}}(\Omega')\cap\Omega) \\
\ &+ \ C_2\Bigg( 
  \frac{1}{\xi^\frac32 \eta n} 
+ \left( \frac{|\Omega'|}{\eta\xi^3 n^2} \right)^\frac12
\Bigg)
\end{align*}
for $n\sim \xi^{-a}$ for some $a>\frac32$. 
In this regime for $n$, the energy estimate is asymptotically sharp.
\end{proof}

} 

Having established these properties of $\Qexn$, we are able to identify the size and structure of the set where $\Qexn$ is close to being prolate uniaxial as stated in the next Lemma.

\begin{lemma}\label{lem:def_set_X_eps}
There exists a constants $C,\lipQ>0$ such that for all $\delta>0$, there exists a finite set $I\subset\Omega$ which satisfies
\begin{enumerate}
\item the following inclusions
\begin{align}\label{lem:def_set_X_eps:eq_covering}
U_\delta \ \subset \ \bigcup_{x\in I} B_{\frac{\delta}{\lipQ n}}(x)\ \subset \ U_{\delta/2} \, ,
\end{align}
where $U_\delta\defi \{x\in\Omega\sd \dist(\Qexn(x),\N)>\delta\}$,
\item and
\begin{align}\label{lem:def_set_X_eps:eq_size_X_eps}
\# I \ \leq \ C \: \frac{n^3}{\eta\: f_\mathrm{min}^\delta\: \delta^3}\left( \xi^2 + \frac{1}{n^2}\right)\, ,
\end{align}
where $f_\mathrm{min}^\delta = \min\{ f(Q)\sd \dist(Q,\N)\geq\delta/2 \}$.
\end{enumerate} 
\end{lemma}

\begin{proof}
Let $\delta>0$ and $x_0\in U_\delta$. 
Since $\Qexn$ is Lipschitz continuous (Proposition~\ref{prop:prop_reg_seq}), we can define $\lipQ := \frac{8}{3}\pi R>0$, i.e.\ $\lipQ n \geq 2\Vert\nabla \Qexn\Vert_{\infty}$. 
We deduce that for any $x\in B_{\frac{\delta}{\lipQ n}}(x_0)$ it holds
\begin{align*}
\dist(\Qexn(x),\N) 
\ \geq \ \dist(\Qexn(x_0),\N) - \Vert\nabla \Qexn\Vert_{\infty} \frac{\delta}{\lipQ n}
\ \geq \ \frac{\delta}{2}\, ,
\end{align*}
so that $x\in U_{\delta/2}$. 
From this, we get that
\begin{align*}
U_\delta \ \subset \ \bigcup_{x\in U_\delta} B_{\frac{\delta}{\lipQ n}}(x)\ \subset \ U_{\delta/2} \, .
\end{align*}
By Vitali covering theorem, we find a subset $I\subset U_\delta$ with the same property and $B_{\frac13\frac{\delta}{\lipQ n}}(x_i)\cap B_{\frac13\frac{\delta}{\lipQ n}}(x_j)=\emptyset$ for $i\neq j$ and $x_i,x_j\in I$. 
Furthermore, using Proposition~\ref{prop:prop_reg_seq} 
\begin{align*}
\frac{C\: \xi^2}{\eta} 
\ &\geq \ \int_\Omega f(\Qex) \dx x 
\ \geq \ \int_\Omega f(\Qexn) \dx x - \frac{C}{\eta}\left(\xi^2 + \frac{1}{n^2}\right) \\
\ &\geq \ \int_{U_{\delta/2}} f(\Qexn) \dx x - \frac{C}{\eta}\left(\xi^2 + \frac{1}{n^2}\right) 
\ \geq \ C \# I |B_{\frac{\delta}{\lipQ n}}| f_\mathrm{min}  - \frac{C}{\eta}\left(\xi^2 + \frac{1}{n^2}\right) \\
\ &\geq \ C \# I \frac{\delta^3 f_\mathrm{min}}{n^3} - \frac{C}{\eta}\left(\xi^2 + \frac{1}{n^2}\right)\, ,
\end{align*}
where we used that $f\geq f_\mathrm{min} >0$ on $U_{\delta/2}$. 
From this it follows that
\begin{align*}
\# I \ \leq \ C \frac{n^3}{\eta\: f_\mathrm{min} \delta^3}\left(\xi^2 + \frac{1}{n^2}\right)\, .
\end{align*}
\end{proof}

In \cite{ACS2021} a similar result was obtained using a regularization related to the energy and using the Euler-Lagrange equation to derive the Lipschitz continuity. 
This approach would also work in the new setting and one could obtain Lemma~\ref{lem:def_set_X_eps} with  $n=\xi^{-1}$. 
However, the new approach has \ds{the advantage to be \emph{local} and
provide us without effort the local estimates which lead to the lower
bounds in Section~\ref{subsec:compactness_alpha_0}}.

From \eqref{lem:def_set_X_eps:eq_size_X_eps} it follows that the volume of the union of balls in \eqref{lem:def_set_X_eps:eq_covering} converges to zero for $\eta,\xi\rightarrow 0$ and $n\sim \xi^{-a}$, \ds{$a\in [1,2)$}. 
The same holds true for the union of the surfaces of those balls. Note however that the sum of the diameters is not bounded and diverges. 
With the tool developed in \cite{Bethuel1994} and used in \cite{Canevari2015a,ACS2021} it would be possible to derive a bound, namely the sum of diameters can be shown to be bounded.

\subsection{Definition of the line singularity}
\label{subsec:def_line_S}

The goal of this section is to define a $1-$chain $\Sexn$ of finite length which satisfies the compactness properties announced in Theorem~\ref{thm:main}.
The necessary analysis has already been carried out in \cite{Canevari2019,Canevari2020} but for the reader's convenience we recall the important steps and results.

For the construction of $\Sexn$, we follow
 Section 3 in \cite{Canevari2019}. We recall that $\C$ is the cone of oblate uniaxial $Q-$tensors which can be seen as a smooth simplicial complex of codimension $2$ in $\Sym$. 
Evoking Thom's transversality theorem, one can assume that, for almost every $Y\in\Sym$, the function $\Qexn-Y$ is transverse to all cells of $\C$.
Subdividing the preimages of the cells under the map $\Qexn-Y$ if necessary, $(\Qexn-Y)^{-1}(\C)$ defines a smooth, simplicial, finite complex of codimension $2$ which we call $\Sexn$. 
Note that $\Sexn$ depends on the choice of $Y$.

The relevant estimates on $\Sexn$ needed to prove Theorem~\ref{thm:main} in Section~\ref{subsec:compactness_alpha_0} are formulated in Theorem C and Section 4 in \cite{Canevari2020}:

\begin{theorem}\label{thm:canevari-orlandi-lower_bound}
There exists a finite mass chain $S$ such that one can choose a subsequence $\Sexn$ (not relabelled) and $\alpha>0$ with
\begin{align*}
\mathbb{F}(\Sexn - S) \ \rightarrow \ 0 \qquad  \text{ for almost every } Y\in B_\alpha(0)\, .
\end{align*}
Furthermore, for any open subset $U\subset\mathbb{R}^3$ it holds
\begin{align*}
\liminf_{\substack{\xi,\eta\rightarrow 0 \\ \dsg{n\rightarrow\infty}}}  \eta\: \Eex(\Qexn,U\cap\Omega)
\ \geq \ \frac{\pi}{2}s_*^2 \beta\: \mathbb{M}(S\restr U) \, .
\end{align*}
\end{theorem}

In our situation, by construction of $\Qexn$ and for $Y\in B_\alpha(0)$ ($\alpha$ small enough) it holds that 
\begin{align*}
(\Qexn-Y)^{-1}(\C) \ \subset \ U_\delta \ \subset \ \bigcup_{x\in I} B_{\frac{\delta}{\lipQ n}}(x)\, .
\end{align*}
Hence $\supp(\Sexn)\subset \bigcup_{x\in I} B_{\frac{\delta}{\lipQ n}}(x)$ and in view of the lower bound in Theorem~\ref{thm:canevari-orlandi-lower_bound} we deduce that the energy coming from $\Sexn$ in $U$ is already contained in $U\cap\bigcup_{x\in I} B_{\frac{\delta}{\lipQ n}}(x)$.

\subsection{Construction of $T$ and estimates for $Q$ close to $\N$}
\label{subsec:constr_T_Q_uniax}

In this subsection we carry out the first steps to define the $2-$chain $T$. 
We start by defining 
\begin{align*}
\T \defi \{ Q\in\Sym\sd s>0\, , 0\leq r<1\, , n_3 = 0 \}\, ,
\end{align*}
where $r,s,\nn$ are defined as in \eqref{eq:q_decom}. 
From this we want to define $\Texn$ close to $\Qexn^{-1}(\T)$. 
As carried out in \cite{Canevari2019} and described in Subsection~\ref{subsec:def_line_S}, for almost every $Y$ the set $(\Qexn-Y)^{-1}(\T)$ is in fact a smooth finite complex.
In Lemma~\ref{lem:estim_T_uniax}, we show that in addition for a.e.\ $Y\in\Sym$, the definition 
\begin{align*}
\Texn \ \defi \ (\Qexn-Y)^{-1}(\T)
\end{align*}
allows to control the area in regions where $\Qexn$ is close to $\N$. 
Since both the constructions of $\Sexn$ and $\Texn$ are valid for a.e.\ $Y$, we can choose the same $Y$ and hence $\partial\Texn\cap\Omega = \Sexn$. 
In parts of $\Omega$ where $\Qexn$ is far from $\N$, e.g.\ close to $\Sexn$, we need to modify $\Texn$. This will be the subject of the next subsection.

At first, we recall the (intuitively obvious) result that $\T$ is well behaved close to $\N$ in the sense that the level sets $\{n_3=s\}$ for $s$ small have a similar $\mathcal{H}^4-$volume as $\T$.
This can be interpreted as control on the curvature of $\T\cap \N$.

\begin{lemma}\label{lem:T_level_set_=_s}
There exists $\alpha_0,\alpha_1,C>0$ such that for $Q\in\Sym$, $\dist(Q,\N)\leq\alpha_0$ and $\alpha\in (0,\alpha_1)$ it holds that
\begin{align*}
\lim_{s\rightarrow 0} \mathcal{H}^4(\{Y\in B_\alpha(0)\sd n_3(Q-Y) = s\}) 
\ &= \ \mathcal{H}^4(B_\alpha(Q)\cap \T)\, .
\end{align*}
\end{lemma}

In the smooth case this lemma follows as in \cite[Lemma 3]{Modica1987a}, however we give a proof here for completeness.

\begin{proof}
The parameter $\alpha_0$ needs to be small enough to avoid problems far from $\N$ due to the non-smoothness of $\T$ at the singularity $0\in\Sym$.
So we choose $0<\alpha_0<\frac{1}{8}\dist(0,\N)$.
To avoid dealing with the topology of the sets involved, we pick $0<\alpha_1<\frac{1}{8}\diam(\N)$. Hence, $B_\alpha(Q)\cap T$ is diffeomorphic to a $4-$dimensional ball. 

\dsg{
In the appendix, we check that $\partial_{n_3}Q \neq
0$ in a neighbourhood of each $Q\in \mathcal{T}$ close to $\N$, so that by the implicit function theorem, $n_3$ is a smooth function of $Q$ with $D_Q n_3(Q)\neq 0$ near $Q\in \mathcal{T}$. 
In addition, $n_3$ is uniformly bounded in $C^2$ in $\dist(\cdot,\N)\leq \alpha_0$ for $\alpha_0$ small enough (since then $r\approx 0$ and $s\approx s^*$).
It follows that
\begin{align*}
\mathcal{H}^4(\{Y\in B_\alpha(0)\sd n_3(Q-Y)=s\})
\ &= \ 
\mathcal{H}^4(\{Y \in B_\alpha(Q): n_3(Y)=s \} ) \\
\ &\rightarrow \
\mathcal{H}^4(\{Y \in B_\alpha(Q): n_3(Y)=0 \})
\ = \
\mathcal{H}^4(B_\alpha(Q)\cap\T) 
\end{align*}
as $s\to 0$.
} 
\end{proof}

For $\delta>0$, we introduce the set $A_\delta\subset\Omega$ in which $\Qexn$ is close to being prolate uniaxial with norm $\sqrt{\frac23}s_*$ as
\begin{align} \label{eq:A_delta}
A_\delta \ \defi \ \{ x\in\Omega \sd \dist(\Qexn(x),\mathcal{N})<\delta \} \, .
\end{align}
The next lemma shows that (in average) the $\mathcal{H}^2-$measure of $(\Qexn - Y)^{-1}(\T)$ that lies in $A_\delta$ is controlled by the energy.

\begin{lemma}\label{lem:estim_T_uniax}
There exists $\alpha_0,\delta_0>0$ such that for all $\alpha\in (0,\alpha_0)$, $\delta\in(0,\delta_0)$ one can find a constant $C>0$ such that 
\begin{align}\label{lem:estim_T_uniax_eq}
\int_{B_\alpha(0)} &\mathcal{H}^2(A_\delta\cap(\Qexn-Y)^{-1}(\mathcal{T}) \}) \dx Y 
\ \leq \ C \ \eta \ \Eex(\Qexn,A_\delta) \, .
\end{align}
\end{lemma}

\begin{proof}
Let $\alpha,\delta>0$ small enough such that for $Y\in B_\alpha(0)$, the map $Q\mapsto n_3(Q-Y)$ is smooth on $\{Q\in\Sym\sd\dist(Q,\N)<\delta\}$.
Let $A_\delta$ be defined as in \eqref{eq:A_delta}.
In order for the map $x\mapsto n_3(\Qexn(x)-Y)$ to be well defined, we need to restrict ourselves to a simply connected subset of $A_\delta$. 
For this, take $x_0\in A_\delta$ and $r>0$ such that $A_\delta\cap B_r(x_0)$ is simply connected. 
We carry out the analysis on $A_\delta\cap B_r(x_0)$, noting that we can cover $A_\delta$ by such balls to find the estimate \eqref{lem:estim_T_uniax_eq}.
With $x_0\in A_\delta$ and $r>0$ fixed as described, we can calculate
\begin{align*}
\int_{B_\alpha(0)} &\mathcal{H}^2(B_r(x_0)\cap A_\delta\cap(\Qexn(x)-Y)^{-1}(\mathcal{T}) \}) \dx Y \\
\ &= \ \int_{B_\alpha(0)} |D\chi_{\{x\in \Omega | n_3(\Qexn(x)-Y)>0\}}|(B_r(x_0)\cap A_\delta) \dx Y \\
\ &\leq \ \liminf_{\epsilon\rightarrow 0} \int_{B_\alpha(0)} \int_{B_r(x_0)\cap A_\delta} |\nabla_x (h_\epsilon\circ n_3 \circ (\Qexn-Y))(x)| \dx x \dx Y \\
\ &= \ \liminf_{\epsilon\rightarrow 0} \int_{B_\alpha(0)} \int_{B_r(x_0)\cap A_\delta} |h_\epsilon'(n_3(\Qexn(x)-Y)) \nabla_Q n_3(\Qexn(x)-Y) : \nabla_x Q(x)| \dx x \dx Y \, ,
\end{align*}
where $h_\epsilon\in C^1(\mathbb{R},[0,1])$ is an approximation of the Heaviside function, i.e. $h_\epsilon(x)=0$ for $x\leq 0$, $h_\epsilon(x)=1$ for $x\geq \epsilon$ and $h_\epsilon'>0$ on $(0,\epsilon)$. The above inequality is then just the lower semi continuity of the total variation. 
With the identity $h_\epsilon'(n_3(\Qexn(x)-Y)) \nabla_Q n_3(\Qexn(x)-Y) = -\nabla_Y (h_\epsilon\circ n_3\circ(\Qexn(x)-Y))$ and the Fubini theorem we can rewrite
\begin{align*}
\int_{B_\alpha(0)} &\int_{B_r(x_0)\cap A_\delta} |h_\epsilon'(n_3(\Qexn(x)-Y)) \nabla_Q n_3(\Qexn(x)-Y) : \nabla_x \Qexn(x)| \dx x \dx Y \\
\ &\leq \ \int_{B_r(x_0)\cap A_\delta} |\nabla \Qexn| \int_{B_\alpha(0)} |\nabla_Y (h_\epsilon\circ n_3\circ(\Qexn-Y))| \dx Y \dx x  \\
\ &= \ \int_{B_r(x_0)\cap A_\delta} |\nabla \Qexn| \int_{0}^1 \mathcal{H}^4(\{Y\in B_\alpha(0)\sd h_\epsilon(n_3(\Qexn(x)-Y)) = s\}) \dx s \dx x\\
\ &= \ \int_{B_r(x_0)\cap A_\delta} |\nabla \Qexn| \int_{0}^1 \mathcal{H}^4(\{Y\in B_\alpha(0)\sd n_3(\Qexn(x)-Y) = h_\epsilon^{-1}(s)\}) \dx s \dx x \, ,
\end{align*}
where we also used the coarea formula. By Lemma~\ref{lem:T_level_set_=_s} in the liminf $\epsilon\rightarrow 0$ this equals
\begin{align*}
\liminf_{\epsilon\rightarrow 0} &\int_{B_r(x_0)\cap A_\delta} |\nabla \Qexn| \int_{0}^1 \mathcal{H}^4(\{ Y\in B_\alpha(0)\sd  n_3(\Qexn(x)-Y) = h_\epsilon^{-1}(s)\}) \dx s \dx x \\
\ &= \ \int_{B_r(x_0)\cap A_\delta} |\nabla \Qexn| \: \mathcal{H}^4(B_\alpha(\Qexn)\cap\T)\dx x
\end{align*}
by translation invariance of $\mathcal{H}^4$. Applying the inequality $2ab\leq a^2+b^2$ we get
\begin{align*}
\int_{B_r(x_0)\cap A_\delta} &|\nabla \Qexn| \: \mathcal{H}^4(B_\alpha(\Qexn)\cap\T)\dx x \\
\ &\leq \ \int_{B_r(x_0)\cap A_\delta} \frac{\eta}{2}|\nabla \Qexn|^2 \ + \ \frac{1}{2\eta}\mathcal{H}^4(B_\alpha(\Qexn)\cap\T)^2\dx x\, .
\end{align*}
The Dirichlet term appears in the energy, so it remains to estimate $\mathcal{H}^4(B_\alpha(\Qexn)\cap\T)^2$ in terms of $g(\Qexn)$. We first note that $\T\cap B_\alpha(\Qexn(x))=\emptyset$ if $\dist(\Qexn(x),\T)>\alpha$
and since $\dist(\Qexn,\N)<\delta$ we have $\mathcal{H}^4(B_\alpha(\Qexn)\cap\T)\leq C_\delta \alpha^4$ by Proposition~\ref{prop:app_vol_Ba_in_T_N}. Hence, we get
\begin{align*}
\int_{B_r(x_0)\cap A_\delta} \mathcal{H}^4(B_\alpha(\Qexn)\cap \mathcal{T})^2 dx \ 
\leq \ (C_{\delta}\alpha^4)^2 |B_r(x_0)\cap A_\delta\cap \{x\in\Omega \sd \dist(\Qexn(x),\mathcal{T})<\alpha\}|\, .
\end{align*}
For $x\in A_\delta\cap \{x\in\Omega \sd \dist(\Qexn(x),\T)<\alpha\}$ we can estimate
\begin{align*}
g(\Qexn(x)) 
\ &\geq \ g(\mathcal{R}(\Qexn(x))) - C_g \dist(\Qexn(x),\N) 
\ \geq \ \sqrt{\frac{3}{2}}(1 - n_3^2(\Qexn(x))) - C_g \delta \\
\ &\geq \ \sqrt{\frac{3}{2}}(1 - C_\mathcal{T} \alpha) - C_g \delta
\ \geq \ G > 0
\end{align*}
for $\alpha,\delta\ll 1$ small enough. Hence, 
\begin{align*}
G |B_r(x_0)\cap A_\delta\cap \{x\in\Omega \sd \dist(\Qexn(x),\T)<\alpha\}| \ \leq \ \int_{B_r(x_0)\cap A_\delta} g(\Qexn) \dx x\, .
\end{align*}
\end{proof}

We remark that although Lemma~\ref{lem:estim_T_uniax} controls the size for a.e.\ \emph{fixed} $Y\in B_\alpha(0)$, the estimate degenerates with $\alpha$.
Hence it does not provide a uniform bound in $\alpha$ allowing to pass to the limit $Y\rightarrow 0$.
A bound independent of $\alpha$ will be derived in the section on the lower bound.

\subsection{Estimates near singularities}
\label{subsec:near_sing}

At points $x\in\Omega$ where $\dist(\Qexn(x),\N)>\delta$, the estimates we derived in the previous subsection are no longer available and we need new tools to bound the mass of $\Texn$. We are concerned with two different cases: 
The first case is the one of $x\in\Texn$ far from the boundary $\Sexn$. We can simply ``cut out'' those pieces and replace them by parts of surfaces of spheres which are controlled in mass. This will be made precise using Lemma~\ref{lem:def_set_X_eps}. 
The second case is more challenging. We will modify $\Texn$ close to the boundary $\Sexn$ by using a construction similar to the one used in the deformation theorem (see Lemma~\ref{lem:sing_in_T_exterior}). This will allow us to express the mass of the modified $2-$chain in terms of the surface of cubes and Lemma~\ref{lem:def_set_X_eps} permits us to control the number of such cubes.

\begin{lemma}[Deformation in the interior]\label{lem:sing_in_T_interior}
Let $I^\mathrm{int}\subset I$ be the subset of points $x_0\in I$ such that $\dist(x_0,\Sexn)>\frac{\delta}{\lipQ n}$ and $\dist(x_0,\Texn)<\frac{\delta}{\lipQ n}$. 
Then, there exists a flat $2-$chain $\widetilde{T^\mathrm{int}}$ with values in $\pi_1(\N)$ and support in $B^\mathrm{int}\defi \bigcup_{x\in I^\mathrm{int}} \overline{B_{\frac{\delta}{\lipQ n}}(x)}$ such that
\begin{enumerate}
\item $\partial\widetilde{T^\mathrm{int}} = \partial(\Texn\restr B^\mathrm{int})$,
\item and $\MM(\widetilde{T^\mathrm{int}})\lesssim \frac{n}{\eta}\left(\xi^2+\frac{1}{n^2}\right)$.
\end{enumerate}
\end{lemma}

\begin{proof}
Since $B^\mathrm{int}\cap\supp(\Texn)\neq \emptyset$ and $B^\mathrm{int}\cap\supp(\Sexn)=\emptyset$ we know that $\emptyset\neq \partial(\Texn\restr B^\mathrm{int})\subset\partial B^\mathrm{int}$. Furthermore, since $\partial^2=0$ it follows that $\partial(\Texn\restr B^\mathrm{int})$ consists of closed curves and divides $\partial B^\mathrm{int}$ into domains. Let $D$ be the set of these domains. Now pick any subset $D'\subset D$ such that $\partial\left(\bigcup_{U\in D'} U\right) = \partial(\Texn\restr B)$.
We define $\widetilde{T^\mathrm{int}}\defi\sum_{U\in D'} [U]$. Then, by definition $\Texn\restr B^\mathrm{int}$ and $\widetilde{T^\mathrm{int}}$ have the same boundary and since $\widetilde{T^\mathrm{int}}\subset\partial B^\mathrm{int}$ we also have 
\begin{align*}
\MM(\widetilde{T^\mathrm{int}}) 
\ &\leq \ \MM(\partial B^\mathrm{int})  
\ \leq \ \sum_{x\in I^\mathrm{int}} \MM(\partial B_{\frac{\delta}{\lipQ n}}) 
\ \leq \ \# I \: 4\pi\:\frac{\delta^2}{\lipQ^2 n^2} 
\ \lesssim \ \frac{n}{\eta}\left(\xi^2 + \frac{1}{n^2}\right) \, .
\end{align*}
\end{proof}


At the boundary we cannot remove a disk without the risk of creating new boundary which might not be controlled, so another method has to be used. The idea is the following: Take a cube $K$ of size $\frac{\delta}{n}$ which contains a part of the singular line $\Sexn$ and intersects with $\Texn$. We then modify (deform) the ``surface'' connecting $\Texn\cap\partial K$ and $\Sexn\cap K$ by pushing it onto a part of $\partial K$ (see also Figure~\ref{fig:deformation_near_bdry}). The result is a modified $\Texn$ with the same boundary as before and the surface inside the cube is controlled by the surface area of $K$ and the length of the singular line.
We point out that this procedure and its proof are closely related to the deformation theorem (for flat chains) (see \cite{White1999}, Chapter 5 in \cite{Federer1960}, Theorem 7.3 in \cite{Fleming1966} and Chapter 4.2 in \cite{Federer1996}) but differs in some details so that we give a full proof here. 

\begin{lemma}[Deformation close to the boundary]\label{lem:sing_in_T_exterior}
Let $I^\mathrm{bdry}\subset I$ be the subset of points $x_0\in I$ such that $\dist(x_0,\Sexn)<\frac{\delta}{\lipQ n}$ and $\dist(x_0,\Texn)<\frac{\delta}{\lipQ n}$. 
Then there exists a flat $2-$chain $\widetilde{T^\mathrm{bdry}}$ with values in $\pi_1(\N)$ and support in a finite union of cubes of side length $\delta/n$ called $B^\mathrm{bdry}$ verifying $\bigcup_{x\in I^\mathrm{bdry}} \overline{B_{\frac{\delta}{\lipQ n}}(x)}\subset B^\mathrm{bdry}$ such that
\begin{enumerate}
\item $|B^\mathrm{bdry}|\lesssim \frac{1}{\eta}(\xi^2 + \frac{1}{n^2})$,
\item $\partial\widetilde{T^\mathrm{bdry}} = \partial(\Texn\restr B^\mathrm{bdry})$,
\item and  
\begin{align}\label{lem:sing_in_T_exterior:control_M(T)_eq}
\MM(\widetilde{T^\mathrm{bdry}})  
\ \lesssim \ \frac{n}{\eta}\left(\xi^2 + \frac{1}{n^2}\right)\, .
\end{align}
\end{enumerate}
\end{lemma}

\begin{proof}
For the sake of readability we drop the dependences on $\xi,\eta,n$ in the notation of this proof and simply write $\tilde{T}$ instead of $\widetilde{T^\mathrm{bdry}}$.
Covering $S$ with a cubic grid of size $h=\frac{\delta}{n}$ such that $S$ is in a general position, we can assume that the center $x_K$ of all cubes $K$ that contain parts of $S$ does not intersects $S$ or $T$, i.e.\ $x_K\notin \supp(T),\supp(S)$.
Indeed, this is possible \dsg{since} $S$ intersects only a finite number of cubes according to Lemma~\ref{lem:def_set_X_eps}.
Let $G$ be the set of those cubes and $X$ the set of its centers.

\textit{Step 1 (Construction and properties of the retraction map).}
Let $K\in G$ be a cube and let $x_K\in X$ be its center. 
Let $P$ be the central projection onto $\partial K$ originating in $x_K$. 
We define a homotopy $\Phi:[0,1]\times(K\setminus\{x_K\})\rightarrow K$ between the identity on $K$ and $P$ by  simply taking $\Phi(t,x) = (1-t)x + t P x$. 
Note that by definition this homotopy is relative to $\partial K$, i.e.\ $\Phi(t,x)=x$ for all $t\in[0,1]$ and $x\in\partial K$. Furthermore, for all $x\in K\setminus\{x_k\}$ and $t\in [0,1]$ it holds
\begin{align} \label{lem:sing_in_T_exterior:eq_estim_Phi_Lip}
\dist(\Phi(t,x),x_K) \ \geq \ \dist(x,x_K)\, .
\end{align}
Since $|\partial_t \Phi(t,x)| = |-x+Px|\leq \sqrt{3} h$ and by \eqref{lem:sing_in_T_exterior:eq_estim_Phi_Lip} we deduce that $\Phi$ is locally Lipschitz continuous and $\Lip(\Phi(t,x))\leq C\: h\: \dist(x,x_K)^{-1}$.
Since $\Phi$ is relative to $\partial K$ we can glue together all those functions defined on the cubes $K\in G$ with the identity on cubes $K\notin G$ to get a function $\Phi$ defined everywhere in $\mathbb{R}^3\setminus X$.

\textit{Step 2 (Intermediate estimate).}
In this step we want to show that if we allow for a small translation of the chain $S$, then the mass of $\Phi_\#([0,1]\times S)$ can be bounded by $\mathbb{M}(S)$ times the size of the cube $h$, up to a constant.

Applying Corollary 2.10.11 in \cite{Federer1996} (or  Section 2.7 in \cite{Federer1960}) we get as in \cite[Lemma 2.1]{White1999}
\begin{align*}
\mathbb{M}(\Phi_\#([0,1]\times S))
\ &\leq \ \Vert \id - P \Vert_\infty \int_{\mathbb{R}^3} \sup_{t\in [0,1]}\Lip(\Phi(t,x)) \dx\mu_S(x) \\
\ &\leq \ C\:h^2\: \int_{\mathbb{R}^3}  \dist(x,X)^{-1} \dx\mu_S(x)\, .
\end{align*}
Taking the mean over translations \dsg{$\tau_{hy}$} for a vector $y\in [0,1]^3$, we arrive at
\begin{align*}
\int_{[0,1]^3} \mathbb{M}(\Phi_\#([0,1]\times (\tau_{hy} S))\dx y
\ &= \ C\: h^2 \int_{[0,1]^3} \int_{\mathbb{R}^3}  \dist(x,X)^{-1} \dx\mu_{\tau_{hy} S}(x) \dx y \\
\ &= \ C\: h^2 \int_{[0,1]^3} \int_{\mathbb{R}^3}  \dist(x+hy,X)^{-1} \dx\mu_{S}(x) \dx y \\
\ &= \ C\: h^2 \int_{\mathbb{R}^3}  \int_{[0,1]^3} \dist(x+hy,X)^{-1}  \dx y \dx\mu_{S}(x) \\
\ &\leq \ C\: h \int_{\mathbb{R}^3}  \dx\mu_{S}(x) \\
\ &= \ C\: h\: \mathbb{M}(S)\, . 
\end{align*}
Hence, we can assume that $S$ is in a position such that
\begin{align}\label{lem:sing_in_T_exterior:eq_estim_Phi_S}
\mathbb{M}(\Phi_\#([0,1]\times S))
\ & \leq \ C\: h\: \mathbb{M}(S)\, .
\end{align}

\textit{Step 3 (Definition of $\tilde{T}$).}
We define 
\begin{align*}
\tilde{T} \defi \partial(\Phi_\#([0,1]\times T)) - T\, .
\end{align*}
Considering a cube $K\in G$, one can think of this construction as the boundary of the three dimensional object created by filling the space between $T$ and its projection onto $\partial K$ according to Step 1 and then removing the original part $T$. Another but equivalent point of view is  to take $\tilde{T}$ as all the points along the path created by projecting $T\restr\partial K$, $S$ together with the projection $P_\#(T)$, see also Figure~\ref{fig:deformation_near_bdry}. Indeed, one can calculate for $K\in G$
\begin{align*}
\partial(\Phi_\#([0,1]\times (T\restr K))) 
 &= \Phi_\#((\partial[0,1])\times (T\restr K)) + \Phi_\#([0,1]\times (\partial T)\restr K) + \Phi_\#([0,1]\times T\restr(\partial K)) \\
 &=  P_\#(T\restr K) - (\id_K)_\#(T) + \Phi_\#([0,1]\times (S\restr K)) + \Phi_\#([0,1]\times T\restr(\partial K)) \, .
\end{align*}
Thus, we have the formula 
\begin{align*}
\tilde{T}\restr K
\ &= \ P_\#(T\restr K) \ + \ \Phi_\#([0,1]\times (S\restr K)) \ + \ \Phi_\#([0,1]\times T\restr(\partial K)) \, .
\end{align*}
Since $P_\#(T\restr K) + \Phi_\#([0,1]\times T\restr (\partial K))\subset \partial K$ from which we derive the bound on the mass of $\tilde{T}$
\begin{align} \label{lem:sing_in_T_exterior:M(T)_on_K}
\mathbb{M}(\tilde{T}\restr K)
\ &\leq \ \mathbb{M}(\partial K) \ + \ \mathbb{M}(\Phi_\#([0,1]\times (S\restr K))) 
\ \leq \ 6\: h^2 \ + \  C\: h\:\mathbb{M}(S\restr K) \, ,
\end{align}
where we also used the estimate \eqref{lem:sing_in_T_exterior:eq_estim_Phi_S} on $K$ of Step 2.
On all cubes $K\notin G$, $\tilde{T}\restr K=0$, so that we find $\supp(\tilde{T})\subset \bigcup_{K\in G}\overline{K}$. 
Defining $B^\mathrm{bdry}\defi\bigcup\{ \overline{K}\sd K\text{ is cube of the grid}, \: \exists x\in I^\mathrm{bdry}\text{ with } K\cap B_{\frac{\delta}{\lipQ n}}(x)\neq\emptyset \}$, it is clear that $\tilde{T}$ is supported in $B^\mathrm{bdry}$ since $\bigcup_{K\in G}\overline{K}\subset B^\mathrm{bdry}$.
Furthermore, by definition of $B^\mathrm{bdry}$, we have the claimed
inclusion $\bigcup_{x\in I^\mathrm{bdry}} \overline{B_{\frac{\delta}{\lipQ n}}(x)}\subset B^\mathrm{bdry}$.
The measure of $B^\mathrm{bdry}$ can easily be estimated since it is formed by cubes covering $\bigcup_{x\in I^\mathrm{bdry}} B_{\frac{\delta}{\lipQ n}}(x)$, the cubes having the same length scale $\frac{\delta}{n}$ as the balls. 
Therefore, up to a constant only depending on the space dimension and $\lipQ$, \eqref{lem:def_set_X_eps:eq_size_X_eps} implies that $|B^\mathrm{bdry}|\lesssim h^3\frac{n^3}{\eta \delta^3}(\xi^2+\frac{1}{n^2}) = \frac{1}{\eta}(\xi^2+\frac{1}{n^2})$.
Since $\partial\circ\partial=0$, the boundary of $\tilde{T}$ coincides with $\partial T$.
Since all calculations in Step 3 were local and $\Phi$ is relative to the boundaries of the cubes, \eqref{lem:sing_in_T_exterior:control_M(T)_eq} follows from summing up \eqref{lem:sing_in_T_exterior:M(T)_on_K} over all cubes $K\in G$.
\end{proof}

\begin{figure}[H]
\centering
\includegraphics[scale=1]{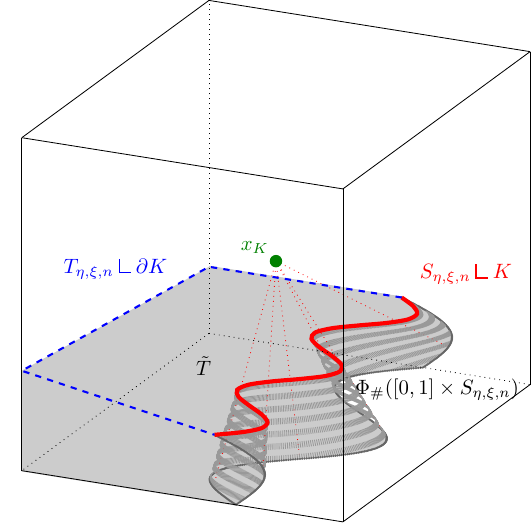}
\caption{Construction near the boundary inside a cube: The newly created area (grey) is controlled by the surface of the cube and the length of the line singularity (red).}
\label{fig:deformation_near_bdry}
\end{figure}

As a direct consequence of Lemma~\ref{lem:estim_T_uniax}, Lemma~\ref{lem:sing_in_T_interior} and Lemma~\ref{lem:sing_in_T_exterior} we have the following corollary:

\begin{corollary} \label{cor:control_M(T)}
There exists a flat $2-$chain $\widetilde{\Texn}$ with values in $\pi_1(\N)$ such that 
\begin{enumerate}
\item $\partial\widetilde{\Texn} = \Sexn$,
\item for all $x_0\in\Omega$ and $R>0$ 
\begin{align}\label{cor:control_M(T)_eq}
\MM(\widetilde{\Texn}\restr B_R(x_0)) 
\ \lesssim \  \eta\:\Eex(\Qexn,B_R(x_0)) + \frac{n}{\eta}\left(\xi^2 + \frac{1}{n^2}\right)\, ,
\end{align}
\item and
\begin{align}\label{cor:control_M(T)_cpt}
\FF(\widetilde{\Texn} - \Texn) \ \lesssim \  \frac{n}{\eta}\left(\xi^2 + \frac{1}{n^2}\right)\, .
\end{align}
\end{enumerate}
\end{corollary}

\begin{proof}
Starting from $\Texn$ and the estimate in Lemma~\ref{lem:estim_T_uniax}, we can modify $\Texn$ according to Lemma~\ref{lem:sing_in_T_interior} and Lemma~\ref{lem:sing_in_T_exterior} in the region $B^\mathrm{int}\cup B^\mathrm{bdry}$ to obtain $\widetilde{\Texn}$ without changing the boundary $\Sexn$ by setting
\begin{align*}
\widetilde{\Texn} 
\ \defi \ \widetilde{T^\mathrm{int}}\restr B^\mathrm{int} \ + \  \widetilde{T^\mathrm{bdry}}\restr B^\mathrm{bdry} \ + \ \Texn\restr((B^\mathrm{int}\cup B^\mathrm{bdry})^c)\, .
\end{align*}
Estimate \eqref{cor:control_M(T)_eq} is a direct consequence of the three aforementioned lemmata.
Finally, by construction $\Texn-\widetilde{\Texn}$ is supported in $B^\mathrm{int}\cup B^\mathrm{bdry}$ and $\partial(\Texn-\widetilde{\Texn})=0$. 
Hence, $\FF(\Texn-\widetilde{\Texn})\leq |B^\mathrm{int}\cup B^\mathrm{bdry}|$, from which \eqref{cor:control_M(T)_cpt} follows for $n$ large and $\eta,\xi$ small enough.
\end{proof}

In the following analysis we only work with $\widetilde{\Texn}$.
In order to improve readability, we drop the tilde in our notation from now on.

\subsection{Proof of compactness for fixed $Y$}
\label{subsec:compact_fix_Y}

Let $B\subset \Omega$ open, bounded and choose $n:=\xi^{-a}$ \ds{for $a\in(\tfrac{3}{2},2)$}. 
Then, by Lemma~\ref{lem:estim_T_uniax} and Corollary~\ref{cor:calT_manifold_bdry}, we deduce that for $\alpha>0$ and $\xi,\eta>0$, there exist $Y_{\eta,\xi}\in B_\alpha(0)\subset\Sym$ such that our construction yields a flat chain $\Texn\in\flatch^2$ such that $\partial\Texn = \Sexn + \Gamma_n$ and
\ds{
\begin{align*}
\MM(\Texn\restr B)
\ &\leq \ C\left( \eta\:\Eex(\Qexn,B) + \frac{\xi^{2-a}}{\eta}\right) \\
\ &\leq\ C\left( \eta\:\Eex(\Qex,B_{2\xi^{a}}(B)\cap\Omega) + \frac{\xi^{2-a}}{\eta} + \frac{\xi^{a-\tfrac32}}{\eta}(1+|B|^\frac12) \right) \, ,
\end{align*}
} where we also used \eqref{prop:prop_reg_seq_E_bound} of Proposition~\ref{prop:prop_reg_seq}.
In particular the energy bound \eqref{thm:global_Eex_bound} implies that $\MM(\Texn\restr B)$ is bounded. 
Applying a compactness theorem for flat chains as stated in the preliminary part (\hspace{1sp}\cite[Cor. 7.5]{Fleming1966}), there exists a subsequence (which we do not relabel) and a flat chain $T\in\flatch^2$ with support in $\overline{\Omega}$ such that $\FF((\Texn - T)\restr B) \rightarrow 0$ for $\xi,\eta\rightarrow 0$. 
Extracting another subsequence if necessary we can assume that $Y_{\eta,\xi}\rightarrow Y\in \overline{B_\alpha(0)}$ for $\eta,\xi\rightarrow 0$.
We note that the $T$ constructed here depends on $Y$ (and $\alpha$).
In order to keep our notation simple, we only write this dependence explicitly when necessary, i.e.\ when we pass to the limit $\Vert Y\Vert,\alpha\rightarrow 0$ in Subsection~\ref{subsec:compactness_alpha_0}.
Since the boundary operator $\partial$ is continuous we conclude with Theorem~\ref{thm:canevari-orlandi-lower_bound} that $\partial T = S + \Gamma_Y$, where \ds{$\Gamma_Y=\{ \omega\in\M\sd n_3(s_*(\nu\otimes\nu - \frac13\id)-Y)=0 \}$}. 
The finite mass of $T$ and $S$ immediately implies rectifiability \cite[Thm 10.1]{Fleming1966}.


\section{Lower bound}
\label{sec:lower_bound}

This section is devoted to the $\Gamma-$liminf inequality of Theorem~\ref{thm:main}. 
The proof necessary to deduce the line energy has already been given in \cite{Canevari2020}, so that we will only state the result for completeness (Proposition~\ref{prop:lower_bound_-_line}). 
The energy contributions of $T$ far from $\M$ are to be derived in Subsection~\ref{subsec:lower_bound_-_blowup}.
In the remaining, we are concerned with the energy of $T$ and $F$ close to, resp. on, $\M$. 


The precise cost of a singular line in our setting has been derived first in \cite{Canevari2015} based on ideas in \cite{Sandier1998, Jerrard1999}.
In our case, the result reads as follows.

\begin{proposition}\label{prop:lower_bound_-_line}
Let $B\subset\overline{\Omega}$ be a bounded open set and $U_\eta\defi\{ x\in\Omega\sd \dist(x,\Sexn)\leq \sqrt{\eta} \}$. Then
\begin{align}\label{prop:lower_bound_-_line_eq}
\liminf_{\eta,\xi\rightarrow 0} \eta\: \Eex(\Qexn,U_\eta\cap B)
\ &\geq \ \frac{\pi}{2}s_*^2\beta\: \MM(S\restr B)\, .
\end{align}
\end{proposition}

\begin{proof}
See Theorem C and Proposition 4.1 in \cite{Canevari2020} for a proof of the version we used here.
\end{proof}

To derive the exact minimal energy for the lower bound related to $T$, we introduce the following auxiliary problem:
\begin{equation} \label{def:one_dim_aux_problem}
I(r_1,r_2,a,b) \ := \ \inf_{\underset{n_3(r_1)=a,\: n_3(r_2)=b}{n_3\in H^1([r_1,r_2],[-1,1])}} \int_{r_1}^{r_2} \frac{s_*^2 |n_3'|^2}{1-n_3^2} + c_*^2(1-n_3^2) \dx r
\end{equation}
for $0\leq r_1\leq r_2\leq \infty$, $a,b\in [-1,1]$.
It is one dimensional and only takes into account the derivative along the integration path. 
Problem \eqref{def:one_dim_aux_problem} is equivalent to minimizing $\int \left(\frac{1}{2}|\partial_r Q|^2 + g(Q)\right) \dx r$ subject to a $\N-$valued function $Q$ and fitting boundary conditions. 
This reflects that by Lemma~\ref{lem:def_set_X_eps}, the regions where $\Qexn$ is far from $\N$ are small. 
Indeed, if $Q(r)=s_*(\nn(r)\otimes\nn(r)-\frac13\id)$ for a $\mathbb{S}^2-$valued function $\nn$, then $\partial_r Q = s_*((\partial_r\nn)\otimes\nn + \nn\otimes(\partial_r\nn))$ and hence 
$|\partial_r Q|^2 = 2 s_*^2 |\partial_r \nn|^2$ since $|\nn|^2=1$ and therefore $2(\partial_r\nn)\cdot\nn = \partial_r|\nn|^2=0$.
Using again that $\nn\in\mathbb{S}^2$, we can write $\nn=(\pm\sqrt{1-n_3}\tilde{n}_1,\pm\sqrt{1-n_3^2}\tilde{n}_2, n_3)$, where $\tilde{n}=(\tilde{n}_1,\tilde{n}_2)$ is a $\mathbb{S}^1-$valued function.
One can then easily calculate that $|\partial_r \nn|^2 \geq |\sqrt{1-n_3^2}2n_3\partial_r n_3|^2 + |\partial_r n_3|^2 = |\partial_r n_3|^2/(1-n_3^2)$ (with equality if $\tilde{n}$ is constant), which is the first term in \eqref{def:one_dim_aux_problem}.
For the second term in \eqref{def:one_dim_aux_problem} we note that by \eqref{def:g:uniax} it holds that  $g(s_*(\nn\otimes\nn-\frac13\id))=c_*^2(1-n_3^2)$.
The functional in \eqref{def:one_dim_aux_problem} has been introduced in \cite{Alama2017} and studied in \cite[Lemma 3.4]{Alama2017}, \cite[Lemma 4.17]{ACS2021}, which show the following lemma:

\begin{lemma} \label{lem:radial_turning}
Let $0\leq r_1\leq r_2 \leq \infty$. Then,
\begin{enumerate}
\item $I(r_1,r_2,-1,1)\geq 4s_*c_*$.
\item Let $\theta\in [0,\pi]$. Then the minimizer $\nnopt$ of $I(0,\infty,\cos(\theta),1)$ is explicitly given by
\begin{align}\label{lem:radial_turning:nn3_opt}
\nnopt(r,\theta) = \frac{A(\theta)-\exp(-2c_*/s_* r)}{A(\theta)+\exp(-2c_*/s_* r)}\, , \quad A(\theta)=\frac{1+\cos(\theta)}{1-\cos(\theta)}
\end{align}
and 
\begin{align}\label{lem:radial_turning:cos}
I(0,\infty,\cos(\theta),\pm 1) = 2s_*c_*(1\mp \cos(\theta))\, .
\end{align}
\end{enumerate}
\end{lemma}

In this lemma, we use that $g$ reduces to $c_*^2(1-n_3^2)$ for $Q$ in $\N$, as demanded in \eqref{def:g:uniax}. 
However, as pointed out in Remark 4.18 in \cite{ACS2021}, this is not necessary. 

During the blow up procedure in the next subsection, we want to quantify the energy necessary for a $\Qexn$ close to $\N$ to pass from $n_3(\Qexn)\approx \pm 1$ to $n_3(\Qexn-Y) = 0$, i.e.\ to intersect $\Texn$.
Since problem \eqref{def:one_dim_aux_problem} does not take into account the perturbation made by subtracting $Y\in B_\alpha(0)$ from $\Qexn$, we introduce for $\alpha>0$ small enough
\begin{align} \label{def:one_dim_aux_problem_Y}
I_\alpha(r_1,r_2,a,b) \defi \inf\{ I(r_1,r_2,a,\pm n_3(Q))\sd Q\in\Sym,\, n_3(Q-Y)=\pm b,\, Y\in B_\alpha(0) \}\, .
\end{align}
Since $n_3(Q)$ and $n_3(Q-Y)$ are only defined up to a sign, it is necessary to define $I_\alpha$ using the infimum not only over $Y$ but also the choice of sign. 
This leads to the slightly counter-intuitive situation that e.g.\ $I_\alpha(r_1,r_2,a,-a)=0$ for all $a\in [-1,1]$.
As a consequence, we only have convergence of $I_\alpha(r_1,r_2,a,b)\rightarrow I(r_1,r_2,a,b)$ for $\alpha\rightarrow 0$ if $ab\geq 0$.
In what follows, we will only be concerned with the case $b=0$ as this corresponds to a point on $\Texn$, and hence we have convergence of $I_\alpha(r_1,r_2,a,0)$ to $I(r_1,r_2,a,0)$ for all $a\in [-1,1]$ for $\alpha\rightarrow 0$.

The knowledge about the optimal profile in \eqref{lem:radial_turning:cos} is also used in the construction of the upper bound, in particular the fact that $|\nnopt|-1$ and all derivatives of $\nnopt$ decay fast enough (here exponentially) as $r\rightarrow\infty$.
The result that for minimizers of \eqref{def:one_dim_aux_problem}, $n_3^2$ approaches $1$ exponentially fast is complemented by the next lemma.
It states that for a bounded energy configuration on a line, $n_3^2$ cannot always stay far from $1$.

\begin{lemma}\label{lem:n3Q_on_line_close_to_1}
There exist constants $\mathfrak{C}>0$ and $\delta_0>0$ such that for all $K>0$, for all $\delta\in (0,\delta_0)$, there exists $C_\delta>0$ such that for all $\eta,\xi>0$ small enough, any one-dimensional interval $\ell$ and any $Q\in H^1(\ell,\Sym)$ satisfying the bound $\eta\:\Eex(Q,\ell) \leq K < \infty$ there is a set $I_\delta\subset\ell$ such that
\begin{align*}
|\ell\setminus I_\delta| \ \leq \ \frac{K}{C_\delta}\eta
\quad\text{and}\quad
|n_3(Q)| \geq 1-\mathfrak{C}\delta
\text{ on }I_\delta\, .
\end{align*}
\end{lemma}

\begin{proof}
For $\delta>0$ small enough let
\begin{align*}
g_\mathrm{min}^\delta 
\defi \min\{ g(Q)\sd Q\in\Sym,\: \dist(Q,\N)\leq \delta, \: |Q-Q_\infty|\geq\mathfrak{a}\sqrt{\delta} \}\, ,
\end{align*}
where $\mathfrak{a}>0$ is chosen as in \cite{ACS2021}.
Proposition 2.5 in \cite{ACS2021} then implies that $g_\mathrm{min}^\delta>0$.
\dsg{
Then, with the notation
\begin{align*}
A \ &:= \ \{x\in\ell\sd|Q-Q_\infty|<\mathfrak{a}\sqrt{\delta}\} \, , \\
B \ &:= \ \{x\in\ell\sd \dist(Q,\N)<\delta\} \, ,
\end{align*}
it holds
\begin{align*}
K 
\ &\geq \ 
\eta\:\Eex(Q,\ell)
\ \geq \ 
\frac{1}{\eta} g_\mathrm{min}^\delta |B\cap(\ell\setminus A)|\, .
\end{align*}
Expanding $|Q-Q_\infty|$ for $Q\in B$ and $\delta>0$ small enough, one can see that  $|n_3(Q)|\leq 1-p$ implies $|Q-Q_\infty|\geq \frac12 s_* \sqrt{p}$.
With the choice $p=\frac{4\mathfrak{a}^2\delta}{s_*^2}$ we then see
\begin{align}\label{lem:n3Q_on_line_close_to_1:eq}
K 
\ &\geq \ 
\frac{1}{\eta} g_\mathrm{min}^\delta \bigg|B\cap\bigg\{x\in\ell\sd |n_3(Q(x))|\leq 1-\frac{4\mathfrak{a}^2}{s_*^2}\delta \bigg\}\bigg| 
\, .
\end{align}
So we set $\mathfrak{C}\defi \frac{4\mathfrak{a}^2}{s_*^2}$ and $I_\delta \defi B\cap \{x\in\ell\sd |n_3(Q(x))| \leq 1-\mathfrak{C}\delta \}$.
It remains to show the estimate on $|\ell\setminus I_\delta|$. 
In view of Proposition~\ref{prop:f__g_geq}, it holds that \begin{align*}
\frac{\xi^2}{\eta} K
\ &\geq \
\int_{\ell\setminus B} \gamma_2\dist^2(Q,\N_{\eta,\xi}) \dx x
\ \geq \
C\gamma_2 \Big(\delta^2 - C\frac{\xi^4}{\eta^4}\Big) |\ell\setminus B| 
\, ,
\end{align*}
from which we infer that $|\ell\setminus B|\lesssim K\frac{\xi^2}{\eta \delta^2}$.
Together with \eqref{lem:n3Q_on_line_close_to_1:eq} this implies
\begin{align*}
|\ell\setminus I_\delta|
\ &\leq \
|\ell\setminus B| + |B\cap(\ell\setminus A)|
\ \lesssim \
\frac{K \xi^2}{\eta \delta^2} + \frac{K \eta}{g_{\text{min}}^\delta}
\, .
\end{align*}
} 
\end{proof}

In the following two sections, we detail how Lemma~\ref{lem:radial_turning} combined with Lemma~\ref{lem:n3Q_on_line_close_to_1} can be applied  in the case of $T\restr\Omega$ and on the surface $\M$.

\subsection{Blow up at points of $T$ in $\Omega$}
\label{subsec:lower_bound_-_blowup}

We define the measure $\mu_{\eta,\xi}(U) \defi \eta\Eex(\Qexn,U)$ for any open set $U$.
Since the energy $\eta\Eex(\Qexn)$ is bounded, the measure $\mu_{\eta,\xi}$ converges (up to extracting a subsequence) weakly* to a measure $\mu$.

\begin{lemma}\label{lem:blow_up_T}
For $\mathcal{H}^2-$a.e. point of rectifiability $x_0\in\Omega$ of $T$  it holds that
\begin{align}\label{lem:blow_up_T:eq}
\frac{\dx\mu}{\dx\mu_T}(x_0) 
\ \geq \  2\: I_\alpha(0,\infty,0,1)\, .
\end{align}
\end{lemma}

\begin{proof}
\textit{Step 1: Notation and preliminaries.}
Recall that for a point of rectifiability $x_0\in \rect(T)$ it holds
\begin{align*}
\lim_{r\rightarrow 0}\frac{\mu_T(B_r(x_0))}{\pi r^2}
\ &= \ \lim_{r\rightarrow 0}\frac{\mathcal{H}^2(\rect(T)\cap B_r(x_0))}{\pi r^2}
\ = \ 1\, .
\end{align*}
We note that for $\mathcal{H}^2-$a.e. point $x_0\in \rect(T)$ there exists the limit 
\begin{align}\label{lem:blow_up_ass_lim_mu_T}
\lim_{r\rightarrow 0} \frac{\mu(B_r(x_0))}{\pi r^2} =: L \, .
\end{align}
In the following we assume that $x_0\in\Omega$ is a point of rectifiability of $T$ which also satisfies \eqref{lem:blow_up_ass_lim_mu_T}.

Let $r_0>0$ such that $B_{r_0}(x_0)\subset \Omega$.
Next, we introduce some notation. 
Let $\Phi_r(x)\defi (x-x_0)/r$ be a rescaling and define $T_r\defi (\Phi_r)_\# T$.
Note that $\Phi_r(B_r(x_0))=B_1(0)=:B_1$.
The rectifiability ensures that there exists a unit vector $\nu\in\mathbb{S}^2$ such that
\begin{align}\label{lem:blow_up_T:conv_F}
\FF\left(T_r\restr B_1 - P_\nu\restr B_1\right) \rightarrow 0 \quad\text{ for } r\rightarrow 0 \, ,
\end{align}
where $P_\nu = \{\nu\}^\perp$ is the two dimensional plane perpendicular to $\nu$ passing through $0$.
Indeed, by Theorem 10.2 in \cite{Maggi2009} we know that $(T_{r_k}-x_0)/r_k^2$ approaches $P_\nu$ in a weak sense and by Theorem 31.2 in \cite{Simon1983} we get the equivalence between the weak convergence and convergence in the $\FF-$norm in our case of $T$ having integer coefficients and $T,\partial T$ being of bounded mass.

Since $\mu_{\eta,\xi}\rightharpoonup \mu$ and $\Tex\rightarrow T$ w.r.t. the flat norm for $\eta,\xi\rightarrow 0$, it holds for almost every $r$ that
\begin{equation}\label{lem:blow_up_ass_1}
\begin{aligned} 
\mu_{\eta,\xi}(B_r(x_0)) &\rightarrow \mu(B_r(x_0)) \\
\FF((\Tex - T)\restr B_r(x_0)) &\rightarrow 0 \, .
\end{aligned}
\end{equation}
We further choose a sequence $(r_k)_{k\in\mathbb{N}}$ converging to zero as $k\rightarrow \infty$ such that \eqref{lem:blow_up_ass_1} holds for each $r_k$ and
\begin{align}\label{lem:blow_up_est_1_1-k}
\left| \frac{\mu(B_\dsg{r_k}(x_0))}{\MM(T\restr B_\dsg{r_k}(x_0))} - L \right| + \FF(T_{r_k}\restr B_1 - P_\nu\restr B_1) \leq \frac{1}{k}\, .
\end{align}
Given the sequence $r_k$, we can extract a subsequence $\xi_k,\eta_k$ such that $\eta_k/r_k\leq \frac{1}{k}$ and
\begin{align}\label{lem:blow_up_est_2_1-k}
\FF((\Phi_{r_k})_\# T_{\eta_k,\xi_k}\restr B_1 - T_{r_k}\restr B_1)
 + \left| \frac{\mu_{\xi_k,\eta_k}(B_{r_k}(x_0)) - \mu(B_{r_k}(x_0))}{\MM(T\restr B_{r_k}(\dsg{x_0}))} \right| 
 \leq \frac{1}{k}\, .
\end{align}

\textit{Step 2: Flat norm convergence.}
Denote $T_k\defi ((\Phi_{r_k})_\# T_{\eta_k,\xi_k})\restr B_1$.
By \eqref{lem:blow_up_T:conv_F} and \eqref{lem:blow_up_est_2_1-k} it follows that $T_k\rightarrow P_\nu\restr B_1$ in the flat norm.
Hence there exist flat chains $A_{2,k}\in\flatch^2$ and $A_{3,k}\in\flatch^3$ with $\MM(A_{2,k}),\MM(A_{3,k})\rightarrow 0$ (for $k\rightarrow\infty$) such that
\begin{align}\label{lem:blow_up_T:T=P+A+dA}
T_k - P_\nu\restr B_1 = A_{2,k} + \partial A_{3,k}\, .
\end{align}
This implies that $\partial(T_k - P_\nu\restr B_1) = \partial A_{2,k}$ or in other words $\partial(T_k - A_{2,k}) = \partial (P_\nu\restr B_1) = 0$ in $B_1$ since $P_\nu$ is the boundary of the half space $H_\nu=\{p+t\nu\sd p\in P_\nu,\: t> 0\}$, i.e.\ $P_\nu\restr B_1 = \partial (H_\nu\restr B_1)$ in $B_1$.
This implies the existence of a flat chain $W_k\in\flatch^3(B_1)$ such that $T_k - A_{2,k} = \partial W_k = \partial (1 - W_k)$, where $1\in\flatch^3(B_1)$ is the flat chain associated to the set $B_1$.
Note that we can also choose the complement set $W_k^c = 1 - W_k$ since it has the same boundary in $B_1$.
From \eqref{lem:blow_up_T:T=P+A+dA} we deduce that
\begin{align*}
\partial(H_\nu\restr B_1 - W_k) 
&= P_\nu\restr B_1 - A_{2,k} + T_k = \partial A_{3,k}\, . 
\end{align*}
This implies that
\begin{align*}
H_\nu\restr B_1 - W_k = A_{3,k} 
\quad\text{or}\quad 
\dsg{1 - (H_\nu\restr B_1 - W_k) =}
H_\nu\restr B_1 - W_k^c = A_{3,k} \, .
\end{align*}
Without loss of generality we choose $W_k$ such that $H_\nu\restr B_1 - W_k = A_{3,k}$ and since $\MM(A_{3,k})\rightarrow 0$ as $k\rightarrow \infty$ we conclude that the symmetric difference $|W_k\Delta(H_\nu\restr B_1)|$ also converges to zero for $k\rightarrow \infty$. 

\begin{figure}
\begin{center}
\includegraphics[scale=1.5]{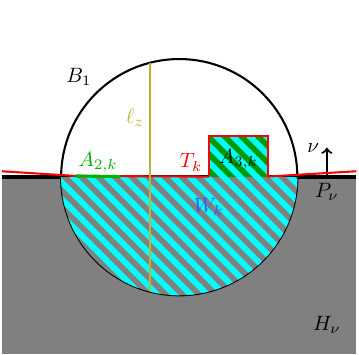}
\end{center}
\caption{Schematic illustration of the quantities involved in Step 2 of the blow up procedure for $T$.}
\label{fig:blow_up_T}
\end{figure}

\textit{Step 3: One dimensional estimates.}
For $z\in P_\nu$ we define the line $\ell_z \defi \{z+t\nu\sd t\in\mathbb{R}\text{ and }z+t\nu\in B_1\}$.

From Step 2 we recall that $|W_k\Delta(H_\nu\restr B_1)|,\MM(A_{2,k})\rightarrow 0$ as $k\rightarrow \infty$. 
This implies that for a subsequence (not relabelled) and almost all $z\in P_\nu$
\begin{align}\label{lem:blow_up_T:ss_conv}
|W_k\Delta (H_\nu\cap \ell_z)|, \mathcal{H}^0(A_{2,k}\cap\ell_z) \rightarrow 0 \quad\text{for } k\rightarrow \infty 
\end{align}
and hence for $k$ large enough $\ell_z$ crosses $T_k = \partial W_k + A_{2,k}$.

\ds{Defining $Q_k(y):=\Qexn(x_0+r_k y)$,} the energy in $B_{r_k}(x_0)$ can be expressed as
\begin{align}\label{lem:blow_up_T:rescale_E}
\frac{\mu_{\eta_k\xi_k}(B_{r_k}(0))}{\pi r_k^2} 
&= \frac{1}{\pi}\int_{B_1(0)} \frac{\eta_k}{2  r_k} |\nabla Q_k|^2 + \frac{r_k}{\eta_k}g(Q_k) + \frac{\eta_k r_k}{\xi_k^2}f(Q_k) + \eta r_k C_0 \dx y \\
&\geq \frac{1}{\pi} \int_{P_\nu\cap B_1} \int_{\ell_z} \frac{\eta_k'}{2  } |\partial_t Q_k|^2 + \frac{1}{\eta_k'}g(Q_k) + \frac{\eta_k'}{(\xi_k')^2}f(Q_k) + \eta_k' C_0' \dx t \dx z\, .
\end{align}
where we introduced the notation $\eta_k' \defi \frac{\eta_k}{r_k}$ and $\xi_k' \defi \frac{\xi_k}{r_k}$ and $C_0' = C_0(\xi_k',\eta_k')$.
Note that $\eta_k'/\xi_k' = \eta_k/\xi_k$.

This implies that by Fatou's lemma
\begin{align*}
L \ \geq \ 
\frac{1}{\pi} \int_{P_\nu\cap B_1} \liminf_{k\rightarrow\infty}\int_{\ell_z} \frac{\eta_k'}{2  } |\partial_t Q_k|^2 + \frac{1}{\eta_k}g(Q_k) + \frac{\eta_k'}{(\xi_k')^2}f(Q_k) + \eta_k' C_0' \dx t \dx z\, .
\end{align*}

In view of \eqref{lem:blow_up_T:eq} that we want to prove, we can restrict ourselves even further to the lines $\ell_z$ with 
\begin{align}\label{lem:blow_up_T:E_upper_bound_on_ell}
\liminf_{k\rightarrow\infty}\int_{\ell_z} \frac{\eta_k'}{2} |\partial_t Q_k|^2 + \frac{1}{\eta_k'}g(Q_k) + \frac{\eta_k'}{(\xi_k')^2}f(Q_k) + \eta_k' C_0' \dx t \ \leq \ 2 I_\alpha(0,\infty,0,1) \, ,
\end{align}
otherwise there is nothing to prove. 
By choosing another subsequence (which depends on $z$), we can assume this liminf is a limit and therefore that the sequence is bounded.

Using the inequality $A^2+B^2\geq 2AB$, the bound \eqref{lem:blow_up_T:E_upper_bound_on_ell} implies that
\begin{align*}
2 I_\alpha(0,\infty,0,1)
\ &\geq \ 
2 \int_{\ell_z} \dsg{\frac{1}{\sqrt{2}}}|\partial_t Q_k|\sqrt{\left(\frac{\eta_k'}{\xi_k'}\right)^2 f(Q_k) + g(Q_k) + C_0'} \dx t \\
\ &\geq \  
\dsg{\sqrt{2}}\gamma_2 \frac{\eta_k'}{\xi_k'} \int_{\ell_z} |\partial_t Q_k|\dist(Q_k,\N_{\eta_k',\xi_k'})\dx t\, ,
\end{align*}
where we also used \eqref{prop:f__g_geq:eq} of Proposition~\ref{prop:f__g_geq} in the last inequality.
Denoting $m\defi \min_{\ell_z} \dist(Q_k,\N_{\eta_k,\xi_k})$ and $M\defi \max_{\ell_z} \dist(Q_k,\N_{\eta_k,\xi_k})$ we can estimate the energy necessary for switching from $(M+m)/2$ to $M$ \ds{on the set $\ell_{z,m,M}\defi \{x\in\ell_z\sd \dist(Q_k(x),\N_{\eta_k,\xi_k}) > \tfrac12(M+m)\}$} by
\begin{align}\label{lem:blow_up_T:upper_bound_on_max_dist}
2 I_\alpha(0,\infty,0,1)
\ &\geq \ 
\dsg{\sqrt{2}}\gamma_2 \frac{\eta_k'}{\xi_k'} \frac{M+m}{2} \int_{\ell_{z,m,M}} |\partial_t Q_k| \dx t 
\ \geq \ 
\dsg{\sqrt{2}}\frac{\gamma_2}{\dsg{4}} \frac{\eta_k'}{\xi_k'} (M^2-m^2)\, ,
\end{align}
\ds{where in the last step we used $|Q^{(M)}-Q^{(\tfrac12(M+m))}|\geq M - |Q-Q^{(\tfrac12(M+m))}|\geq M - \tfrac12(M+m)$ for matrices $Q^{(d)}\in\Sym$ at distance $d$ from $\N_{\eta_k,\xi_k}$ and any $Q\in \N_{\eta_k,\xi_k}$.
Taking then the supremum of the RHS in $Q\in\Sym$ yields $\int_{\ell_{z,m,M}} |\partial_t Q_k| \dx t \geq \tfrac12(M-m)$.}
In order to obtain a uniform convergence of $\dist(Q_k,\N_{\eta_k,\xi_k})$, it remains to estimate $m$.

Again from \eqref{lem:blow_up_T:E_upper_bound_on_ell} by using \eqref{prop:f__g_geq:eq} in Proposition~\ref{prop:f__g_geq}, we get that
\begin{align*}
2 I_\alpha(0,\infty,0,1)
\ &\geq \ \gamma_2 \frac{\eta_k'}{(\xi_k')^2} \int_{\ell_z} \dist^2(Q_k,\N_{\eta_k',\xi_k'}) \dx t 
\ \geq \  \gamma_2 \frac{\eta_k'}{(\xi_k')^2} |\ell_z| m^2\, .
\end{align*}
In other words, $m^2\leq \frac{2 I_\alpha(0,\infty,0,1)}{\gamma_2|\ell_z|}\frac{(\xi_k')^2}{\eta_k'}$.
Plugging this estimate into \eqref{lem:blow_up_T:upper_bound_on_max_dist} yields 
\begin{align}\label{lem:blow_up_T:unif_bound_on_dist_N_on_ell}
\sup_{\ell_z} \dist^2(Q_k,\N_{\eta_k',\xi_k'}) 
\ &= \ M^2
\ \leq \ \frac{\dsg{8} I_\alpha(0,\infty,0,1)}{\dsg{\sqrt{2}}\gamma_2}\left(\frac{\xi_k'}{\eta_k'} + \frac{(\xi_k')^2}{|\ell_z|\eta_k'}\right)\, .
\end{align}
In view of \eqref{prop:f__g_geq:Nex_N} of Proposition~\ref{prop:f__g_geq} we can conclude that $Q_k$ is uniformly close to $\N$ and \dsg{$\dist(Q_k,\N)$} converges uniformly to zero as $\xi_k',\eta_k'\rightarrow 0$.

This implies together with the convergences in \eqref{lem:blow_up_T:ss_conv} that there exists a sequence $t_k\rightarrow 0$ such that $n_3(Q_k - Y_k)(z + t_k \nu) = 0$, where $Q_k(y)\defi Q_{\eta_k,\xi_k}(x_0+r_k y)$ and $Y_k\defi Y_{\eta_k,\xi_k}$.

We now split $\ell_z$ into $\ell_z^\pm$, where $\ell_z^+\defi \{z+t\nu\in\ell\sd t\geq t_k\}$ and $\ell_z^-\defi \{z+t\nu\in\ell\sd t\leq t_k\}$ and show that on both rays there are points for which $Q_k$ is close to $Q_\infty$.

Applying Lemma~\ref{lem:n3Q_on_line_close_to_1} for $\delta>0$ with the bound in \eqref{lem:blow_up_T:E_upper_bound_on_ell} implies that for $k$ large enough there exists $t_k^+\in (t_k,1)$ such that $|n_3(Q_k(t_k^+)-Y_k)|>1-\mathfrak{C}\sqrt{\delta}$.
The goal is to take $\delta\rightarrow 0$. 
For this, we choose a sequence $\delta_k$ depending on $k$ such that $\delta_k$ and $\eta_k'/C_{\delta_k}\leq \frac{1}{k}$ converge to zero as $k\rightarrow\infty$, where $C_{\delta_k}$ is the constant from Lemma~\ref{lem:n3Q_on_line_close_to_1}.
Similarly, there exists $t_k^-\in (-1,t_k)$ such that $|n_3(Q_k(t_k^-)-Y_k)|>1-\mathfrak{C}\sqrt{\delta_k}$.

The final estimate for the integral over $\ell_z$ then follows by summing the contributions from $\ell_z^\pm$, both in which we pass from $n_3(Q_k-Y_k)=0$ to $|n_3(Q_k-Y_k)|=1$.
Knowing that $Q_k$ is uniformly close to $\N$, we can apply 
Lemma 17 in \cite{Canevari2015a}, the Lipschitz assumption on $g$ in \eqref{def:g:N-lipschitz} and use the definition of $I_\alpha$ 
to determine the energetic cost on $\ell_z^\pm$. 
This yields
\begin{equation}\label{lem:blow_up_T:final_est_fatou}
\begin{aligned}
L 
\ &\geq \ \liminf_{k\rightarrow\infty} \frac{1}{\pi} \int_{P_\nu\cap B_1} \int_{\ell_z} \frac{\eta_k'}{2} |\nabla Q_k|^2 + \frac{1}{\eta_k'}g(Q_k) + \frac{\eta_k'}{(\xi_k')^2}f(Q_k) + \eta_k' C_0' \dx t \dx z \\
\ &\geq \ \frac{1}{\pi} \int_{P_\nu\cap B_1} \liminf_{k\rightarrow\infty}  \int_{\ell_z} \frac{\eta_k'}{2} |\nabla Q_k|^2 + \frac{1}{\eta_k'}g(Q_k) + \frac{\eta_k'}{(\xi_k')^2}f(Q_k) + \eta_k' C_0' \dx t \dx z \\
\ &\geq \ \frac{1}{\pi} \int_{P_\nu\cap B_1} \liminf_{k\rightarrow\infty}  \int_{\ell_z} \frac{\eta_k'}{2} (1-C\Vert Q_k - \R(Q_k)\Vert_{L^\infty(\ell_z)})|\nabla \R(Q_k)|^2 \\
&\qquad\qquad\qquad\qquad\qquad\qquad\qquad + \frac{1}{\eta_k'}g(\R(Q_k)) - C |Q_k - \R(Q_k)| \dx t \dx z \\
\ &\geq \ \frac{1}{\pi} \int_{P_\nu\cap B_1}  2I_\alpha(0,\infty,0,1) \dx z \\
\ &\geq \ 2\: I_\alpha(0,\infty,0,1)\, .
\end{aligned}
\end{equation}
\end{proof}

\subsection{Surface energy}
\label{subsec:lower_bound_-_surface}

In this section we do the necessary calculations to find the announced energy contribution on $\M$. 
For this, we estimate the energy in a boundary layer around $\M$. 
More precisely, we define $\M_{\sqrt{\eta}}\defi\{ x\in\Omega\sd \dist(x,\M)\leq \sqrt{\eta} \}$.
Then we proceed similarly to the previous section, the goal is to apply Lemma~\ref{lem:radial_turning} to the rays perpendicular to $\M$ on which $\Qexn$ is taking values close to $\N$.

We assume $\eta$ small enough such that $\sqrt{\eta} < \frac12\radM$, where $\radM$ was fixed in the beginning of Section~\ref{sec:cpt} such that $\radM$ is smaller than the minimal curvature radius of $\M$.
For $\omega\in\M$ and $r>0$ we define
\begin{align}\label{def:lower_bound_surface_rays}
L_{\omega,r} \defi \{ \omega + t\nu(\Omega)\sd t\in [0,r] \} \, .
\end{align}
We now rewrite the energy so that the line integrals over $L_{\omega,\sqrt{\eta}}$ appear.
We note that for $0<\eta\ll 1$ the map $\M\times [0,\sqrt{\eta}]\rightarrow \Omega$ given by $(\omega,r)\mapsto \omega + r\nu(\omega)$ is injective. 
The differential of this map is given by $\id_{T_\omega\M} + r\dx_\omega \nu + \nu$. 
Using the normalized eigenvectors $\nu,\overline{\omega_1},\overline{\omega_2}$ corresponding to the eigenvalues $1,\kappa_1,\kappa_2$ with $\kappa_i$ being the principal curvatures of $\M$ at $\omega$, i.e. the eigenvalues of the Gauss map $\dx_\omega\nu$. Then
\begin{align*}
\det(\id+r \dx_\omega \nu(\omega)) = (1+r\kappa_1)(1+r\kappa_2)
\end{align*} 
and the gradient transforms as
\begin{align*}
|\nabla \Qexn|^2 = |\partial_r \Qexn|^2 + \frac{1}{|1+r|\kappa_1||^2}|\partial_{\overline{\omega_1}} \Qexn|^2 + \frac{1}{|1+r|\kappa_2||^2}|\partial_{\overline{\omega_2}} \Qexn|^2\, .
\end{align*}
In order to shorten our formulas, we still use the notation $\nabla \Qexn$.
The energy can then be rewritten as
\begin{align*}
\Eex(\Qexn,\M_{\sqrt{\eta}})
\ &= \ \int_\M \int_0^{\sqrt{\eta}} \left( \frac{1}{2}|\nabla \Qexn|^2 + \frac{1}{\xi^2} f(\Qexn) + \frac{1}{\eta^2}g(\Qexn) \right) \prod_{i=1}^2(1+r\kappa_i) \dx r \dx \omega \, .
\end{align*}

We now distinguish two cases depending on whether the ray $L_{\omega,\sqrt{\eta}}$ intersects $\Texn$ or not.

\textit{Case 1: $L_{\omega,\sqrt{\eta}}$ does not intersect $\Texn$.}
In this case we can assume that
\begin{align*}
\int_0^{\sqrt{\eta}} \left( \frac{1}{2}|\nabla \Qexn|^2 + \frac{1}{\xi^2} f(\Qexn) + \frac{1}{\eta^2}g(\Qexn) \right) \prod_{i=1}^2(1+r\kappa_i) \dx r 
\ \leq \ I_\alpha(0,\infty,1,\cos(\theta))\, ,
\end{align*}
otherwise there is nothing to prove.
With the same argument as in \eqref{lem:blow_up_T:E_upper_bound_on_ell}-\eqref{lem:blow_up_T:unif_bound_on_dist_N_on_ell} we can show that $\sup_{L_{\omega,\sqrt{\eta}}} \dist(\Qexn,\N_{\eta,\xi})$ converges to zero as $\xi,\eta\rightarrow 0$.
Analogously to the blow up procedure, for $\delta>0$ we use Lemma~\ref{lem:n3Q_on_line_close_to_1} to deduce that there exists a radius $t_\omega\in [0,\sqrt{\eta}]$ such that $|n_3(\Qexn)(\omega+t_\omega\nu(\omega))|\geq 1-\mathfrak{C}\sqrt{\delta}$. 
We choose a sequence $\delta_{\eta}\rightarrow 0$ such that $\eta/C_{\delta_\eta}\rightarrow 0$ as $\eta\rightarrow 0$.
Note that $\Qexn$ does not verify the boundary condition \eqref{eq:bc}, but a slightly perturbed version. 
For $\eta,\xi\rightarrow 0$ we still obtain the right energy thanks to Proposition~\ref{prop:prop_reg_seq} and the uniform convergence therein.
As in \eqref{lem:blow_up_T:final_est_fatou} we then obtain 
\begin{equation} \label{surface_energy_wo_T}
\begin{aligned}
\liminf_{\eta,\xi\rightarrow 0} &\int_{L_{\omega,\sqrt{\eta}}} \frac{\eta}{2} |\nabla \Qexn|^2 + \frac{1}{\eta}g(\Qexn) + \frac{\eta}{\xi^2}f(\Qexn) + \eta C_0' \dx t \\
\ &\geq \ 
\liminf_{\eta,\xi\rightarrow 0}  \int_{L_{\omega,\sqrt{\eta}}} \frac{\eta}{2} (1-C\Vert \Qexn - \R(\Qexn)\Vert_{L^\infty(\ell_z)})|\nabla \R(\Qexn)|^2 \\
&\qquad\qquad\qquad + \frac{1}{\eta}g(\R(\Qexn)) - C |\Qexn - \R(\Qexn)| \dx t \\
\ &\geq \ I(0,\infty,1,|\cos(\theta)|)\, .
\end{aligned}
\end{equation}

\textit{Case 2: $L_{\omega,\sqrt{\eta}}$ intersects $\Texn$.}
Let $t_\omega'\in (0,\sqrt{\eta})$ denote the radius of intersection between $L_{\omega,\sqrt{\eta}}$ and $\Texn$.
The only difference to Case 1 is that we estimate the two parts $t\leq t_\omega'$ and $t\geq t_\omega'$ separately.

With the same reasoning as before we can assume that the energy on the ray is bounded and that $\dist(\Qexn,\N_{\eta,\xi})$ is uniformly converging to zero on the ray. 
On $L_{\omega,t_\omega'}$ we obtain just like in Step 1 the estimate
\begin{equation} \label{surface_energy_w_T_1}
\begin{aligned}
\liminf_{\eta,\xi\rightarrow 0} \int_{L_{\omega,t_\omega'}} \frac{\eta}{2} |\nabla \Qexn|^2 + \frac{1}{\eta}g(\Qexn) + \frac{\eta}{\xi^2}f(\Qexn) + \eta C_0' \dx t 
\ &\geq \ I_\alpha(0,\infty,\cos(\theta),0)\, .
\end{aligned}
\end{equation}

On the remaining part of the ray $L_{\omega,\sqrt{\eta}}$ we want to find the energy $I_\alpha(0,\infty,1,0)$.
Since $t_\omega'$ might be arbitrarily close to $\sqrt{\eta}$, we cannot apply Lemma~\ref{lem:n3Q_on_line_close_to_1} to conclude that $n_3(\Qexn)$ is close to $\pm 1$ somewhere. 
Extending the ray up to a distance $t=2\sqrt{\eta}$ from $\M$ and repeating the above reasoning, we can find for $\delta>0$ and $\eta$ small enough $t_\omega\in [\sqrt{\eta},2\sqrt{\eta}]$ such that $|n_3(\Qexn)(\omega+t_\omega\nu(\omega))|\geq 1-\mathfrak{C}\sqrt{\delta}$. 
Now we proceed again as in \eqref{surface_energy_wo_T} and combine with \eqref{surface_energy_w_T_1} to obtain
\begin{equation} \label{surface_energy_w_T}
\begin{aligned}
\liminf_{\eta,\xi\rightarrow 0} &\int_{L_{\omega,2\sqrt{\eta}}} \frac{\eta}{2} |\nabla \Qexn|^2 + \frac{1}{\eta}g(\Qexn) + \frac{\eta}{\xi^2}f(\Qexn) + \eta C_0' \dx t \\
\ &\geq \ I_\alpha(0,\infty,\cos(\theta),0) + I_\alpha(0,\infty,1,0)\, .
\end{aligned}
\end{equation}

\subsection{Proof of compactness and lower bound}
\label{subsec:compactness_alpha_0}

We now need to combine the estimates \eqref{prop:lower_bound_-_line_eq}, \eqref{lem:blow_up_T:eq}, \eqref{surface_energy_wo_T} and \eqref{surface_energy_w_T}.
To this aim, we use the localization technique for $\Gamma-$convergence as described for example in \cite[Ch. 16]{Braides2002}.
Let $U_i$, $i=1,2,3$ be three pairwise disjoint sets open in $\overline{\Omega}$.
Then it holds that
\begin{align*}
\liminf_{\eta,\xi\rightarrow 0} \eta\: \Eex(\Qex)
\ &\geq \ \liminf_{\eta,\xi\rightarrow 0} \sum_{i=1}^3 \eta\: \Eex(\Qex,U_i)
\ \geq \ \sum_{i=1}^3 \liminf_{\eta,\xi\rightarrow 0} \eta\: \Eex(\Qex,U_i) \\
\ &\geq \ \frac{\pi}{2}s_*^2\beta\: \MM(S\restr U_1) + 2 I_\alpha(0,\infty,1,0)\MM(T\restr(\Omega\cap U_2)) \\
&\qquad + \int_{\M} I(0,\infty,1,|\cos(\theta)|) \dx\mu_{(1-T\restr\M)\restr U_3} \\
&\qquad + \int_{\M} I_\alpha(0,\infty,1,0) + I_\alpha(0,\infty,\cos(\theta),0) \dx\mu_{T\restr(\M\cap U_3)}\, .
\end{align*}
Since the LHS does not depend on the sets $U_i$, we can take the supremum over all pairwise disjoint open sets.
For $\epsilon>0$ and by inner regularity we can approximate the measure $\MM(S)$ by a compact set $A_{1,\epsilon}\subset\rect(S)$ and an open set $U_{1,\epsilon}\supset A_{2,\epsilon}$ such that $\MM(S)-\MM(S\restr A_{1,\epsilon})\leq \epsilon$, $\mathcal{H}^2(\M\cap\overline{U_{1,\epsilon}})\leq \epsilon/2$ and $\MM(T\restr \overline{U_{1,\epsilon}})\leq \epsilon/2$ since the measures $\mu_S$ and $\mu_T$ are mutually singular.
Furthermore, we find another compact set $A_{2,\epsilon}\subset(\rect(T)\cap\Omega)\setminus U_{1,\epsilon}$ such that $\MM(T\restr\Omega)-\MM(T\restr A_{2,\epsilon})\leq \epsilon$. 
Then, by construction there exists an open set $U_{2,\epsilon}\supset A_{2,\epsilon}$ such that $U_{2,\epsilon}\cap U_{1,\epsilon}=\emptyset$ and $\dist(\M,U_{2,\epsilon})>0$.
Finally, taking an open neighbourhood of $\M$ disjoint from $\overline{U_{2,\epsilon}}$ and removing $\overline{U_{1,\epsilon}}$ from it, we find a third open set $U_{3,\epsilon}$ which satisfies $\mathcal{H}^2(\M\setminus U_{3,\epsilon})\leq \epsilon$.
By monotonicity we then find 
\begin{equation}\label{lower_bound_conclusion_alpha}
\begin{aligned}
\liminf_{\eta,\xi\rightarrow 0} \eta\: \Eex(\Qex)
\ &\geq \ \sup_{U_1,U_2,U_3}  \frac{\pi}{2}s_*^2\beta\: \MM(S\restr U_1) + 2 I_\alpha(0,\infty,1,0)\MM(T\restr(\Omega\cap U_2)) \\
&\qquad + \int_{\M} I(0,\infty,1,|\cos(\theta)|) \dx\mu_{(1-T\restr\M)\restr U_3} \\
&\qquad + \int_{\M} I_\alpha(0,\infty,1,0) + I_\alpha(0,\infty,\cos(\theta),0) \dx\mu_{T\restr(\M\cap U_3)} \\
\ &\geq \ 2 I_\alpha(0,\infty,1,0)\MM(T\restr\Omega) + \frac{\pi}{2}s_*^2\beta\: \MM(S) \\
&\qquad + \int_{\M} I(0,\infty,1,|\cos(\theta)|) \dx\mu_{(1-T\restr\M)} \\
&\qquad + \int_{\M} I_\alpha(0,\infty,1,0) + I_\alpha(0,\infty,\cos(\theta),0) \dx\mu_{T\restr\M}\, .
\end{aligned}
\end{equation}
We now want to pass to the limit $\alpha\rightarrow 0$.
In order to mark the dependence of $T$ and $S$ on $\alpha$, we add the index $\alpha$ in our notation for the rest of the proof.
Since $I_\alpha(0,\infty,1,0)\geq s_*c_*>0$, the mass of $T_\alpha\restr\Omega$ is bounded uniformly in $\alpha$ and since $\M$ has finite surface area it follows that $\MM(T_\alpha)$ is bounded independent of $\alpha$.
Since the mass of $S_\alpha$ and the length of the curves $\Gamma_\alpha$ are also uniformly bounded, we conclude that the flat chains $T_\alpha$ as well as their boundaries $\partial T_\alpha = S_\alpha + \Gamma_\alpha$ have finite mass.
Choosing a sequence $\alpha_k\rightarrow 0$,  \eqref{lower_bound_conclusion_alpha} holds and we can apply the compactness theorem for flat chains \cite[Cor. 7.5]{Fleming1966} as stated in Subsection~\ref{subsec:flat_chains}. 
From this we get that there exists a subsequence (not relabelled) and flat chains $T\in\flatch^2$, $S\in\flatch^2$ such that $\FF(T_{\alpha_k} - T)\rightarrow 0$ and $\FF(S_{\alpha_k} - S)\rightarrow 0$ as $k\rightarrow\infty$.
Since boundaries are preserved under flat convergence and, as we will prove below, $\Gamma_\alpha\rightarrow \Gamma$ in the flat norm, it holds that $\partial T = S + \Gamma$.
We note that 
$I_{\alpha_k}(0,\infty,\pm 1,0)\rightarrow I(0,\infty,\pm 1,0)$, 
and 
$I_{\alpha_k}(0,\infty,\cos(\theta),0)\rightarrow I(0,\infty,\cos(\theta),0)$ as $\alpha_k\rightarrow 0$.
Passing to the limit $\alpha_k\rightarrow 0$ in \eqref{lower_bound_conclusion_alpha} thus yields
\begin{equation}\label{lower_bound_conclusion_alpha_0}
\begin{aligned}
\liminf_{\eta,\xi\rightarrow 0} \eta\: \Eex(\Qex)
\ &\geq \ 2 I(0,\infty,1,0)\MM(T\restr\Omega) + \frac{\pi}{2}s_*^2\beta\: \MM(S) \\
&\qquad + \int_{\M} I(0,\infty,1,|\cos(\theta)|) \dx\mu_{(1-T\restr\M)} \\
&\qquad + \int_{\M} I(0,\infty,1,0) + I(0,\infty,\cos(\theta),0) \dx\mu_{T\restr\M} \\
\ &= \ 4 s_* c_* \MM(T\restr\Omega) + \frac{\pi}{2}s_*^2\beta\: \MM(S) \\
&\qquad + 2s_*c_*\int_{\M} (1-|\cos(\theta)|) \dx\omega  + 4s_*c_*\int_{\M} |\cos(\theta)| \dx\mu_{T\restr\M}\, .
\end{aligned}
\end{equation}
Combining the compactness result from Subsection~\ref{subsec:compact_fix_Y} for fixed $\alpha$ with the above estimates, we can choose a diagonal sequence $\alpha_{k(\xi,\eta)}\rightarrow 0$ as $\eta,\xi\rightarrow 0$ such that
\begin{align*}
\lim_{\eta,\xi\rightarrow 0}\FF(T_{\eta,\xi,n,\alpha_{k(\xi,\eta)}} - T) = 0 \qquad \lim_{\eta,\xi\rightarrow 0}\FF(S_{\eta,\xi,n,\alpha_{k(\xi,\eta)}} - S) = 0
\end{align*}
and \eqref{lower_bound_conclusion_alpha_0} holds.

It remains to verify that $\partial T = S + \Gamma$ as claimed by Theorem~\ref{thm:main}.
We recall from the boundary condition \eqref{eq:bc} that
\begin{align*}
\Qex(\omega) \ = \ Q_b(\omega) = s_*\left( \nu(\omega)\otimes\nu(\omega) - \frac13\id \right) \qquad \text{ for all }\omega\in\M\, .
\end{align*}
This implies that on $\Gamma$
\begin{align}\label{Gam:n3_Qb}
n_3(\Qex(\omega)) = n_3(Q_b(\omega))= \nu_3(\omega) = 0 \, .
\end{align}
by definition.
Furthermore, we assumed that the derivative of $\nu_3$ on $\Gamma$ is non-degenerate, i.e.\ $\nabla_\omega \nu_3(\omega)\neq 0$ for all $\omega\in\Gamma$.
Hence, on $\Gamma$ it holds
\begin{equation} \label{Gam:grad_n3}
\begin{aligned}
\nabla_\omega n_3(\Qex(\omega)) 
\ = \ \nabla_\omega \nu_3(\omega)
\ \neq \ 0 \, .
\end{aligned}
\end{equation}

Next, we consider the function $F(\omega,n,Y) \defi n_3(\Qexn(\omega) + Y)$ for $n\in\mathbb{N}$ and $Y\in B_\alpha(0)\subset\Sym$ for $0<\alpha\ll 1$.
Note that we can rewrite
\begin{align*}
F(\omega,n,Y) 
\ &= \ n_3(\Qex(\omega) + (\Qexn(\omega)-\Qex(\omega)) + Y) \\
\ &= \ n_3(Q_b(\omega) + (\Qexn(\omega)-\Qex(\omega)) + Y) \, .
\end{align*}
Since on $\M$, $\Qexn$ is by construction an approximation by convolution of $Q_b$, it holds that $\Qexn \rightarrow Q_b$ in $C^1$ on $\M$ for $n\rightarrow 0$.
In other words, from \eqref{Gam:n3_Qb} we get that $F(\omega,\infty,0) = 0$.

For the rest of the proof we argue locally on $\M$.
Let $(u,v)$ be a local parametrization on $\M$ such that $\nabla u$ is parallel to $\Gamma$ and $\nabla v$ is in direction of the normal vector of the curve $\Gamma$, called $\nu_\Gamma$.
We can choose $(u,v)$ such that $\omega_0 = (u(0),v(0))\in\Gamma$ and $(u,v(0))$ locally parametrizes $\Gamma$.
Then 
\begin{align*}
\partial_v F(\omega,n,Y)|_{\omega_0,\infty,0}
\ &= \ \partial_v F((u,v),n,Y)|_{(0,0),\infty,0} \\
\ &= \ D n_3(Q_b(\omega_0))(\partial_v Q_b(\omega_0) + \partial_v(\Qexn(\omega)-\Qex(\omega)))\, .
\end{align*}
For $n$ large enough we can assume that $\Vert D n_3 \Vert_{C^0(\N)} \Vert \Qexn - \Qex\Vert_{C^1(\M)}\leq \frac{1}{2}\inf_\omega |\partial_v n_3(Q_b(\omega))|$ by Proposition~\ref{prop:prop_reg_seq}.
Since $D n_3(Q_b(\omega_0))\partial_v Q_b(\omega_0) = \partial_v n_3(Q_b(\omega))|_{\omega=\omega_0}$ it follows from \eqref{Gam:grad_n3}  that $\partial_v F(\omega,n,Y)|_{\omega_0,\infty,0}\neq 0$.

The assumptions of the implicit function theorem are therefore satisfied and there exists an open neighbourhood $V$ of $(u(0),\infty,0)$ and a function $\tilde{v}$ defined on $V$ such that $F((u,\tilde{v}(u,n,Y)),n,Y) = 0$ on $V$.
In other words,
\begin{align*}
0 
\ &= \ F((u,\tilde{v}(u,n,Y)),n,Y)
\ = \ n_3(\Qexn((u,\tilde{v}(u,n,Y))) + Y)
\, .
\end{align*}
So $(u,\tilde{v}(u,n,Y))$ serves as a local parametrization of the set $\Gamma_{n,Y} := \{ \omega\in\M\sd n_3(\Qexn(\omega) + Y) = 0 \}$.
Noting that $\M$ is of class $C^2$ and hence $\nu\in C^1$, it holds that $\tilde{v}$ and $\Gamma_{n,Y}$ are also of class $C^1$ and in particular $\Gamma_{n,Y}$ has finite length.

Since $\tilde{v}\rightarrow 0$ uniformly as $n\rightarrow \infty$ and $Y\rightarrow 0$, it holds that $n_3(\Qexn+Y)$ also uniformly converges to  $n_3(Q_b)$. By Theorem 3.3 in \cite{Dambrine2020} it follows the Hausdorff convergence of $\Gamma_{n,Y}$ to $\Gamma$, i.e.\
\begin{align*}
\dist_\mathcal{H}(\Gamma,\Gamma_{n,Y})
\ := \ \max\left\{ \sup_{\omega\in\Gamma}\dist(\omega,\Gamma_{n,Y}), \sup_{\omega'\in\Gamma_{n,Y}}\dist(\omega',\Gamma) \right\} \rightarrow 0 \text{ for } n\rightarrow\infty\text{ and }Y\rightarrow 0 \, .
\end{align*}

Using the parametrization $\tilde{v}$ to link $\Gamma$ to $\Gamma_{n,Y}$, we can also build a flat $2-$chain $G_{n,Y}$ with boundary $\partial G_{n,Y} = \Gamma - \Gamma_{n,Y}$.
It then holds
\begin{align*}
\FF(\Gamma - \Gamma_{n,Y}) 
\ &\leq \ \MM(G_{n,Y}) 
\ \leq \ \sup_{(\omega,m,Z)\in V} \mathcal{H}^1(\Gamma_{m,Z}) \: \dist_\mathcal{H}(\Gamma,\Gamma_{n,Y})\, .
\end{align*}

\section{Upper bound}
\label{sec:upper_bound}

This section is devoted to the construction of the recovery sequence of Theorem~\ref{thm:main}. 
Essentially, there are three steps in this construction:
\begin{enumerate}
\item We approximate $T$ by a sequence $T_n$, solution to a minimization problem. 
The advantage of replacing $T$ by $T_n$ is the gain of regularity. 
Indeed, as we will see in Subsection~\ref{subsec:reg_T}, $T$ and its boundary inside $\Omega$ will be of class $C^{1,1}$. Furthermore, by a comparison argument, we can show that $\partial (T_n\restr\M)$ is a line of finite length.
\item We introduce local coordinate systems in which we can define $\Qexn$ and estimate its energy. 
\item Choosing a diagonal sequence $n(\xi,\eta)$ we find the recovery sequence.
\end{enumerate}

\subsection{A first regularity result for (almost) minimizers}
\label{subsec:reg_T}

In this subsection, we rewrite the limit energy $\E_0$ in a way that it only depends on $T$:
\begin{align} \label{eq:limit_energy_only_T}
\E_0(T) 
\ &= \ E_0(\M,\ee_3) 
+ 4s_*c_*\int_{\M} |\cos(\theta)| \dx\mu_{T\restr\M} 
+ 4s_*c_* \MM(T\restr\Omega)
+ \frac{\pi}{2}s_*^2\beta \MM(\partial T - \Gamma)\, ,
\end{align}
where $\Gamma\in\flatch^1$ is given by the curve $\{\nu_3=0\}\subset\M$. 
For the approximation of a flat chain $T\in\flatch^2$ we are going to study the following minimization problem: 
\begin{align} \label{eq:limit_energy_T_regularity}
\min_{\tilde{T}\in\flatch^2} \E_0(\tilde{T}) + n\: \FF(\tilde{T} - T)\, ,
\end{align}
for $n\in\mathbb{N}$. 
The existence of a minimum of \eqref{eq:limit_energy_T_regularity} is imminent since by assumption $T$ verifies $\E_0(T) + n\FF(T-T) = \E_0(T)<\infty$, the energy is non-negative and lower semi-continuous with respect to convergence in the flat norm.
We have the following result:

\begin{proposition}\label{prop:low_reg_T}
Let $T\in\flatch^2$ with $\E_0(T)< \infty$. For all $n\in\mathbb{N}$, the problem \eqref{eq:limit_energy_T_regularity} has a solution $T_n\in\flatch^2$. The minimizer $T_n$ verifies
\begin{enumerate}
\item $T_n\rightarrow T$ for $n\rightarrow\infty$ in the flat norm.
\item $T_n\restr\Omega$ is of class $C^{1}$ up to the boundary $(\partial T_n)\restr\Omega$. 
\item $\partial T_n$ is of class $C^{1,1}$ \dsg{(with uniformly bounded curvature) outside of $\Gamma$}. 
\end{enumerate}
\end{proposition}

We note that the above Proposition also holds true for $n=0$, i.e.\ minimizers of \eqref{eq:limit_energy_only_T} and hence of our limit problem are of class $C^1$ up to the boundary in $\Omega$ which is of class $C^{2}$. As we will see later, the minimizers of $\E_0$ are in fact smooth (see Proposition~\ref{prop:high_reg_T}).
In order to make this subsection more readable and simplify notation, we divide \eqref{eq:limit_energy_only_T} by $4s_*c_*$ and redefine the parameter $\beta$ to replace the constant $\frac{1}{8}\frac{s_*}{c_*}\beta$. Also, we will simply write $n$ instead of $\frac{n}{4c_*s_*}$. Since in this subsection we are only concerned with the regularity of minimizers, this change of notation does not impact our results.

The proof of Proposition~\ref{prop:low_reg_T} makes use of a series of lemmas which we are going to state and prove first.
The main ideas for the regularity of $T_n$ and $\partial T_n$ have already been developed in earlier papers \cite{Morgan1994a,Taylor1976,David2019,David2010}, so it remains to check that we can apply them in our case. 
For the sake of simple notation, we drop the subscript $n$ for the rest of this section and define $S\defi \partial T - \Gamma$. 
We recall from Subsection~\ref{subsec:flat_chains} that $\rect(S)$ is the set of all points of rectifiability of $S$. In particular, for $x_0\in\rect(S)$ the density $\lim_{r\rightarrow 0} \mu_S(B_r(x_0))/(2r)$ exists and is strictly positive.

\begin{lemma}\label{lem:reg_T_-_S_closed}
It holds that $\supp(S)=\overline{\rect(S)}$ and $\mathcal{H}^1(\supp(S)\setminus\rect(S))=0$.
\end{lemma}

\begin{proof}
Let's show first that $S$ is supported by a closed $1$-dimensional set.

For this, we prove that $S$ cannot contain subcycles of arbitrary small length.
Assume that $S_1$ is a subcycle of $S$, i.e.\ $\MM(S)=\MM(S_1)+\MM(S-S_1)$ and $\partial S_1=0$, and that $S_1$ is supported in $B_r(x_0)$ for $r\in (0,\frac12\radM)$. 
By (7.6) in \cite{Fleming1966}, there exists a constant $b>0$ and $T_1\in \flatch^2$ such that $S_1 = \partial T_1$ and $\MM(T_1)\leq b\MM(S_1)^2$.
By projecting $T_1$ onto $B_r(x_0)\cap\overline{\Omega}$, we can furthermore assume that $T_1$ is supported in $B_r(x_0)$ and lies within $\overline{\Omega}$.
Projecting onto $B_r(x_0)$ does not affect the previous estimate since it decreases the mass. 
Projecting $T_1\restr E$ onto $\M$ has Lipschitz constant less than $1+4\frac{r}{\radM}$ and hence, the estimate stays true with an additional factor of $1+4\frac{r}{\radM}$.
We estimate by minimality of $T$
\begin{align*}
\E_0(T) + n\FF(T-T_0) 
\ &\leq \ \E_0(T+T_1) + n\FF(T+T_1-T_0) \\
\ &\leq \ \E_0(T) + \MM(T_1) - \beta\MM(S_1) + n\FF(T-T_0) + n\MM(T_1) \\
\ &\leq \ \E_0(T) - \beta\MM(S_1) + n\FF(T-T_0) + (n+1)(1+4\frac{r}{\radM}) b \MM(S_1)^2\, ,
\end{align*}
and thus $\beta\MM(S_1) - b(n+1)(1+4\frac{r}{\radM})\MM(S_1)^2\leq 0$. We hence find that either $\MM(S_1)=0$ or that $\MM(S_1)\geq \beta /(3b(n+1))$.

Now, let $x_0$ be a point of rectifiability of $S$ and $r\leq \beta /(6b(n+1))$. 
Assume that $\mu_S(B_r(x_0))<2r$. Then, since 
\begin{align*}
\int_0^r \mu_S(\partial B_s(x_0)) \dx s \leq \mu_S(B_r(x_0)) < 2r\, ,
\end{align*}
we can invoke Theorem 5.7 of \cite{Fleming1966} to deduce that there exists a set of positive measure $I\subset [0,r]$ such that $\mu_S(\partial B_s(x_0))<2$ for all $s\in I$. 
Thus, we can find radii $s<r$ such that $\MM(\partial(S\restr B_s(x_0)))\leq 1$. 
\dsg{But since $\partial\Gamma=0$, it also holds that $\partial S=\partial\partial T - \partial\Gamma=0$, so $S_1\defi S\restr B_s(x_0)$ cannot have just one end}.
We conclude that $\partial S_1 = 0$. 
In addition $\MM(S_1)<2r$ by assumption. 
Hence, we have $\MM(S_1)<\beta/(3b(n+1))$ and the above calculation shows that necessarily $\MM(S_1)=0$. 
In particular, $x_0$ is not in the support of $S$ which is a contradiction. 

Let us conclude now that $S$ is indeed a closed set. Let $\rect(S)$ be the rectifiability set of $S$. 
Since $S$ has coefficients in a finite group, it is rectifiable \cite{White1999a} with $\mu_S=\mathcal{H}^1\restr\rect(S)$. 
Now, take a sequence $x_k\in\rect(S)$ and assume $x_k\rightarrow x$. 
By the above reasoning it holds $\mu_S(B_r(x_k))\geq 2r$ for all $r\leq \beta/(6b(n+1))$ and in the limit $k\rightarrow\infty$ also $\mu_S(B_r(x))\geq 2r$. 
\dsg{It follows from Theorem 2.56 in \cite{Ambrosio2000} that $\mu_S\geq \mathcal{H}^1\restr \overline{\rect(S)}$ and we conclude} that $\mathcal{H}^1(\supp(S)\setminus\rect(S))=0$.
\end{proof}

After having established this basic property of $S$, we can state a first regularity result:

\begin{lemma}\label{lem:reg_T_-_S_C1,1/2}
The flat chain $S$ is supported on a finite union of closed $C^{1,\frac12}-$curves.
\end{lemma}

\begin{proof}
Our goal is to prove that $S$ is an almost minimizer of the length functional $\MM$ and apply Theorem 3.8 in \cite{Morgan1994a} to deduce $C^{1,\frac12}-$regularity. 
 
Let $x_0\in \Omega$ and $\overline{r}\in(0,\frac12\radM)$ such that $B_{\overline{r}}(x_0)\subset\Omega$. 
Let $r\in (0,\overline{r})$.
Consider $T'\in\flatch^2$ with $\supp(T-T')\subset B_{r}(x_0)\subset\Omega$. 
For almost every $r\in(0,\overline{r})$, it holds that $S_r\defi S\restr B_r(x_0)$ is a flat chain with boundary $\partial S_r = S\restr \partial B_r(x_0)$. 
In this case, $S_r'\defi\partial T'\restr B_r(x_0)$ has the same boundary. 
Hence, the flat chain $A\defi S_r+S_r' = \partial T + \partial T'$ is a cycle, i.e. verifies $\partial A=0$ and is supported inside $B_{r}(x_0)$.
We can construct the cone $C'$ with vertex $x_0$ over $A$. Then, $\partial C'=A$ and $\MM(C')\leq c r \MM(A)$. 
Now, we distinguish two cases: It holds either $\MM(S_r)\leq \MM(S_r')$ (which is enough for our conclusion as we will see below) or $\MM(S_r)\geq \MM(S_r')$ and hence $\MM(A)\leq 2\MM(S_r)$.
Comparing $T$ to $T+C'$ and by minimality of $T$ we get that
\begin{align*}
\beta\MM(S_r) 
\ &\leq \ \beta\MM(S_r') + (n+1)\MM(C')
\ \leq \ \beta\MM(S_r') + 2c(n+1)r\MM(S_r)\, .
\end{align*}
For $r$ small enough we conclude that 
\begin{align}\label{lem:reg_T_-_S_C1,1/2:eq_r}
\MM(S_r)
\ &\leq \ \left(1 + \frac{4c(n+1)}{\beta}\: r\right) \MM(S_r')\, .
\end{align}
In case $T'$ is not entirely contained in $\Omega$, we need to project those parts of $T'$ and of the boundary $S_r'$ onto $\M$. 
Since we assumed $r<\overline{r}\leq\frac12\radM$, the Lipschitz constant of the projection can be estimated by $1+4\frac{r}{\radM}$, i.e.\ our analysis and in particular \eqref{lem:reg_T_-_S_C1,1/2:eq_r} holds true if we replace $\MM(S_r')$ by $(1+4\frac{r}{\radM})\MM(S_r')$.
\dsg{If we now consider $x_0\in\M$, we can carry out the same construction, projecting all objects onto $\overline{\Omega}$. 
Since the projection onto $\overline{\Omega}$ has a Lipschitz constant $1+O(r)$, the estimate \eqref{lem:reg_T_-_S_C1,1/2:eq_r} holds with a bigger constant in front of $r$.}
This shows that there exists a constant $C=C_{n,\beta,\radM}>0$ such that $S$ is $(\MM,C r,\overline{r})-$minimal in the sense of Almgren.
Together with Lemma~\ref{lem:reg_T_-_S_closed}, \eqref{lem:reg_T_-_S_C1,1/2:eq_r} allows us to apply Theorem 3.8 in \cite{Morgan1994a} which gives the $C^{1,1/2}-$regularity and the decomposition of $\supp(S)$ into a finite union of curves, possibly meeting at triple points.
Finally, since our flat chains take values only in $\pi_1(\N)=\{0,1\}$, we can exclude the existence of triple points since they would create boundary. Hence, $S$ is a union of curves.
\end{proof}

The regularity of $S$ in Lemma~\ref{lem:reg_T_-_S_C1,1/2} is not optimal. The following Lemma provides us with the smoothness we announced in Proposition~\ref{prop:low_reg_T}:

\begin{lemma}\label{lem:reg_T_-_S_C1,1}
The flat chain $S$ is supported on a finite union of closed $C^{1,1}-$curves. In particular, the curvature of $S$ is bounded.
\end{lemma}

\begin{proof}
Let $x_0\in\supp(S)$.
\dsg{Assume first that $x_0\in\Omega$} and take $r>0$ such that $B_{r}(x_0)\subset\Omega$ and $\mu_S(\partial B_{r}(x_0))=2$. 
Let $\{x_1,x_2\}\defi\supp(S)\cap\partial B_{r}(x_0)$ and define $S_{r}\defi S\restr B_{r}(x_0)$.
We compare $S_r$ (and $T\restr B_r(x_0)$) to two competitors. 

The first one is the geodesic segment $S_g$ joining $x_1$ and $x_2$ in $\partial B_r(x_0)$. For the corresponding $T_g$ we use a piece of $\partial B_r(x_0)$ between $T\restr (\partial B_r(x_0))$ and $S_g$. By minimality of $S_r$ we find
\begin{align}\label{lem:reg_T_-_S_C1,1:competitor1}
\beta\MM(S_r) \ \leq \ 2\pi r \left( \beta + 4(n+1)r \right)\, .
\end{align}

Our second competitor is the flat chain supported on the straight line segment joining $x_1$ to $x_2$ which we call $S'$.
Then $S'+S_r$ is supported in $B_r(x_0)$ and is closed, i.e.\ $\partial(S'+S_r)=0$.
By the construction (7.6) in \cite{Fleming1966}, we get the existence of a flat chain $T'\in\flatch^2$ supported in $\Omega$ and a constant $b>0$ (depending only on the dimensions of the flat chains and the ambient space) such that $\partial T' = S'+S_r$ and $\MM(T')\leq b(\MM(S') + \MM(S_r))^2$. 
Since $x_0\in\supp(S)$ it also holds that $\MM(S_r)\geq 2r$. This, together with the minimality of $S_r$ and \eqref{lem:reg_T_-_S_C1,1:competitor1} implies that
\begin{align}\label{lem:reg_T_-_S_C1,1:competitor2}
2\beta r
\ \leq \ \beta\MM(S_r) 
\ &\leq \ \beta\MM(S') + b(n+1)(\MM(S') + \MM(S_r))^2 \nonumber \\
\ &\leq \ \beta\MM(S') + b(n+1)\left(\MM(S') + 2\pi r \left( 1 + \frac{4(n+1)}{\beta} r \right)\right)^2 \\
\ &\leq \ \beta\MM(S') + C_1 r^2\, , \nonumber
\end{align}
for $C_1 = 2(2+2\pi)^2 b(n+1)$ and $r$ small enough. 
Hence, \eqref{lem:reg_T_-_S_C1,1:competitor2} implies $(2-(C_1/\beta) r)r\leq \MM(S')$.
If we now choose $r$ even smaller to assure $r\leq r_1\defi (C_1)^{-1}\beta$, one gets even $\MM(S')\geq r$, i.e.\ the points $x_1$ and $x_2$ must not be too close.

Our goal is now to show that $S_r$ is in fact close to $S'$ and that $S'$ is almost a diameter of $B_r(x_0)$, in the sense that $S_r$ lies in a small neighbourhood of $S'$ and the distance between $x_0$ and $S'$ is of order $r^2$.
Let's denote $\ell\defi\MM(S')=|x_2-x_1|$.
Suppose $\MM(S_r)\leq \ell + \alpha$ for a $\alpha>0$ and let $\rho>0$ be the smallest positive number such that $S_r$ lies within a $\rho-$neighbourhood of $S'$. Then, $\MM(S_r)\geq \sqrt{\ell^2 + 4\rho^2}$ and hence $\ell^2+4\rho^2\leq \MM(S_r)\leq (\ell + \alpha)^2$ which yields the bound
\begin{align}\label{lem:reg_T_-_S_C1,1:iteration_rho}
\rho \ \leq \ \sqrt{\frac{\ell\alpha}{2} + \frac{\alpha^2}{4}} \ \leq \ \sqrt{2 r \alpha}\, ,
\end{align}
provided $\alpha\leq 4r$ and $\ell\leq 2r$.
Applying this result to our case where $\alpha = \beta^{-1}C_1r^2$, we get $S_r$ is contained in a neighbourhood of $S'$ of radius $\rho\leq \sqrt{2\beta^{-1}C_1r^3}$.

In addition, if $S_r$ is supported in a $\rho-$cylinder around $S'$, there exists a $T_\rho\in\flatch^2$ and a constant $c$ (depending only on the space dimension) such that $\MM(T_\rho)\leq c \rho \MM(S_r)$ and $\partial T_\rho = S' + S_r$. 
This implies that $\MM(S_r) \leq \ell + \beta^{-1}(n+1)c\rho\MM(S_r)$. 
Previously, we have also shown that $\MM(S_r)\leq \ell + \beta^{-1}C_1 r^2 \leq 3r$, leading to
\begin{align}\label{lem:reg_T_-_S_C1,1:iteration_exponent}
\MM(S_r) \ \leq \ \ell + C_2 \rho r\, ,\quad\text{where } C_2 = 3c\frac{n+1}{\beta}\, .
\end{align}

Now, we want to iterate this procedure. Let $\alpha_0\defi \beta^{-1}C_1r^2$ as start of our induction. 
\begin{enumerate}
\item Knowing that $\MM(S_r)\leq \ell + \alpha_k$ (either by \eqref{lem:reg_T_-_S_C1,1:competitor2} or by induction hypothesis) and by \eqref{lem:reg_T_-_S_C1,1:iteration_rho} we can deduce that $S_r$ lies in a $\rho_k-$neighbourhood of $S'$, for $\rho_k = \sqrt{2r\alpha_k}$. 
\item Since $S_r$ lies in a $\rho_k-$neighbourhood of $S'$, one can use \eqref{lem:reg_T_-_S_C1,1:iteration_exponent} with $\rho=\rho_k$ to obtain $\MM(S_r)\leq \ell + \alpha_{k+1}$, where $\alpha_{k+1}\defi C_2 r \rho_k$.
\end{enumerate}
Throughout this iteration, $\alpha_k$ and $\rho_k$ verify $\rho_{k+1} = \sqrt{2r\alpha_{k+1}} = \sqrt{2C_2 \rho_k} \: r$.
Thus, $\rho_k$ converges to $2C_2 r^2$ in the limit $k\rightarrow \infty$.
We can conclude that the distance between a point in $S_r$ and $S'$ is at most $2C_2 r^2$.
In particular, since $x_0\in\supp(S_r)$, it holds that $\dist(x_0,\supp(S'))$ is of order $r^2$ which shows that the line $S'$ is close to being a diameter.

Let us turn now to the assertion of the lemma. 
For $x_0\in\supp(S)$ and $r>0$ chosen small enough, we denote $\tau_r(x_0)$ the vector $\frac{x_2-x_1}{\Vert x_2-x_1\Vert}$, where $x_1,x_2$ are constructed as before. 
By our previous calculations, we know that the corresponding points for $\frac{r}{2}$ are at most at distance $2C_2 r^2$ from the line connecting $x_1$ and $x_2$ which gives $\Vert\tau_r(x_0) - \tau_{\frac r2}(x_0)\Vert \leq C_3 r$. 
This shows that the limit $\tau(x_0) = \lim_{r\rightarrow 0} \tau_r(x_0)$ exists and that $\Vert\tau_r(x_0) - \tau(x_0)\Vert \leq 2C_3 r$. 
The triangle inequality then yields the existence of another constant $C_4>0$, depending on $\beta$ and $n$, such that for $x,y\in\supp(S)$ with $|x-y|=:r$ small enough we have $\Vert\tau(x) - \tau(y)\Vert \leq C_4 r$.

\dsg{
Now if $x_0\in \M$, we observe that as $r>0$,
the projection onto $\overline{\Omega}$ has a Lipschitz constant which converges to $1$ as $r\to 0$. 
We can reproduce the same construction, projecting back all the competitors onto $\overline{\Omega}$ and we end up with the same estimate, up to an error which vanishes as $r\to 0$. 
In particular, the curvature of $S$ is bounded by $C_4$.
}
\end{proof}

Having reached the optimal regularity for $S$, we now turn to the properties of $T$. 

\begin{lemma}\label{lem:reg_T_-_interior_T}
The flat chain $T\restr\Omega$ is supported on a hypersurface of class $C^1$ up to the boundary.
\end{lemma}

\begin{proof}
We first discuss the regularity in the interior of $T\restr\Omega$. 
Let $x_0\in \Omega$, $r>0$ such that $B_r(x_0)\cap\supp(T\restr\Omega)\neq\emptyset$ and consider a variation $T'$ of $T$ in $B_r(x_0)$. 
Then, by minimality of $T$ we find
\begin{align*}
\MM(T) 
\ &\leq \ \MM(T') \ + \ n\FF(T-T') 
\ \leq \ \MM(T') \ + \ \frac{4}{3}\pi n r^3\, . 
\end{align*}
We can then apply the result of Taylor \cite{Taylor1976}, or more general Theorem 1.15 in \cite{David2010} to obtain $C^{1,\alpha}-$regularity in $\Omega$, for some $\alpha>0$. 

For the regularity up to the boundary we want to apply Theorem 31.1 in \cite{David2019}. In order to do this we need to show that on a certain scale, the boundary $S$ is close to a straight line and $T$ is almost flat. 

Take a point of rectifiability $x_0\in S$. 
We define a blow-up sequence $r_k\searrow 0$. 
Since $S$ is supported by $C^{1,1}-$curves, a blow up of $S$ converges to a straight line.
We claim that a blow up of $T$ converges to a limit $T_0$ which is a half plane. 
For this, we use the minimality of $T$ to deduce for $r>0$ small enough that
\begin{align*}
\MM(T\restr B_r(x_0)) + 2\beta r 
\ &\leq \ \MM(T\restr B_r(x_0)) + \beta\MM((\partial T)\restr B_r(x_0)) \\
\ &\leq \ \MM(\mathrm{cone}(T\restr \partial B_r(x_0))) + \beta\MM(\mathrm{cone}((\partial T)\restr \partial B_r(x_0))) \\
\ &\leq \ \frac{r}{2}\MM(T\restr \partial B_r(x_0)) + \beta r \MM((\partial T)\restr \partial B_r(x_0)) \\
\ &= \ \frac{r}{2}\MM(T\restr \partial B_r(x_0)) + 2\beta r\, .
\end{align*}
This implies that $\MM(T\restr B_r(x_0))/r^2$ is monotonically increasing and thus admits a unique limit $d$. 
We define $T_{r_k}=(T-x_0)/r_k$ and by monotonicity we get for $s_1<s_2$ that $\MM(T_{r_k}\restr B_{s_1})/s_1^2 \leq \MM(T_{r_k}\restr B_{s_2})/s_2^2$. For $r_k\rightarrow 0$ both sides converge to the same limit $\pi d$. 
But this means that $\MM(T_{0}\restr B_{s_1})/s_1^2 = \MM(T_{0}\restr B_{s_2})/s_2^2$ , i.e.\ $T_0$ is a cone and hence a half-plane. 
Since a half plane has density $\frac12$, we find $d=\frac{1}{2}$. 
In particular, we have for $k$ large enough
\begin{align*}
\MM\left( \frac{T_{r_k}-x_0}{r_k}\restr B_1 \right) 
\ &= \ \frac{\pi}{2} + o(1)\, ,
\end{align*}
from which it follows that condition (31.6) in \cite{David2019} holds and thus we can apply Theorem 31.1 on a length scale $R\leq r_k$.
We remark that by convergence in the flat norm, following \cite{Matveev2016}, we also verify the condition (31.4) of Theorem 31.1 in \cite{David2019}.
By compactness of the boundary $(\partial T)\restr\Omega$, we find a finite cover with balls of uniformly positive radius to which we can apply Theorem 31.1. This allows us to conclude.
\end{proof}

\begin{proof}[Proof of Proposition~\ref{prop:low_reg_T}]
We have already established the existence of a minimizer of \eqref{eq:limit_energy_T_regularity}. 
The convergence $\FF(T_n - T_0) \rightarrow 0$ for $n\rightarrow\infty$ is obvious since $n\:\FF(T_n - T_0)\leq \E_0(T_0)<\infty$ for all $n\in\mathbb{N}$.

The regularity of $T_n$ follows from Lemma~\ref{lem:reg_T_-_S_C1,1} and Lemma~\ref{lem:reg_T_-_interior_T}. 
\end{proof}

\subsection{Construction of the recovery sequence}

In this section we will use the approximation of $T$ given by the minimizer of \eqref{eq:limit_energy_T_regularity} to construct our recovery sequence.
First we establish the following Proposition which yields additional control over $\partial(T\restr\M)\setminus\partial T$ and its boundary which will be necessary for the final construction in Proposition~\ref{prop:upper_bound:recov_Q}.

\begin{proposition}\label{prop:upper_bound:approx_recov}
Let $T\subset\overline{\Omega}$ be a flat $2-$chain of finite mass and $S\subset\overline{\Omega}$ be a flat $1-$chain of finite mass such that $\partial S = 0$ and $\partial T = S + \Gamma$. 
Then, there exist finite mass flat chains $T_n\in\flatch^2$ of class $\mathrm{Lip}$ up to the boundary and $S_n\in\flatch^1$ of class $C^{1,1}$ such that
\begin{enumerate}
\item $\partial S_n = 0$ and $\partial T_n = S_n + \Gamma$,
\item $\FF(T_n-T)\rightarrow 0$ and $\E_0(T_n)\rightarrow \E_0(T)$  as $n\rightarrow\infty$,
\item and there exists a constant $C_n>0$ such that $\MM(\partial (T_n\restr\M)\setminus\partial T_n )\leq C_n$ and $\MM(\partial(\partial (T_n\restr\M)\setminus\partial T_n ))\leq C_n$. 
\end{enumerate}
\end{proposition}

Essentially, the first two parts of Proposition are proved by Proposition~\ref{prop:low_reg_T}, saying that the minimizer $T_n$ of \eqref{eq:limit_energy_T_regularity} fulfils our claims.
It remains to prove the last assertion i.e.\ that we can modify $T_n$ to control the length of the set where the  $T_n$ attaches to $\M$. For this, we need the following average argument stating that we can find radii $r$ such that the corresponding sets $T_n\restr\M_r$, for $\M_r\defi\{x\in\Omega\sd\dist(x,\M)=r\}$, are of finite length.

\begin{lemma}\label{lem:reg_T_-_finite_length_on_M}
Let $T_n$ be as constructed in the previous subsection.
There exist a constant $c>0$ and a radius $r\in (0,\frac12\radM)$ such that 
\begin{align}
\MM(T_n\restr\M_r)
\ &\leq \ \frac{4c\MM(T_n)}{\radM}\, .
\end{align}
\end{lemma}

\begin{proof}
Assume that $\MM(T_n\restr\M_r) > \frac{4c\MM(T_n)}{\radM}$ for all $r\in (0,\frac12\radM)$ and some $c>0$. This implies
\begin{align*}
\int_0^{\radM/2} \mu_{T_n}(\M_r)\dx r 
\ > \ 2c\MM(T_n)\, .
\end{align*}
Now, there exists a constant $c>0$ such that $\int_0^{\radM/2} \mu_{T_n}(\M_r)\dx r \leq c \MM(T_n)$ (see (5.7) in \cite{Fleming1966}). Hence, the lemma is proved. 
\end{proof}

Now, we can modify $T_n$ by replacing a small part close to $\M$ by a projection to control the boundary of $T_n\restr\M$ which is not included in $S$.

\ds{
\begin{proof}[Proof of Proposition~\ref{prop:upper_bound:approx_recov}]
We construct $T_n$ as in Proposition~\ref{prop:low_reg_T}.
To ensure the additional estimate, we choose a radius $r$ and a slice $\M_r$ as in Lemma~\ref{lem:reg_T_-_finite_length_on_M}.
With the same argument as in Lemma~\ref{lem:reg_T_-_finite_length_on_M} for $S_n$ one can choose $r\in (0,\frac12\radM)$ for which additionally $\MM(S_n\restr\M_r)$ is finite. 
Let $\Pi:\Omega_{\radM}\rightarrow\M$ be the projection onto $\M$.
We define $\Phi:\M_r\times[0,1]\rightarrow \overline{\Omega}$ by $\Phi(x,t)\defi (1-t)x + t\Pi x$.
Then, by \cite[Sec. 2.7]{Federer1960}, \cite[Cor. 2.10.11]{Federer1996}, $ \MM(\Phi_\#(T_n\restr\M_r\times [0,1])) \leq C r\MM(T_n\restr\M_r)$ and also $\MM(\Pi_\#(T_n\restr\M_r))\leq C\MM(T_n\restr\M_r)$.
Again by the same argument, we get $\MM(\partial\Pi_\#(T_n\restr\M_r))\leq C\MM(\partial(T_n\restr\M_r))$.
This procedure can be applied to almost every $r\in (0,\frac12\radM)$, in particular, we can choose a sequence $r_n\rightarrow 0$ as $n\rightarrow \infty$. 
Replacing $T_n$ close to $\M$ with these projections, we get the desired estimates.

The convergence of the energy $\E_0(T_n)$ to $\E_0(T)$ is a consequence of the convergence statements in Proposition~\ref{prop:low_reg_T} and the fact that $T_n\restr\M$ approaches $T\restr\M$.
\end{proof}

The recovery sequence $\Qex$ for our problem will be constructed around the regularized sequence of $T$. 
The gained regularity permits us to define $\Qex$ directly and without the need to write $T$ as a complex and ``glue'' together the parts of $\Qex$ on each simplex. 

%
%

\begin{proposition}\label{prop:upper_bound:recov_Q}
There exists a recovery sequence $\Qex$ for the problem \eqref{thm:main:limsup}.
\end{proposition}

The construction relies on the approximation and regularisation made in the previous subsection.
We will construct $\Qex$ step by step: 
The straightforward parts are the profile on $F$ and $\M\setminus F$ away from $\partial F$, as well as the transition across $T$. 
We recall the notation $F$ from \eqref{def:set_F}
that if we write $\mu_{T\restr\M} = \chi_G \mathcal{H}^2\restr\M$ for a measurable set $G\subset\M$, then $F$ is given by
\begin{align*}
F \defi \{\omega\in\M\setminus G\sd \nu(\omega)\cdot\ee_3>0\} \cup \{\omega\in\M\cap G\sd\nu(\omega)\cdot\ee_3\leq 0\}\, .
\end{align*}
In order to be compatible with the latter, we have to adjust the construction made in \cite{ACS2021} for the singular line $S$. 
The profile of the part of $S$ that approaches the surface $\M$ can be chosen as in \cite{ACS2021}. 
Last, we need to connect $\partial F\setminus S$ to the profile of $T$ already constructed. 
This last part is a bit subtle since $\partial F\setminus S$ does not appear in the energy. 
The control we obtained in Proposition~\ref{prop:upper_bound:approx_recov} is artificial and indeed we do not control the length of $\partial F\setminus S$. 
\ds{
Another problem for our construction originates in the fact that the optimal profile $\nnopt$ needs to be accompanied by a horizontal vectorfield (denoted $v$ in the proof) to form a director field. 
Far from $\M$, this can be chosen constant, but on $\M$ the director must match the normal.
In order to be able to construct a regular vectorfield $v$, we need to ``cut holes'' into $T$ in regions on $\M$ where $\nu=\pm\ee_3$ that are also covered by $T$.
Similarly, we also cut out some pieces of $T$ close to $\M$ as the transition of $v$ from boundary data to constant far from $\M$ would result in infinite energy.
} 

\begin{proof}
From now on, we fix $n$ large. In view of Proposition~\ref{prop:upper_bound:approx_recov} we can find a constant $0 < C_n < +\infty$ such that $\MM(\partial(T_n\restr\M)),\MM(\partial(\partial(T_n\restr\M)\setminus \partial T_n))\leq C_n$ and that the curvature of $S_n$ is also bounded by $C_n$.

Furthermore, whenever this does not lead to confusion, we drop the subscript parameters $\eta,\xi$ and $n$ in order to make the construction more readable.

\ds{
\textit{Step 0: Modification of $T$.}
We start by noting that the construction ``close'' to the particle surface $\M$ will take place in a neighbourhood of size $\eta$. 
More precisely, let $M>1$. We will focus on on the $3\radP$ neighbourhood of $\M$ in $\Omega$, denoted by $\M_{3\radP}$.
By taking $3\radP<\tfrac12\radM$, we can ensure that the extension $\overline{\nu}=\nu\circ\Pi_\M$ exists on $\M_{3\radP}$ as before. 

Throughout the construction, we make the assumption that $\mathcal{H}^0(\supp(T)\cap \overline{\nu}^{-1}(\pm\ee_3))$ is a finite set, say $\{x_1,...,x_{N}\}$ for some $N\in\mathbb{N}$.
We furthermore assume that $x_i\notin S$.
It is a simplification and we will explain in Step 7 how to adapt the proof for the general case.

If all points $x_i$ for $i=1,...,N$ lie inside $\Omega$, we can choose $0<\epsilon'<\min\{\tfrac{\epsilon}{N} , \tfrac12 \dist(x_i,\M)\sd i=1,\ldots,N \}$. 
Then by slicing it holds that
\begin{align}\label{prop:upper_bound:recov_Q:slicing_for_holes}
\int_{\epsilon'/2}^{\epsilon'} \MM(\partial(T\setminus B_s(x_i)) \dx s
\ &\leq \ 
\MM(T\restr B_{\epsilon'}(x))
\ \leq \ C (\epsilon')^2 \, ,
\end{align}
where the last inequality follows from the $C^1$ regularity of $T$ but may also be easily deduced from the minimality of $T$, similarly to the case $x_i\in M$ discussed below.
From \eqref{prop:upper_bound:recov_Q:slicing_for_holes} it follows that we can choose a radius $s_i\in (\tfrac{\epsilon'}{2},\epsilon')$ such that $\MM(\partial(T\setminus B_{s_i}(x_i))\leq C \epsilon'$.
We define $\tilde{T} \defi T\setminus \bigcup_{i=1}^N B_{s_i}(x_i)$. 
Note that $\partial\tilde{T}$ differs from $\partial T$ since we introduced boundaries coming from $\partial B_{s_i}(x_i)$.
Since the lengths of those boundaries are controlled by $C\epsilon'$, the energy of $\tilde{T}$ can be estimated as $\E_0(\tilde{T})\leq \E_0(T) + C N \epsilon' \leq \E_0(T) + C \epsilon$. 

In the case where one (or more) point $x_i$ lies on $\M$, we also ``cut a hole'' into $T_n$ around $x_i$ in the following way.
By minimality of $T_n$ we compare the energy in \eqref{eq:limit_energy_T_regularity} with $T'$ where the $T_n\restr B_{\epsilon'}(x_i)$ has been pushed out onto the boundary of $B_{\epsilon'}(x_i)\cap\Omega$. 
The newly created surface has surface at most $4\pi(\epsilon')^2$ and the additional term in the flat distance is of order $n(\epsilon')^3$ so that the difference in energy between $T$ and $T'$ is of order $(\epsilon')^2$.
Since $\nu_3(x_i)=1$ it holds $\nu_3(x)\geq \tfrac12$ for $x$ in a neighbourhood of each $x_i$, and therefore $|\MM(T) - \MM(T')|\leq C (\epsilon')^2$. 
We can proceed as before by slicing to select a radius $s_i$ and then modify $T'$ to obtain a $\tilde{T}$.

With this procedure we obtain a modified flat chain $\tilde{T}$ with     does not intersect $\{\overline{\nu} = \pm \ee_3\}$ in $\M_{3\radP}$ and with energy $\E_0(\tilde{T})\leq \E_0(T) + C\epsilon$.
We will continue to work with $\tilde{T}$ in the next steps and omit the tilde in our notation. 

} 

\textit{Step 1: Adaptation of the optimal profile.}
The goal of this step is to construct a one dimensional profile close to the optimal one in Lemma~\ref{lem:radial_turning}, but where the transition takes place on a finite length and which gives the correct energy density \eqref{lem:radial_turning:cos} for $\eta\rightarrow 0$.
To this goal, we use the ``artificial'' length scale $\radP$ introduced in Step 0 and define
\begin{align}\label{prop:upper_bound:recov_Q:def_quasi_optimal_profile_1}
\Phi^{\pm}_\eta(t,\theta,v) 
\ := \ 
\sstar(\nn^{\pm}(t/\eta,\theta)\otimes\nn^{\pm}(t/\eta,\theta) - \frac13\id) \quad\text{ for } t\in [0,\radP]\, ,
\end{align}
with $\nn^\pm=(\sqrt{1-\nnopt^2}(v_1,v_2),\pm\nnopt)$, where $\nnopt(t,\theta)$ is the optimal profile from \eqref{lem:radial_turning:nn3_opt} and $(v_1,v_2)\in\mathbb{S}^1$. 
We will later take the limit $\eta\rightarrow0$ and $M\rightarrow\infty$ such that $\nn^\pm(\radP/\eta)\rightarrow\pm \ee_3$. 
Now we define $\Phi^{\pm}_\eta(t,\theta,v) $ for $t$ in the interval $[\radP,2\radP]$ to be
\begin{align}\label{prop:upper_bound:recov_Q:def_quasi_optimal_profile_2}
\Phi^{\pm}_\eta(t,\theta,v) 
\ := \ 
\sstar(\nn^{\pm}_\text{inter}(t,\theta)\otimes\nn^{\pm}_\text{inter}(t,\theta) - \frac13\id)\quad\text{ for } t\in (\radP,2\radP]\, ,
\end{align}
where $\nn^{\pm}_\text{inter}(t,\theta)$ is the unit vector interpolating between $\nn^{\pm}(M,\theta)$ and $\pm\ee_3$, that is, for $\alpha = \pm\arcsin(\nn^\pm(M,\theta)\cdot\ee_3)$ 
\begin{align*}
\nn^{\pm}_\text{inter}(t,\theta)
\ = \ 
\cos\left(\alpha\tfrac{2\radP-t}{\radP} \pm \tfrac{\pi}{2}\tfrac{t-\radP}{\radP}\right)
\hat{\nn}^\pm(M,\theta)
+
\sin\left(\alpha\tfrac{2\radP-t}{\radP} \pm \tfrac{\pi}{2}\tfrac{t-\radP}{\radP}\right)\ee_3
\end{align*}
for
\begin{align*}
\hat{\nn}^\pm(M,\theta)
\ &= \ \frac{\nn^\pm(M,\theta) - \nn_3^\pm(M,\theta)\ee_3}{|\nn^\pm(M,\theta) - \nn_3^\pm(M,\theta)\ee_3|} \, .
\end{align*}
This definition achieves the transition of $\Phi^{\pm}_\eta(\radP,\theta,v)$ to $Q_\infty$ in a way that the bulk potential $f$ vanishes. 
Finally, we define the transition between $\Qinfty$ and $\Qexinfty$ to take place for $t\in [2\radP,3\radP]$ via a linear interpolation
\begin{align}\label{prop:upper_bound:recov_Q:def_quasi_optimal_profile_3}
\Phi^{\pm}_\eta(t,\theta,v) 
\ := \ 
(3\radP-t)\Qinfty + (t-2\radP)\Qexinfty \quad\text{ for } t\in (2\radP,3\radP]\, .
\end{align}

To finish the definition of the profile $\Phi_\eta^\pm$, we have to construct the vectorfield $v:\Omega\rightarrow\mathbb{R}^2$ with modulus $1$.
In order for $\Phi_\eta^\pm$ to meet the boundary conditions, we require that $v = \nu' := \tfrac{(\nu_1,\nu_2)}{\sqrt{\nu_1^2+\nu_2^2}}$ on $\M$.
We define $v$ as follows:
Let $\ee\in\mathbb{S}^1$ and let $\overline{\nu'}$ denote a radial extension of $\nu'$ into a neighbourhood of $\M$ in $\Omega$, i.e.\ $\overline{\nu'} = \nu'\circ\Pi_\M$.
We introduce a monotone cut-off function $\varphi:\Omega\rightarrow [0,1]$ depending only on the distance to $\M$ which satisfies $\varphi=1$ in a $\tfrac{\radM}{4}-$neighbourhood of $\M$ and $\varphi=0$ at distance greater than $\tfrac{\radM}{2}$. Note that in the region where $\varphi\neq 0$ the function $\overline{\nu'}$ is defined. 
We can furthermore assume that $\varphi$ is Lipschitz, i.e.\ $|\nabla\varphi|\leq C$. 
We define 
\begin{align*}
v 
\ := \ \frac{\varphi \overline{\nu'} + (1-\varphi)\ee}{|\varphi \overline{\nu'} + (1-\varphi)\ee|} \, .
\end{align*}
This vector field is well-defined and $\mathbb{S}^1$ valued except for the set on which $\varphi(\overline{\nu'}-\ee) = -\ee$.
For a generic choice of $\ee$ and $\varphi$, this set is $1-$dimensional and we can assume that it intersects $T$ only in finitely many points.

For the construction of the profile around $T$, we can cut out small parts of $T$ around these points as we did in Step 0, so that $v$ is a Lipschitz $\mathbb{S}^1-$valued vector field in a neighbourhood of $T$ and coincides with $\nu'$ on $\M$ and with a constant vector $\ee$ far from $\M$.
Since the removal of parts of $T$ creates new boundary components, as in Step 0, this procedure introduces a further error of order $\radv$, where $\radv>0$ is the length scale of the holes.

All in all, we end up with a slightly modified chain $T$ (of approximately same energy) and a ``horizontal'' vector field $v$ which is $C^1$  in a neighbourhood of the support of $T$.

\textit{Step 2: Construction on $F$ and $F^c$.}
Let $\omega\in F_{3\radP}\defi\{\omega\in F\sd\dist(\omega,\partial F)\geq 3\radP\}\subset\M$ and let $0\leq r<3\radP \leq \frac12\radM$. 
For $v$ defined in Step 1 we set
\begin{align}
\label{prop:upper_bound:recov_Q:def_Q_on_F}
\Qex(\omega + r\nu(\omega))
\ := \ \Phi_\eta^{+}(r,\theta,v(\omega + r\nu(\omega))) \quad\text{where }\theta = \arccos(\nu(\omega)\cdot\ee_3)
\end{align}
and as before $\nu$ is the normal vector field of $\M$. 
We note that with this definition $\Qex(\omega) = Q_b(\omega)$. 

It remains to calculate the energy contribution coming from $F_{3\radP,R}$, where $F_{3\radP,R} := \{x\in\Omega\sd x = \omega+r\nu(\omega),\, \omega\in F_{3\radP},\, r\in [0,R]\}$ for $R>0$.
It holds that
\begin{align*}
\eta\:&\Eex(\Qex,F_{3\radP,3\radP}) \\
\ &= \ \int_{F_{3\radP}} \int_0^{3\radP} \left( \frac{\eta}{2}|\nabla\Qex|^2 + \frac{\eta}{\xi^2}f(\Qex) + \frac{1}{\eta}g(\Qex) + \eta C_0 \right)\prod_{i=1}^2(1+r\kappa_i) \dx r \dx \omega \, ,
\end{align*}
where $\kappa_i$ denotes the principal curvatures of $\M$ as in the previous section.
By definition of $\Phi_\eta^{+}$ it holds that $f(\Qex)=0$ for $r\in [0,2\radP]$.
Furthermore, by Proposition~\ref{prop:f__g_geq}, $C_0\lesssim \xi^2/\eta^4$ and by exponential convergence of $\nnopt$ to $1$ we deduce that
\begin{align*}
\int_{2\radP}^{3\radP}\left| \frac{\eta}{\xi^2}f(\Qex) + \frac{1}{\eta}g(\Qex) + \eta C_0 \right| \dx r
\ &\lesssim \ \frac{\eta}{\xi^2} \frac{\xi^2}{\eta^2} e^{-M} M\eta = M e^{-M}\, .
\end{align*}
We also point out that $\int_{\radP}^{2\radP}\tfrac{1}{\eta} g(\Qex)\dx r \lesssim M e^{-M} $. 
By the construction in Step 0 and Step 1 we can bound the gradient of $v$ \dsr{uniformly for all $x=\omega'+r\nu(\omega')$, where $\dist_\M(\omega',\omega)\geq \epsilon$ for all $\omega\in\M$ with $\nu_3(\omega)= \pm 1$}.
Lemma~\ref{lem:radial_turning} implies that the derivative of $\nnopt$ w.r.t.\ $\theta$ is bounded. 
\dsr{
Around the points $\omega\in\M$ where $\nu_3(\omega)= 1$ and for $3\radP\leq \radM/2$ it holds $v(x) = \overline{\nu'}(x)$. 
The gradient of $\Phi_\eta^{+}$ can be bounded by
\begin{align}\label{prop:upper_bound:recov_Q:estim_nu_pm_1}
|\nabla\Phi_\eta^{+} |^2
\ &\lesssim \
|\nabla \nnopt|^2
+ \Big|\nabla \Big(\sqrt{1-\nnopt^2} \: (v_1,v_2)\Big)\Big|^2
\, .
\end{align}
We point out that the first term in \eqref{prop:upper_bound:recov_Q:estim_nu_pm_1} is easily seen to be bounded since $\nabla \nnopt = 0$ in $\omega$ as $\nu_3= 1$ is a extremal value.
For the second term we recall that $v(x) = \overline{\nu'}(x)$.
A direct calculation using the explicit profile from Lemma~\ref{lem:radial_turning} shows that 
\begin{align*}
\sqrt{1-\nnopt^2(x)} \frac{(\overline{\nu}_1(x),\overline{\nu}_2(x))}{\sqrt{1-\overline{\nu}_3(x)^2}}
\ = \
2\frac{\exp(- \tfrac{c_* \: r}{s_*\: \eta})(\overline{\nu}_1(x),\overline{\nu}_2(x))}  {1+\overline{\nu}_3(x) + (1-\overline{\nu}_3(x)) \exp(-\tfrac{2c_* \: r}{s_*\: \eta})}
\, .
\end{align*}
Since $\overline{\nu}=\nabla\dist(\cdot,\M)$ and $|\nabla \overline{\nu}|=|D^2 \dist(\cdot,\M)|\leq C$, one can see from this representation that the tangential gradient is uniformly bounded and the radial derivative is bounded by $C/\eta$.
Therefore, integrating $\eta|\nabla \Phi_\eta^{+}|^2$ over the $\epsilon-$neighbourhood of the points $\omega\in\M$ with $\nu_3(\omega)=+1$ and $r\leq 3\radP$ leads to the upper bound $CM\epsilon^2$.
Note that the same argument does not work for $\Phi_\eta^+$ near the points where $\nu_3=-1$.
This is due to the fact that in this situation $1-\nnopt^2$ first increases to $1$ before decaying.
Together with the singularity of $v$ at $\nu_3=-1$ this implies that $(1-\nnopt^2)|\nabla v|^2$ is not necessarily integrable. 
This is the reason why we did not attribute those points to $F$ so that $\Qex$ is defined around $\nu_3=-1$ using the profile $\Phi_\eta^{-}$ instead.
} 
The above allows us to estimate
\begin{align*}
\eta\:\Eex(\Qex,F_{3\radP,3\radP})
\ &\leq \ \int_{F_{3\radP}} \left[\int_0^{\radP} \left( \frac{\eta}{2}|\partial_r\Qex|^2 + \frac{1}{\eta}g(\Qex) \right)\prod_{i=1}^2(1+r\kappa_i) \dx r  + C M e^{-M} \right] \dx \omega \\
\ &\leq \ (1+C M e^{-M})\int_{F_{3\radP}} I\left(0,M,\cos(\theta),+1\right) \dx \omega + \dsr{CM\epsilon^2} + o(1)\, .
\end{align*}

Analogously, we can define $\Qex$ on $F^c$ away from $\partial F$ by using $\Phi^{-}$ and estimate its energy.
Note that this construction may already create the part of $T$ 
that attaches to the surface $\M$ in the limit $\eta,\xi\rightarrow 0$.
Indeed, if a point $\omega$ is contained in $F$ although the energy density corresponding to $F^c$ would be lower, the profile constructed passes trough $n_3=0$ within a distance $\radP$ from $\M$ and hence is included in the limiting $T$. 

\textit{Step 3: Construction on $T$.}
Let $x\in T_{\eta}\defi\{x\in\supp(T)\sd \dist(x,\M)>3\radP\text{ and }\dist(x,S)>3\radP\}$. 
For each connected component of $T$ (and thus of $T_\eta$) we can associate a sign depending on the sign of the degree of the singularity line $S$ (if the component of $T$ has such). 
This must be compatible with the part of $T$ that reaches $\M$ and already has been constructed in Step 2. 
The compatibility corresponds to the choice of the signs of $\Phi_\eta^\pm$ and of the distance function, viewing $T_\eta$ as  a boundary, locally.
Assuming that in Step 2 we chose $\Phi^{+}_\eta$ whenever $\dist(\cdot,T_\eta)>0$ and $\Phi^{-}_\eta$ for $\dist(\cdot,T_\eta)<0$, we define
\begin{align*}
\Qexn(x) \ \defi \ \Phi_\eta^{+}(\dist(x,T),\frac{\pi}{2},v(x))\, .
\end{align*}
We recall that $T$ has been modified in such as way that $v$ from Step 1 is Lipschitz in a neighbourhood of $T$ and hence $|\nabla v|$ is bounded.
Writing $T_{\eta,t}\defi\{x\in\Omega\sd\dist(x,T_\eta)=\dist(x,T)\text{ and }\dist(x,T_\eta)\leq t\}$ for $t\geq 0$ we can estimate by Lemma~\ref{lem:radial_turning} and the coarea formula
\begin{align*}
\int_{T_{\eta,3\radP}} &\left[\frac{\eta}{2}|\nabla \Qexn|^2 + \frac{\eta}{\xi^2}f(\Qexn) + \frac{1}{\eta}g(\Qexn) + \eta C_0 \right] \dx x \\
\ &\leq \ \int_{T_{\eta,\radP}} \left[\frac{\eta}{2}|\nabla \Qexn|^2 + \frac{1}{\eta}g(\Qexn)\right] \dx x + C M e^{-M}\:\MM(T)  \\
\ &= \ 2s_*c_*\int_{T_{\eta,\radP}} |\nnopt'(\dist(x,T_n)/\eta)| \dx x  + C M e^{-M}\:\MM(T) \\
\ &= \ 2s_*c_*\int_{0}^{\radP} \int_{T_{\eta,\radP}\cap \{\dist(\cdot,T)=s\}} |\nnopt'(s/\eta)| \dx s + o(1) + C M e^{-M}\:\MM(T) \\
\ &= \ 2s_*c_*\int_{0}^{\radP} \mathcal{H}^2(T_{\eta,\radP}\cap \{\dist(\cdot,T)=s\}) |\nnopt'(s/\eta)| \dx s + o(1) + C M e^{-M}\:\MM(T) \\
\ &\leq \ 2s_*c_*\big( 2\MM(T) + o(1) \big)\int_{0}^{\radP} |\nnopt'(s/\eta)| \dx s + o(1) + C M e^{-M}\:\MM(T) \\
\ &= \ 4s_*c_* |\nnopt(M)| \:\MM(T) + o(1) + C M e^{-M}\:\MM(T)  \, ,
\end{align*}
where we also used that $\mathcal{H}^2(T_{\eta,\radP}\cap\{\dist(\cdot,T)=s\})\rightarrow 2\MM(T)$ for $s\rightarrow 0$.
Note that $|\nnopt(M)|\leq 1$.
Hence, for $\eta,\xi\rightarrow 0$ we end up with
\begin{align*}
\limsup_{\eta,\xi\rightarrow 0}\int_{T_{\eta,3\radP}} &\frac{\eta}{2}|\nabla \Qexn|^2 + \frac{\eta}{\xi^2}f(\Qexn) + \frac{1}{\eta}g(\Qexn) + \eta C_0 \dx x \\
\ &\leq \ 4s_*c_*(1 + C M e^{-M})\MM(T) \, .
\end{align*}

\textit{Step 4: Construction on $S\restr\Omega$.}
Following the notation we used in Step 2 and 3, we introduce the region 
\begin{align}\label{prop:upper_bound:recov_Q:df_region_Sn_Ome}
S_{3\radP} 
\defi 
\{x\in\Omega\sd \dist(x,\M)>3\radP, \,
\exists y\in T\text{ with }&\dist(y,S)\leq 3\radP \\
&\text{ and }\dist(x,T)=\Vert x-y\Vert\leq 3\radP\} \nonumber
\end{align}
around the singular line $S$ (see also Figure~\ref{fig:upper_bound_constr_S_Ome}). 
We will construct $\Qexn$ as follows:
Depending on the sign attributed to the connected component of $T$ in Step 3 or the change between $F$ and $F^c$ in Step 2, we place a singularity of degree $-\frac12$ (resp. $\frac12$) as in Lemma 5.2 in \cite{ACS2021} in the center of $S_{3\radP}$. 
We do so by setting $Q=0$ in a disk of radius $\xi$ (perpendicular to $S$) and oblate $Q$ uniaxial with director field $(\sin(\phi/2),0,\cos(\phi/2))$ on the annulus between the radii $2\xi$ and $\eta$, interpolating linearly in radial direction between these two regions. 
From the circle of radius $\eta$ to the boundary of the region \eqref{prop:upper_bound:recov_Q:df_region_Sn_Ome}, we use the profile $\Phi_\eta^\pm$ to make a transition to $Q_\infty$ along $\nabla\dist(\cdot,T)$. By doing so, we get the compatibility between the construction made for $T$ and $S$. 

More precisely, we define as in \cite{ACS2021}(Lemma 5.2, Step 3, Equation (55))
\begin{align*}
Q_B(r,\phi) 
\ := \ \begin{cases}
0 & r\in [0,\xi) \, , \\
\left(\frac{r}{\xi}-1\right)Q(\phi) & r\in [\xi,2\xi)\, , \\
Q(\phi) & r\in [2\xi,\eta)\, ,
\end{cases}
\end{align*}
where $r\in [0,\eta)$, $\phi\in [0,2\pi)$ and 
\begin{align*}
Q(\phi) 
\ = \ s_*\left(\nn(\phi)\otimes\nn(\phi)-\frac13\id\right) 
\quad\text{with}\quad 
\nn(\phi) \ = \ \begin{pmatrix}
\sin(\phi/2) \\ 0 \\ \cos(\phi/2)
\end{pmatrix}\, .
\end{align*}
We use this to define $\Qex$ on a small $\eta-$neighbourhood of $S$ as follows. 
For $\eta$ small enough, we can assume that the $\eta-$neighbourhood is parametrized by the projection onto $S$, the radius $\dist(\cdot,S)$ and an angle $\phi$.

Modifying $T$ close to $S$ if necessary, we can furthermore assume that on each (small) plane disk perpendicular to $S$, the restriction of $T$ to this disk is given by a straight line segment. 
To see that this modification is possible, we claim that one can select a radius $r\in (2\eta,3\eta)$ and a slice $T_r$ of $T$ at $\dist(\cdot,S)=r$ such that $3\eta T_r \leq 2\eta\int_{2\eta}^{3\eta} \mathcal{H}^1(T_s) \dx s \leq C\mathcal{H}^2(\{T\cap \{\dist(x,S)\leq 3\eta\}\}$. 
Indeed, this follows from $C^1-$regularity of $T$ up to the boundary or by constructing a competitor for $T$ in the following way:
Around a point $p\in S$ one can choose a tubular neighbourhood, depending on the curvature of $S$, and translate $S$. In case all of the $T_r$ did not satisfy the above condition, this operation decreases the energy of $T$ locally up to lower order terms.
One can then replace $T$ by a $\tilde{T}$ inside the tubular neighbourhood $\{\dist(x,S)\leq r\}$ where $\tilde{T}$ is defined by the straight lines connecting $S$ to $T_r$ on each disk perpendicular to $S$ with asymptotically negligible energy cost.

In order to define the profile on disks perpendicular to $S$, we introduce a orthonormal $C^{1,1}-$frame along $S$.
By Lemma~\ref{lem:reg_T_-_S_C1,1} we already know that the tangent vector field $\tau_S$ of $S$ is of this class. 
Ideally, one would like to choose the normal vector of $T$ to be part of the frame, however we only know that $T$ is of class $C^1$ up to the boundary. 
Instead, we will take an arbitrary normal vector field $\nu_S$ to $S$ of class $C^{1,1}$.
The existence of such a vector field can easily be seen via the following construction:
Since $S$ is compact and of bounded curvature, we can find finitely many points $p_i\in S$ such that there exist normal vectors $\nu_i$ to $S$ in $p_i$, where neighbouring vectors $\nu_i,\nu_j$ form an angle of strictly less than $\pi$. 
We can then choose $C^2-$smooth curves on the sphere $\mathbb{S}^2$ connecting all those $\nu_i$, resulting in $\nu_S$.
The third vector for our frame is simply obtained by taking the cross product $\tau_S\times\nu_S$.

Consider $x_0\in S$.
By applying rotations if necessary, we can assume that $\nu_{S}=\ee_1$ and $\tau_S\times\nu_S = \ee_3$.
We then set
\begin{align*}
\Qexn(x) 
\ &\defi \ Q_B(\dist(x,S),\phi(x))\, , 
\end{align*}
where
\begin{align*}
\phi(x) = \begin{cases}
\arccos\left( \nu_{S}\cdot \frac{x-x_0}{\Vert x-x_0\Vert} \right) & \text{if }(\tau_S\times\nu_S)\cdot \frac{x-x_0}{\Vert x-x_0\Vert} \geq 0 \, , \\
2\pi - \arccos\left( \nu_{S}\cdot \frac{x-x_0}{\Vert x-x_0\Vert} \right) & \text{otherwise.}
\end{cases}
\end{align*}

It remains the transition from the set $\{\dist(\cdot,S)=\eta\}$ to the boundary of \eqref{prop:upper_bound:recov_Q:df_region_Sn_Ome}.
Since $\tau_S\times\nu_S$ might not agree with $\nu_T$, the $T$ constructed around $S$ and the $T$ coming from Step 3 does not necessarily line up.
However, we have enough space to smoothly connect both parts inside $A_{r,\eta} = \{\dist(\cdot,S)=r\}\setminus \{\dist(\cdot,S)=\eta\}$ with asymptotically negligible contribution to the energy. 
Indeed, there exists a Lipschitz deformation $\mathfrak{D}:A_{r,\eta}\rightarrow \{\dist(\cdot,S)=\eta\}$ relative to $\{\dist(\cdot,S)=\eta\}$ such that $T_r = T \cap \{\dist(\cdot,S)=r\}$ gets mapped onto $T\cap \{\dist(\cdot,S)=\eta\}$.
We can then extend $\Qex$ from $\{\dist(\cdot,S)=\eta\}$ to all $\{\dist(\cdot,S)\leq r\}$ along this deformation $\mathfrak{D}$
by setting $\Qex(x) = \Qex(\mathfrak{D}(x))$.

Let $\Pi$ be the projection along $\nabla\dist(\cdot,T)$ onto $\{\dist(\cdot,S)=r\}\cup(T\cap \{ \dist(\cdot,S)\geq r\})$.
The function $\Qex$ is already defined on the first set in the union, for the second we simply pose $\Qex(x) = s_*((v(x),0)\otimes(v(x),0)-\frac13\id)$ in order to be compatible with Step 3.
For $x\in S_{3\radP}\setminus ( \{\dist(\cdot,S)\leq r\}\cup(T\cap \{ \dist(\cdot,S)\geq r\}) )$ we then define
\begin{align*}
\Qex(x)
\ := \ 
\Phi_\eta^{+}(\Vert x-\Pi x\Vert,\theta(x),v(x))\, ,
\end{align*}
where $\theta(x)$ is the angle between $\ee_3$ and the director field that we have already constructed in $\Pi x$, i.e.\ $\theta(x)=\arccos(\nn(\phi(\mathfrak{D}(x)))\cdot\ee_3)$ or $\theta(x) = \arccos(v(\Pi x)\cdot\ee_3)$ depending on which set contains $\Pi x$.

It is easy to see that since $f,g$ and $C_0$ are uniformly bounded and the curvature of $S$ is bounded by Lemma~\ref{lem:reg_T_-_S_C1,1}, we get for the integral up to distance $\eta$
\begin{align*}
&\eta\int_{\{\dist(\cdot,S)\leq\eta\}} \frac{1}{2}|\nabla \Qexn|^2 + \frac{1}{\xi^2}f(\Qexn) + \frac{1}{\eta^2}g(\Qexn) + C_0 \dx x \\
\ &\quad\leq \ 
\frac{\eta}{2} \int_{2\xi}^{\eta} \int_{\{\dist(\cdot,S)=r\}} |\nabla (Q\circ\phi)(x)|^2 \dx r 
\ + \ \frac{\eta}{2} \int_{\xi}^{2\xi} \int_{\{\dist(\cdot,S)=r\}} |\nabla (Q_B(r,\phi(x)))|^2 \dx r \\
&\qquad+ \ \ C\frac{\eta}{\xi^2}\mathcal{H}^3(\{ \dist(\cdot,S)\leq 2\xi \})
\ + \ C\frac{\eta}{\eta^2}\mathcal{H}^3(\{ \dist(\cdot,S)\leq \eta \}) \, .
\end{align*}
Estimating the gradient at distance $r\defi\dist(\cdot,S)\in [2\xi,\eta)$ yields
\begin{align*}
\frac12|\nabla (Q\circ\phi)(x)|^2
\ &= \ s_*^2 |\nabla(\nn\circ\phi)(x)|^2 
\ \leq \  s_*^2\left|
\begin{pmatrix}
\tfrac12\cos(\phi/2) \\ 0 \\ -\tfrac12\sin(\phi/2)
\end{pmatrix}
\otimes\nabla\phi(x)\right|^2 
\ \leq \  \frac{s_*^2}{4r^2}\left( 1 + C r \right) + C \, ,
\end{align*}
where we used that the derivative of the polar angle in the disk perpendicular to $S$ is $\tfrac1r$ and that derivatives of the basis vector fields $\tau_S$, $\nu_S$ and $\tau_S\times\nu_S$ of our frame are bounded.
Hence, we get 
\begin{align*}
\frac{\eta}{2} \int_{2\xi}^{\eta} \int_{\{\dist(\cdot,S)=r\}} |\nabla (Q\circ\phi)(x)|^2 \dx r
\ &\leq \ \frac{s_*^2\pi}{2}\, \eta\, \MM(S\restr\Omega)\int_{2\xi}^\eta \frac{1}{r} \dx r + o(1) \\
\ &\leq \ \frac{\pi}{2} s_*^2 \eta|\ln(\xi)|\, \MM(S\restr\Omega) + o(1)\, .
\end{align*}
Similarly, for $r\in [\xi,2\xi)$ we obtain
\begin{align*}
\frac12|\nabla (Q_B(r,\phi(x)))|^2
\ &= \ s_*^2 \left|\left(\frac{r}{\xi}-1\right)\nabla(\nn\circ\phi)(x)\right|^2 
\ + \ s_*^2 \left|\nabla\left(\frac{r}{\xi}-1\right)\right|^2  \\
\ &\leq \  C\left|\left(\frac{r}{\xi}-1\right)
\begin{pmatrix}
\tfrac12\cos(\phi/2) \\ 0 \\ -\tfrac12\sin(\phi/2)
\end{pmatrix}
\otimes\nabla\phi(x)\right|^2 
+ \frac{C}{\xi^2} \, ,
\end{align*}
from which it follows that
\begin{align*}
\frac{\eta}{2} \int_{\xi}^{2\xi} \int_{\{\dist(\cdot,S)=r\}} |\nabla (Q_B(r,\phi(x)))|^2 \dx r
\ &\leq \ C\eta\, .
\end{align*}

For the energy of the remaining part of the domain defined in \eqref{prop:upper_bound:recov_Q:df_region_Sn_Ome} we use Lipschitz continuity of $\Pi,\mathfrak{D}$ to get
\begin{align*}
\eta\int_{S_{3\radP}\setminus\{\dist(\cdot,S)\leq\eta\}} \frac{1}{2}|\nabla \Qexn|^2 + \frac{1}{\xi^2}f(\Qexn) &+ \frac{1}{\eta^2}g(\Qexn) + C_0 \dx x \\ 
\ &\leq \ C\eta \left(\frac{(\radP)^2}{\eta^2} + M^2 e^{-M}\right)\MM(S\restr\Omega) + o(1)
\end{align*}
which vanishes in the limit $\eta\rightarrow 0$.
We obtain
\begin{align*}
\limsup_{\eta,\xi\rightarrow 0} \:\eta\:\int_{S_{3\radP}} &\frac{1}{2}|\nabla \Qexn|^2 + \frac{1}{\xi^2}f(\Qexn) + \frac{1}{\eta^2}g(\Qexn) + C_0 \dx x 
\ \leq \ \frac{\pi}{2}s_*^2 \: \beta \: \MM(S\restr\Omega) \, .
\end{align*}

\begin{figure}
\begin{center}
\includegraphics[scale=1.5]{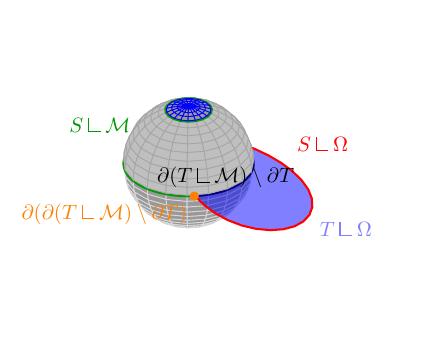}
\end{center}
\caption{Schematic view of the different parts of $T$ and $S$ that are constructed in Step 2 to 6}
\label{fig:upper_bound_constr_global}
\end{figure}

\begin{figure}[H]
\begin{center}
\includegraphics[scale=1.25]{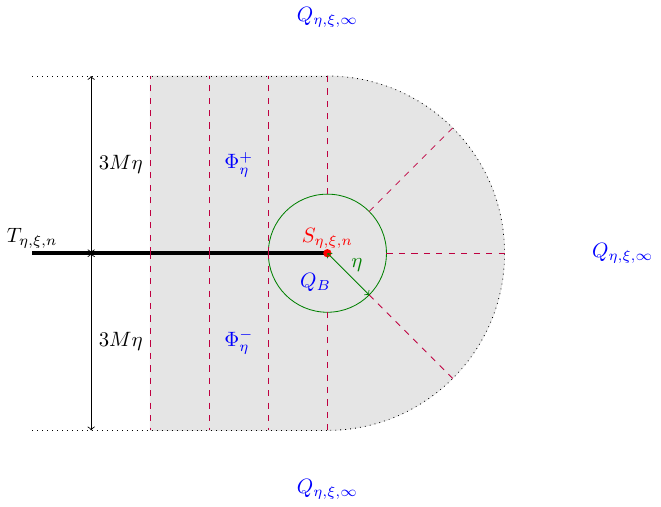}
\end{center}
\caption{Sketch of the construction for $\Qexn$ in Step 4 in the region $S_{3\radP}$ defined in \eqref{prop:upper_bound:recov_Q:df_region_Sn_Ome} (grey shaded area). Dashed lines indicate the direction of the projection $\Pi$.}
\label{fig:upper_bound_constr_S_Ome}
\end{figure}

\textit{Step 5: Construction on $S\restr\M$.}
The domain
\begin{align*}
S_{0,3\radP} \defi
\{ x\in\Omega\sd \dist(x,S)\leq 3\radP, \dist(x,M)\leq 3\radP \text{ and } \dist(x,\partial(\partial (T\restr\M)\setminus\partial T))\geq 3\radP \}
\end{align*}
can essentially be treated in the same manner as in Step 4 or as in \cite[p.1444, Step 3]{ACS2021}. 
Also the boundary of $T$ in $\Omega$ (but close to $\M$) that was created in Step 0 and Step 1 to ensure that $v$ is well-defined, is treated in the same way. 
To give some more details, we can reuse the profile $Q_B$ from the previous step (assuming a $+\frac12-$singularity) for defining $\Qex$ in a ball of radius $\eta$ centered in $x_S$ in the middle of $S_{0,3\radP}$, seen as family of plane sets perpendicular to $S$.
Note that $\Qex$ has already been defined on the boundary on each of those plane sets. 
Thus, a simple two dimensional interpolation of the phase angle along $\nabla\dist(\cdot,x_S)$ as in \cite[eq (56-64)]{ACS2021} shows that the energy contribution is
\begin{align*}
\Eex(\Qex,S_{0,3\radP})
\ &\leq \ (1+C\eta)\frac{\pi}{2}s_*^2|\ln(\xi)|\MM(S\restr\M) + C\eta\frac{(\radP)^2}{\eta^2} + C \beta \radv\, .
\end{align*} 

\textit{Step 6: Construction on $\partial F\setminus S$.}
It remains to fill the ``gaps'' left by the Steps 2 to 5 (see also Fig.~\ref{fig:upper_bound_constr_global}). 
The important part is the transition between the part of $T$ that approaches $\M$ (and which was constructed in Step 2) and the part that stays bounded away, including the region where $S$ detaches from $\M$. 
At distance larger than $3\radP$ from $\M$, we set $\Qex=\Qexinfty$ for all points where we haven't defined $\Qex$ so far. Note that this is compatible with the previous constructions.

Let's consider the set $\Upsilon_{3\radP}\defi\{ x\in\Omega\sd \dist(x,\partial(T\restr\M)\setminus\partial T)\leq 3\radP\text{ and }  \dist(x,\partial(\partial(T\restr\M)\setminus\partial T))\geq 3\radP\}$. 
Considering the slices of $\Upsilon_{3\radP}$ orthogonal to and parametrized by $\partial(T\restr\M)$, we note that Steps 2 to 5 ensure that $\Qex$ takes values in $\N$ whenever meeting the boundary of the slice and $\Qex$ having trivial homotopy class.
For an arbitrary $Q\in N$, we can define $\Qex\defi Q$ on a disk of size $\eta$ in the middle of the slice and 
again by linear interpolation of the phase towards the boundary of the disk.
We thus get a function $\Qex\in H^1(\Upsilon_{0,3\radP},\N)$ respecting the previous constructions and $Q=Q_b$ on $\M$.
Furthermore, the interpolation allows us to estimate $|\nabla Q|^2\leq C((\radP)^{-2}+\eta^{-2})$ and since $g$ is bounded, $f(Q)=0$ (because $Q$ takes values in $\N$) the energy contribution can be estimated
\begin{align*}
\eta\Eex(Q,\Upsilon_{0,3\radP})
\ &\leq \ C \ \eta \ |\Upsilon_{0,3\radP}|\left( \frac{1}{(\radP)^2} + \frac{1}{\eta^2} + 1 \right) \\
\ &\leq \ C \ \MM(\partial(T\restr\M)\setminus \partial T)\left(\eta + \frac{(\radP)^2}{\eta} + \eta(\radP)^2 \right)\, ,
\end{align*}
which vanishes in the limit $\eta,\xi\rightarrow 0$ due to our hypothesis about the finite size of $\partial(T\restr\M)\setminus \partial T$.

It remains the region where $S$ detaches from $\M$ or in other words $\Upsilon_{1,3\radP}\defi\{ x\in\Omega\sd  \dist(x,\partial(\partial(T\restr\M)\setminus S))\leq 3\radP\}$. 
We can also use interpolation to construct $\Qex$ and estimate its energy but we need to be a bit more careful since this time $f(\Qex)$ cannot be chosen to be zero. This is due to the isotropic core of our construction around $S$.
So we connect the ``core'' parts from Step 4 and 5 where we defined $S$ in $\Omega$ and close to $\M$ via a tube in which $\Qex$ is defined via the profile $Q_B$ which has been used in both steps. 
Around this tube, we can again apply the previous idea of linear interpolation of the phase, this time on slices perpendicular to the tube. 
We end up with
\begin{align*}
\Eex(\Qex,\Upsilon_{1,3\radP})
\ &\leq \ C\radP \MM(\partial(\partial(T\restr\M)\setminus \partial T))\, ,
\end{align*}
which vanishes in the limit $\eta\rightarrow 0$ in view of the bound $\MM(\partial(\partial(T\restr\M)\setminus \partial T))\leq C_n$.

\textit{Step 7: Conclusion}
Combining Step 1 to Step 6, we obtain a function $\Qex$ which respects the boundary conditions and satisfies the energy estimate
\begin{align*}
\limsup_{\eta,\xi\rightarrow 0} \eta\: \Eex(\Qex)
\ &\leq \ (1 + C M e^{-M}) \: \E_0(T,S) + C(1+ \beta)\radv \, .
\end{align*}
Since $M$ can be chosen arbitrarily large and $\radv$ arbitrarily small, we can construct a diagonal sequence  and obtain the claim. 

\ds{
It remains to show how to proceed if the assumption that the set $\supp(T)\cap\{x\in\M_{3\radP}\sd \overline{\nu}(x) = \pm\ee_3\}$ is finite does not hold. 
In this case, there is an additional approximation step that needs to be carried out as we detail in the following.

Using the area formula for $\overline{\nu}$ \cite[Thm. 2.91]{Ambrosio2000}, it holds that
\begin{align*}
\int_{\mathbb{S}^2} \mathcal{H}^0(\supp(T)\cap \overline{\nu}^{-1}(y)) \dx \mathcal{H}^2(y)
\ &\leq \ 
C\int_{\supp(T)\cap\{\dist(\cdot,\M)\leq \radM/2\}} |\nabla \overline{\nu}(x)|^2 \dx x \, ,
\end{align*}
which is finite by regularity of $\M$.
Therefore, for all $\epsilon>0$ there exists a unit vector $\ee_3^\epsilon\in\mathbb{S}^2$ such that $|\ee_3^\epsilon - \ee_3|<\epsilon$ and $\mathcal{H}^0(\supp(T)\cap \overline{\nu}^{-1}(\ee_3^\epsilon))<\infty$.
Write $\supp(T)\cap \overline{\nu}^{-1}(\ee_3^\epsilon)=\{x_1,...,x_{N}\}$ .
We can furthermore choose $\ee_3^\epsilon$ such that all of the points $x_i\notin\supp(S)$.
We can then proceed as in Step 0.
In the remaining steps of the proof, we also need to adapt the ``optimal profile''.
Replacing the function $1-n_3^2$ by $1-(\nn\cdot\ee_3^\epsilon)^2$ in Lemma~\ref{lem:radial_turning}, we obtain a new function $\nnopt^\epsilon$ that we use to construct $\nn^\pm_\eta$ in Step 1 by posing
\begin{align*}
\nn^\pm 
\ = \ 
\sqrt{1-(\nnopt^\epsilon)^2} v
\pm \nnopt^\epsilon \ee_3^\epsilon 
\, ,
\end{align*}
the function $v$ being constructed as before but with $\Pi_{(\ee_3^\epsilon)^\perp}(\nu)$ instead of $\nu'$.

By exponential decay of the optimal profile $\nnopt^\epsilon$, the interpolation between $\nn^{\epsilon,\pm}(M,\theta)$ and $\pm \ee_3$ taking place in $\Phi^{\epsilon,\pm}_\eta(t,\theta,v)$ for $t\in (\radP,2\radP]$ is well defined.

Noting that 
\begin{align*}
|1-(\nn\cdot\ee_3^\epsilon)^2 - (1-(\nn\cdot\ee_3)^2)|
\ &\leq \
2\epsilon + \epsilon^2 
\, ,
\end{align*}
we deduce that the additional error introduced by this change is estimated by $C\tfrac{\epsilon}{\eta}(3\radP)(\mathcal{H}^2(\M) + \MM(T))$ which is of order $M\epsilon (1+\E_0(T,S))$.
Therefore, we obtain the final bound
\begin{align*}
\limsup_{\eta,\xi\rightarrow 0} \eta\: \Eex(\Qex)
\ &\leq \ (1 + C M e^{-M} + C M \radv) \: \E_0(T,S) + C \beta\radv + C M \radv
\end{align*}
and passing to the limit first in $\radv\rightarrow 0$ and then in $M\rightarrow \infty$ yields the result.

} 
\end{proof}
} 

\section{Regularity and optimality conditions for the limit problem}
\label{sec:limit_problem}

Let us first state an improved regularity results for minimizers of the energy $\E_0$:

\begin{proposition}\label{prop:high_reg_T}
Let $T$ be a minimizer of \eqref{def:lim_energy} and $S=\partial T-\Gamma$. 
Then each component of $T\restr\Omega$ is an embedded manifold-with-boundary of class $C^\infty$. 
\end{proposition}

\begin{proof}
The main work has been already carried out in the proof of Proposition~\ref{prop:low_reg_T} for $n=0$.
The higher regularity can be obtained by Schauder theory. For details we refer to Theorem 2.1 in \cite{Morgan1994}.
\end{proof}

Next, we give a characterization of minimizers of the limit energy. Because of the regularity given by Proposition~\ref{prop:high_reg_T}, we can take variations of $T\restr\Omega$ and $S\restr\Omega$ in the classical sense to derive the optimality conditions. Furthermore, we can obtain a version of Young's law \cite{Philippis2014,Taylor1977}

\begin{proposition}\label{prop:optim_cond_min_E0}
Let $T$ be a minimizer of \eqref{def:lim_energy} and $S=\partial T-\Gamma$. 
Then $T\restr\Omega$ has zero mean curvature and $S\restr\Omega$ is of constant curvature $\frac{8}{\pi}\frac{c_*}{s_*}\beta^{-1}$.
Furthermore, Young's law holds
\begin{align*}
\nu_{\partial (T\cap\Omega)}\cdot \nu_{\partial F^{+}} 
\ = \ \nu_\M\cdot\ee_3 \quad\text{on }\partial (T\restr\Omega)\setminus S\, ,
\end{align*}
i.e.\ $T$ meets $\M$ in an angle $\theta=\arccos(\nu_\M\cdot\ee_3)$.
\end{proposition}

\begin{proof}
The first claim is a well known fact since the variation of $\MM(T\restr\Omega)$ along a smooth vector field $\Xi$ in $\Omega$ yields \cite[Proposition 2.1.3]{Lopez2013}
\begin{align}\label{prop:optim_cond_min_E0:var_M(T)}
(\MM(T\restr\Omega))'(\Xi)
\ &= \ \int_{T\cap\Omega} H_T (\Xi\cdot\nu_T) \dx x
\ + \ \int_{\partial (T\cap\Omega)} (\Xi\cdot \nu_{\partial T}) \dx x\, ,
\end{align}
where $H_T$ is the mean curvature of $T$, $\nu_T$ is a normal vector of $T$ and $\nu_{\partial T}$ is the inward normal vector of $\partial (T\restr\Omega)$ in the tangent space of $T$.
With the same argument and since $\partial S = 0$, we get that
\begin{align}\label{prop:optim_cond_min_E0:var_M(S)}
(\MM(S))'(\Xi)
\ &= \ \int_S K_S (\Xi\cdot\nu_S) \dx x \, ,
\end{align}
where $K_S$ is the curvature of $S$ and $\nu_S$ is the normal vector of $S$ in $\mathbb{R}^3$, such that the plane for the circle of maximal curvature is spanned by $\nu_S$ and a tangent vector to $S$.
This yields for the boundary that
\begin{align*}
0 \ = \ 
\int_S \Xi\cdot\left( 4s_*c_* \nu_{\partial T} + \frac{\pi}{2}s_*^2\beta K_S \nu_S \right) \dx x \, ,
\end{align*}
from which we deduce $\nu_{\partial T}=\pm\nu_S$ and $K_S = \pm \frac{8}{\pi}\frac{c_*}{s_*}\beta^{-1}$. In particular, the circle of maximal curvature for $S$ lies in the plane spanned by the tangent space of $T$. 
Finally, taking variations on $\M$ we get
\begin{align*}
\left( \int_{F^{\pm}} 1\mp\cos(\theta)\dx\omega \right)'(\Xi)
\ &= \ \int_{\partial F^\pm} \left(1\mp\cos(\theta)\right)(\Xi\cdot \nu_{\partial F^\pm}) \dx\omega\, .
\end{align*}
Since $\nu_{\partial F^{-}} = -\nu_{\partial F^{-}}$, we hence get
\begin{align}\label{prop:optim_cond_min_E0:var_F}
\left( \int_{F^{+}} 1-\cos(\theta)\dx\omega + \int_{F^{-}} 1+\cos(\theta)\dx\omega \right)'(\Xi)
\ &= \ -\int_{\partial F^{+}} 2\cos(\theta)(\Xi\cdot \nu_{\partial F^{+}}) \dx\omega\, .
\end{align}
As in the proof of Theorem 19.8 in \cite{Maggi2009}, \eqref{prop:optim_cond_min_E0:var_M(T)} and \eqref{prop:optim_cond_min_E0:var_F} combine to
\begin{align*}
0 \ &= \
\int_{\partial F^{+}} \Xi\cdot\left( 4s_*c_*\nu_T|_\M - 4s_*c_*\cos(\theta)\nu_{\partial F^{+}} \right)\dx \, .
\end{align*}
If we take a variation with $\Xi\cdot\nu_\M=0$ and write 
\begin{align*}
\Xi\cdot\nu_T|_\M 
\ &= \ \Xi\cdot((\nu_{\partial T}\cdot\nu_{\partial F^{+}})\nu_{\partial F^{+}}) 
\ + \ \Xi\cdot((\nu_T|_\M\cdot \tau)\tau)
\end{align*}
where $\tau$ is a unit tangent vector to $\M$, perpendicular to $\nu_{\partial F^{+}}$, we get 
\begin{align*}
\Xi\cdot((\nu_{\partial T}\cdot \tau)\tau) \ = \ 0 \qquad\text{and}\qquad \nu_{\partial T}\cdot\nu_{\partial F^{+}} \ = \ \cos(\theta)\, .
\end{align*}
The first equality is automatically true since $\nu_{\partial T} \cdot\tau = 0$ ($\nu_{\partial T}$ can only have a component in direction $\nu_{\partial F^{+}}$ and one in direction $\nu_\M$) and the second one implies that
\begin{align*}
\nu_{\partial T}\cdot\nu_{\partial F^{+}}  \ = \ \nu_\M\cdot\ee_3\, .
\end{align*}

\end{proof}


\appendix
\section{The complex $\T$}
\label{sec:app_complex_T}

In this section, we collect and prove all results in connection to the structure of $\T$ as defined in Section~\ref{subsec:constr_T_Q_uniax}.
Recall that
\begin{align*}
\T \defi \{ Q\in\Sym\sd s>0\, , 0\leq r<1\, , n_3 = 0 \}\, .
\end{align*}

Our first result is a characterization of $\T$ that provides us with a more accessible parametrization. 

\begin{proposition}\label{prop:app_calT_param}
Every matrix $Q\in \T$ can be written as
\begin{align*}
Q = \lambda(\nn\otimes\nn - R_\nn^\top M R_\nn)\, ,
\end{align*} 
where $\lambda>0$, $\nn=(n_1,n_2,0)\in\mathbb{S}^2$, $R_\nn$ is the rotation around $\nn\wedge\ee_3$, such that $R_\nn \nn = \ee_3$ and 
\begin{align*}
M 
\ = \ 
\begin{pmatrix}
\multicolumn{2}{c}{M'} & \begin{array}{c} 0 \\ 0 \end{array} \\
0 & 0 & 0
\end{pmatrix}
\end{align*}
with $M'\in\mathbb{R}^{2\times 2}$ symmetric, $\tr(M')=1$ and $\langle M'v,v\rangle>-1$ for all $v\in\mathbb{S}^1$. The matrix $Q$ is oblate uniaxial if and only if $M'=\frac12\id$.
\end{proposition}

\begin{proof}
A matrix $Q$ of the above form $Q = \lambda(\nn\otimes\nn - R_\nn^\top M R_\nn)$ has $\nn$ as an eigenvector to the eigenvalue $\lambda$ and $\nn_3=0$ by definition. Furthermore, since $\min_{v\in\mathbb{S}^1}\langle M'v,v\rangle>-1$ the eigenvalue $\lambda$ is strictly bigger than the other eigenvalues, thus $r<1$ and $Q\in\T$. 
Conversely, we can write every $Q\in\Sym$ as
\begin{align*}
Q \ = \ \lambda_1 \nn\otimes\nn + \lambda_2 \mm\otimes\mm + \lambda_3\pp\otimes\pp\, , 
\end{align*} with $\lambda_1\geq \lambda_2\geq \lambda_3$ and $\nn,\mm,\pp\in\mathbb{S}^2$ pairwise orthogonal eigenvectors of $Q$ to $\lambda_1,\lambda_2,\lambda_3$. 
By definition of $\T$, $n_3=0$ as required for our parametrization and clearly we can identify $\lambda=\lambda_1$.
Setting $M = -R_\nn(\frac{\lambda_2}{\lambda_1}\mm\otimes\mm + \frac{\lambda_3}{\lambda_1}\pp\otimes\pp)R_\nn^\top$, it is obvious that $M$ is of the above form and that $Q\in\T$ can be written as claimed. 

If $M'=\frac12\id$ then
\begin{align*}
Q \ = \ \lambda(\nn\otimes\nn - R_\nn^\top M R_\nn) \ = \ \frac{3}{2}\lambda (\nn\otimes\nn - \frac{1}{3}\id)\, , 
\end{align*} i.e. $Q$ is oblate uniaxial. 
The reverse implication follows similarly, since the matrices $R_\nn^\top,R_\nn$ are invertible.

\end{proof}

\begin{remark}\label{rem:rot_matrix_given_axis_angle}
Given a vector $u\in\mathbb{R}^3$ as axis of rotation and an angle $\theta$, then this rotation is described by the matrix $R$ with
$$ R = \begin{pmatrix}
\cos\theta + u_1^2(1-\cos\theta) & u_1 u_2 (1-\cos\theta) - u_3\sin\theta & u_1 u_3 (1-\cos\theta) + u_2\sin\theta \\
u_1 u_2 (1-\cos\theta) + u_3\sin\theta & \cos\theta + u_2^2(1-\cos\theta) & u_2 u_3 (1-\cos\theta) - u_1\sin\theta \\
u_1 u_3 (1-\cos\theta) - u_2\sin\theta & u_2 u_3 (1-\cos\theta) + u_1\sin\theta & \cos\theta + u_3^2(1-\cos\theta)
\end{pmatrix}\, . $$
\end{remark}

\begin{corollary}\label{cor:calT_manifold_bdry}
$\T$ is a four dimensional smooth complex and $\partial\T = \C$.
\end{corollary}

\begin{proof}
From the characterization in Proposition~\ref{prop:app_calT_param}, it is clear that one can use the map $Q\mapsto (\lambda,\nn,m_{11},m_{12})$  to make $\T$ a four dimensional manifold with a conical singularity in $Q=0$. In particular, $\T$ is a smooth complex.

Proposition~\ref{prop:app_calT_param} furthermore implies that the boundary of $\T$ consists of matrices of the form $\lambda=0$ (from which follows directly $Q=0$) or $M'$ has the eigenvalue $-1$ (which corresponds to $r=1$). 
In particular, the matrices with $r=0$  are not included in $\partial\T$ as one may think from the definition in \eqref{def:cal_T}. 
This implies the inclusion $\partial\T\subset\C$. 
For the inverse inclusion, take $Q\in\C$ with orthogonal eigenvectors $\mm,\pp\in\mathbb{S}^2$ associated to the largest eigenvalue $\lambda_1=\lambda_2$. 
So in fact we have a two dimensional subspace of eigenvectors to this eigenvalue spanned by $\mm$ and $\pp$. 
Since the hyperplane defined through $\{n_3=0\}$ is of codimension one, there exists a unit vector $\nn\in \{n_3=0\}\cap \mathrm{span}\{\mm,\pp\}$ which we were looking for. 
The unit eigenvector orthogonal to $\nn$ in the plane $\mathrm{span}\{\mm,\pp\}$ requires $M'$ to have the eigenvalue $-1$ or in other words $\min_{v\in\mathbb{S}^1}\langle M'v,v\rangle=-1$, so that $Q\in\partial\T$.
\end{proof}

\begin{lemma}\label{lem:app_calT_normal_tangent}
Let $Q\in\T\cap\N$. Then, the normal vector $N_Q$ on $\T$ at $Q$ is given by
\begin{align*}
N_Q = \frac{3}{2}\lambda \begin{pmatrix}
0 & 0 & n_1 \\
0 & 0 & n_2 \\
n_1 & n_2 & 0 
\end{pmatrix}\, ,
\end{align*} where $\nn=(n_1,n_2,0)\in\mathbb{S}^2$ is the eigenvector associated to the largest eigenvalue $\lambda_1$.
\end{lemma}

\begin{proof}
We are going to prove a slightly more general result by first considering $Q\in \T$ and calculating the tangent vectors to $\T$ in $Q$. We use the representation from Proposition~\ref{prop:app_calT_param} and vary $\lambda,\nn,m_{11},m_{12}$ one after another.
\begin{itemize}
\item First, we can easily take the derivative with respect to $\lambda$ and obtain $\TT_1 = (\nn\otimes\nn-  R_\nn^\top M R_\nn)$.
\item Second, we vary the parameter $\nn$. So, let's consider $\nn=(n_1,n_2,0)\in\mathbb{S}^2$. Without loss of generality we assume that $n_2\neq 0$ and write $\nn(t) = (n_1+t,n_2-\frac{n_1}{n_2}t)$. Then $|\nn(t)|^2=1+O(t^2)$ and
\begin{align*}
\nn(t)\otimes \nn(t) 
\ &= \ \nn\otimes \nn + t D_{\nn\otimes\nn} + O(t^2)\, , \qquad D_{\nn\otimes\nn} = \begin{pmatrix}
2n_1 & n_2-\frac{n_1^2}{n_2} & 0 \\
n_2-\frac{n_1^2}{n_2} & -2n_1 & 0 \\
0 & 0 & 0
\end{pmatrix}\, .
\end{align*} 
The derivative of the second term $R_{\nn(t)}^\top M R_{\nn(t)}$ can be calculated using Remark~\ref{rem:rot_matrix_given_axis_angle} with the axis $\nn^\perp(t)\defi \nn(t)\wedge \ee_3$. Since $\nn(t)\perp\ee_3$ we can write
\begin{align*}
R_{\nn(t)} 
\ &= \ R_{\nn} + t D_{R_\nn} + O(t^2)\, , \qquad D_{R_\nn} = \frac{1}{n_2} \begin{pmatrix}
-2n_1 n_2 & -n_2^2+n_1^2 & -n_2 \\
-n_2^2+n_1^2 & 2n_1 n_2 & n_1 \\
n_2 & -n_1 & 0
\end{pmatrix}\, .
\end{align*}
The second tangent vector $\TT_2$ is thus given by $\TT_2 = \lambda(D_{\nn\otimes\nn} - D_{R_\nn}^\top M R_\nn - R_\nn^\top M D_{R_\nn})$.
\item Third, we can take the derivative with respect to $m_{11}$. This is straightforward and we obtain
\begin{align*}
\TT_3 = \lambda R_\nn^\top \begin{pmatrix}
1 & 0 &  \\
0 & -1 & \\
 & & 0
\end{pmatrix} R_\nn\, .
\end{align*}
\item Last, varying $m_{12}$ we easily calculate
\begin{align*}
\TT_4 = \lambda R_\nn^\top \begin{pmatrix}
0 & 1 &  \\
1 & 0 & \\
 & & 0
\end{pmatrix} R_\nn\, .
\end{align*}
\end{itemize}
Before proceeding, we want to calculate a fifth vector by varying $\nn_3$. As it will turn out later, this is indeed the normal vector.
\begin{itemize}
\item Writing once again $\nn=(n_1,n_2,0)$ and defining $\nn(t)\defi (n_1\sqrt{1-t^2},n_2\sqrt{1-t^2},t)$ we can express
\begin{align*}
\nn(t)\otimes \nn(t) 
\ &= \ \nn\otimes \nn + t (\nn\otimes\ee_3 + \ee_3\otimes\nn) + O(t^2)\, .
\end{align*}
As for the second tangent vector, we use Remark~\ref{rem:rot_matrix_given_axis_angle} and the rotation around $\nn^\perp(t) = \nn(t)\wedge\ee_3$. Unlike previously, $\nn(t)$ is no longer orthogonal to $\ee_3$ for $t\neq 0$, namely $\theta=\arccos(\langle\nn(t),\ee_3\rangle)=t$. Substituting this our expression of the rotation matrix we get
\begin{align*}
R_{\nn(t)} \ = \ R_{\nn} + t D_{3} + O(t^2)\, ,\qquad D_{3} = \begin{pmatrix}
1-n_2^2 & n_1 n_2 & 0 \\
n_1 n_2 & 1- n_1^2 & 0 \\
0 & 0 & 1
\end{pmatrix}\, .
\end{align*}
Adding the two partial results, we get
\begin{align*}
N \defi \lambda (\nn\otimes\ee_3 + \ee_3\otimes\nn - D_{3}^\top M R_\nn - R_\nn^\top M D_{3})\, .
\end{align*}
\end{itemize} 
It remains to show that $\{\TT_1,\TT_2,\TT_3,\TT_4, N\}$ are pairwise orthogonal if $Q$ is oblate uniaxial. Indeed, then it follows that $N$ is a normal vector, since it is orthogonal to $T_Q\T$.

It is easy to see that since the trace is invariant by change of basis and since $R_\nn^\top=R_\nn^{-1}$
\begin{align*}
\langle \TT_3,\TT_4\rangle 
\ &= \ \lambda^2\tr\left( \begin{pmatrix}
1 & 0 \\ 0 & -1
\end{pmatrix}\begin{pmatrix}
0 & 1 \\ 1 & 0
\end{pmatrix} \right) 
\ = \ \lambda^2 \tr\left( \begin{pmatrix}
0 & 1 \\ -1 & 0
\end{pmatrix} \right)
\ = \ 0\, .
\end{align*}
Noting that $\nn\otimes\nn R_\nn^\top M R_\nn=0$ for $M\in\Sym$ with $m_{ij}=0$ if $i=3$ or $j=3$, we get
\begin{align*}
\langle \TT_1,\TT_3\rangle 
\ &= \ \lambda \tr\Big((\nn\otimes\nn - R_\nn^\top M R_\nn)(R_\nn^\top \begin{pmatrix}
1 & 0 & \\
0 & -1 & \\
 & & 0
\end{pmatrix} R_\nn ) \Big) \\
\ &= \ \lambda \tr(M\begin{pmatrix}
1 & 0 & \\
0 & -1 & \\
 & & 0
\end{pmatrix})
\ = \ \lambda \tr(\begin{pmatrix}
m_{11} & -m_{12} & \\
m_{12} & -m_{22} & \\
 & & 0
\end{pmatrix})
\ = \ \lambda(2 m_{11}-1) \, .
\end{align*}
With the same argument we also find
\begin{align*}
\langle \TT_1,\TT_4\rangle 
\ &= \ \lambda \tr\Big((\nn\otimes\nn - R_\nn^\top M R_\nn)(R_\nn^\top \begin{pmatrix}
0 & 1 & \\
1 & 0 & \\
 & & 0
\end{pmatrix} R_\nn)\Big) \\
\ &= \ \lambda \tr(M \begin{pmatrix}
0 & 1 & \\
1 & 0 & \\
 & & 0
\end{pmatrix})
\ = \ \lambda \tr(\begin{pmatrix}
m_{12} & m_{11} & \\
m_{22} & m_{12} & \\
 & & 0
\end{pmatrix})
\ = \ 2\lambda m_{12} \, .
\end{align*}
Furthermore, we claim that
\begin{align*}
\langle \TT_1,\TT_2\rangle 
\ &= \ \lambda\: \tr((\nn\otimes\nn - R_\nn^\top M R_\nn)(D_{\nn\otimes\nn} - D_{R_\nn}^\top M R_\nn - R_\nn^\top M D_{R_\nn}))
\ = \ 0\, .
\end{align*}
Indeed, one can check that
\begin{align*}
\tr(\nn\otimes\nn D_{\nn\otimes\nn}) 
\ &= \ 
0 
\ = \ \tr(\nn\otimes\nn D_{R_\nn}^\top M R_\nn)
\, , \\
\tr(\nn\otimes\nn R_\nn^\top M D_{R_\nn}) 
\ &= \ 
0 
\ = \ 
\tr(R_\nn^\top M R_\nn D_{\nn\otimes\nn})
\, , \\
\tr(R_\nn^\top M R_\nn D_{R_\nn}^\top M R_\nn) 
\ &= \ 
0 
\ = \ 
\tr(R_\nn^\top M R_\nn R_\nn^\top M D_{R_\nn})
\, .
\end{align*}
This implies that
\begin{align*}
\langle N,\TT_3\rangle 
\ &= \ \lambda^2 \tr\Big((\nn\otimes\ee_3 + \ee_3\otimes\nn - D_{3}^\top M R_\nn - R_\nn^\top M D_{3})(R_\nn^\top \begin{pmatrix}
1 & 0 & \\
0 & -1 & \\
 & & 0
\end{pmatrix} R_\nn)\Big) 
\ = \ 0\, ,
\end{align*}
since again the traces of all cross terms vanish. Similarly,
\begin{align*}
\langle N,\TT_4\rangle 
\ &= \ 0 \, .
\end{align*}
Next, we have the equality
\begin{align*}
\langle \TT_2,\TT_3\rangle 
\ &= \ -4 \lambda^2 \frac{m_{12}}{n_2} \, .
\end{align*}
This follows since $\tr(D_{\nn\otimes \nn} \TT_3) = 0$ and $\tr(D_{3} M R_\nn \TT_3) = \frac{2m_{12}}{n_2}$.
The latter fact is evident if one calculates $M \begin{pmatrix}
1 & 0 & \\
0 & -1 & \\
 & & 0
\end{pmatrix}
\ = \ \begin{pmatrix}
m_{11} & -m_{12} & \\
m_{12} & -m_{22} & \\
 & & 0
\end{pmatrix}$
and
$R_\nn D_3^\top = \begin{pmatrix}
0 & -1/n_2 & 1 \\
1/n_2 & 0 & -n_1/n_2 \\
-1 & n_1/n_2 & 0
\end{pmatrix}$.
This also permits us to derive
\begin{align*}
\langle \TT_2,\TT_4\rangle 
\ &= \ 2\lambda^2\frac{2m_{11}-1}{n_2} \, .
\end{align*}
Again, we simply calculate the traces of all cross terms. For example
\begin{align*}
\tr(\nn\otimes\ee_3 D_{\nn\otimes\nn}) \ &= \ 0 \, , \\
\tr(\nn\otimes\ee_3 R_{\nn}^\top M D_{R_{\nn}}) \ &= \ 0 \, , \\
\tr(\nn\otimes\ee_3 D_{R_\nn}^\top MR_\nn) \ &= \ \frac{m_{12}}{n_2}(n_1^2-n_2^2)-n_1(2m_{11}-1) \, , \\
\tr(D_{R_\nn}^\top M R_\nn D_{\nn\otimes\nn}) \ &= \ 2\frac{m_{11}n_1}{n_2} + \frac{1}{n_2^2}(n_1^2(2m_{11}-1)+m_{11}) \, , \\
\tr(D_{R_\nn}^\top M R_\nn R_{\nn}^\top M D_{R_{\nn}}) \ &= \ -2\frac{n_1 m_{12}}{n_2} + \frac{1}{n_2^2}\left( 3(m_{11}^2 + m_{12}^2) - (1 + n_1^2)(2m_{11}-1) \right) \, , \\
\tr(D_{R_\nn}^\top M R_\nn D_{R_\nn}^\top MR_\nn) \ &= \  2\frac{m_{11}m_{22}+m_{12}^2}{n_2^2}\, , \\
\end{align*}
We end up with
\begin{align*}
\langle N,\TT_2\rangle 
\ &= \ \frac{6\lambda^2 m_{12} (n_1^2 - n_2^2)}{n_2} - 6\lambda^2 n_1(2 m_{11} -1) \, .
\end{align*}
Another straightforward calculation shows that
\begin{align*}
\langle N,\TT_1\rangle 
\ &= \ \lambda 2 n_1 m_{12}(n_1^5 m_{12} - 2 n_{1}^4 n_{2} m_{11} - 2 n_{1}^3 m_{12} - 2n_1^2 n_2^3 m_{11} + 3 n_1^2 n_2 m_{11} \\
&-2n_1^2 n_2- n_1 n_2^4m_{12} + n_1 m_{12} + n_2^3 m_{11} - n_2 m_{11} - 2 n_2^3 + 2n_2 ) \, .
\end{align*}
After these calculations, it is apparent that for prolate uniaxial $Q\in\Sym$ (and in particular $Q\in\N$), i.e.\ $M'=\frac{1}{2}\id$ all inner products vanish. 
In order to form a basis, we must prove that the vectors themselves never vanish. 
We find
\begin{align*}
\Vert \TT_1\Vert^2 \ &= \ 2(m_{11}^2 - m_{11} + m_{12} + 1) \, , \\
\Vert \TT_2\Vert^2 \ &= \ \frac{2}{n_2^2}(6n_1^2(1-2m_{11}) - 6m_{12}n_1 n_2 + 5m_{11}^2 - 2m_{11} + 5m_{12} + 2) \, , \\
\Vert \TT_3\Vert^2 \ &= \ 2\lambda^2 \, , \\
\Vert \TT_4\Vert^2 \ &= \ 2\lambda^2 \, , \\
\Vert N \Vert^2 \ &= \ \lambda^2(12m_{11}n_1^2 -6n_1^2 + 12 m_{12}n_1n_2 + 2m_{11}^2 - 8m_{11} + 2m_{12}^2 + 8) \, , \\
\end{align*}
and thus for $M'=\frac{1}{2}\id$ it holds that $\Vert \TT_1\Vert^2 = \frac{6}{4}$, $\Vert \TT_2\Vert^2=\frac{9}{2}\lambda^2 n_2^{-2}$ and $\Vert N\Vert^2 = \frac{9}{2}\lambda^2$.

This concludes the proof that $\{\TT_1,\TT_2,\TT_3,\TT_4\}$ form indeed a basis of $T_Q\T$, and since $N$ is orthogonal to $T_Q\T$, the  result follows.
\end{proof}

\begin{proposition} \label{prop:app_vol_Ba_in_T_N}
There exists $C,\alpha_0>0$ such that for all $\alpha\in (0,\alpha_0)$ and  $Q\in \N$ it holds 
$$ \mathcal{H}^4(B_\alpha(Q)\cap\T) \ \leq \ C \alpha^4\, . $$
\end{proposition}

\begin{proof}
As seen before, $\T$ has  the structure of a smooth manifold around $\N$. 
By invariance of $\N$ under rotations, it is enough to show that the claim holds around one $Q\in\N$.
The Ricci curvature $\kappa$ of $\N$ is bounded so that we can choose $\alpha_0>0$ small enough such that $B_\alpha(Q)\cap\T$ is contained in the geodesic ball in $\T$ of size $2\alpha$ around $Q$ for all $\alpha\in (0,\alpha_0)$.
Furthermore, if needed, we can choose $\alpha_0>0$ even smaller such that $1-\frac{\kappa}{36\alpha_0^2}\leq 2$.
Theorem 3.1 in \cite{Gray1974} then implies that
\begin{align*}
\mathcal{H}^4(B_\alpha(Q)\cap\T)
\ &\leq \ \mathrm{vol}_\T(B_{2\alpha}(Q))
\ \leq \ 16\pi^2\alpha^4\, .
\end{align*} 
\end{proof}

\paragraph{Acknowledgment.} The authors thank Antoine Lemenant and Guy David for the useful discussions and comments on the regularity for almost minimizers of length and surface.
\ds{The authors also thank Vincent Millot, Lia Bronsard and the anonymous referee for comments and suggestions, which helped us to fix some issues and greatly improve the manuscript.}



\addcontentsline{toc}{section}{References}
\bibliography{Notes_LC}{}

\end{document}